\documentclass[12pt]{amsart}  

\usepackage{graphicx}
\usepackage{here} 
\usepackage{array} 
\usepackage[b]{esvect}
\usepackage{dsfont}
\usepackage{bm}
\usepackage{bbm}

\usepackage{vmargin}
\usepackage{amssymb}
\usepackage{mathrsfs}
\usepackage[all]{xy}
\usepackage[usenames,dvipsnames]{color}
\usepackage{soul}
\RequirePackage[colorlinks,linkcolor=blue,citecolor=LimeGreen,urlcolor=red]{hyperref} 
\usepackage{amsmath}

\newtheorem{theorem}{Theorem}[section]

\newtheorem{cor}[theorem]{Corollary}

\newtheorem{lemma}[theorem]{Lemma}

\newtheorem{prop}[theorem]{Proposition}
\newtheorem{remark}[theorem]{Remark}

\numberwithin{equation}{section}

\newcommand{\R}{\mathbb{R}}

\newcommand{\N}{\mathbb{N}}

\newcommand{\T}{\mathbb{T}}

\newcommand{\func}[3]{#1 : #2 \longrightarrow #3}

\newcommand{\disp}{\displaystyle}
\newcommand{\abs}[1]{\left|#1\right|}
\newcommand{\eps}{\varepsilon}
\newcommand{\norm}[1]{\left\|#1\right\|}
\newcommand{\normb}[1]{\big\|#1\big\|}

\renewcommand{\leq}{\leqslant}
\renewcommand{\geq}{\geqslant}
\renewcommand{\bar}{\overline}
\renewcommand{\tilde}{\widetilde}
\newcommand{\pa}[1]{\left(#1\right)}
\newcommand{\pab}[1]{\big(#1\big)}
\newcommand{\pabb}[1]{\Big(#1\Big)}
\newcommand{\cro}[1]{\left[#1\right]}
\newcommand{\br}[1]{\left\{#1\right\}}
\newcommand{\scalprod}[1]{\left\langle#1\right\rangle}
\newcommand{\scalprodb}[1]{\big\langle#1\big\rangle}
\newcommand\restr[2]{{
  \left.\kern-\nulldelimiterspace 
  #1 
  \right|_{ #2} 
  }}

\newcommand{\Sf}{\mathbb{S}^2}

\newcommand{\dd}{\mathrm{d}}            

\newcommand{\init}{{\hspace{0.4mm}\mathrm{in}}}
\newcommand{\ms}{\text{{\tiny MS}}}
\newcommand{\bb}{\text{{\tiny B}}}

\newcommand{\LL}{\text{{\tiny \bf L}}}

\newcommand{\KK}{\text{{\tiny \bf K}}}

\newcommand{\NuNu}{\text{{\tiny $\pmb{\nu}$}}}

\newcommand{\QQ}{\text{{\tiny \bf Q}}}

\newcommand{\TT}{\text{{\tiny \bf T}}}
\newcommand{\Ker}{\mathrm{ker\hspace{0.5mm}}}
\newcommand{\Span}{\textrm{Span}}

\newcommand{\ee}{\mathrm{e}}
\newcommand{\One}{\text{{\tiny (1)}}}
\newcommand{\Enn}{\text{{\tiny (N)}}}
\newcommand{\boldF}{\mathbf{F}}
\newcommand{\boldf}{\mathbf{f}}
\newcommand{\boldg}{\mathbf{g}}
\newcommand{\boldh}{\mathbf{h}}
\newcommand{\boldL}{\mathbf{L}}

\newcommand{\boldT}{\mathbf{T}}
\newcommand{\boldQ}{\mathbf{Q}}
\newcommand{\boldmu}{\pmb{\mu}}
\newcommand{\boldc}{\mathbf{c}}
\newcommand{\boldu}{\mathbf{u}}
\newcommand{\bepsM}{\mathbf{M^{\varepsilon}}}
\newcommand{\bmathepsM}{\bm{\mathcal{M}_{\varepsilon}}}
\newcommand{\bmathM}{\bm{\mathcal{M}}}
\newcommand{\bepsL}{\mathbf{L^{\varepsilon}}}
\newcommand{\bepsK}{\mathbf{K^{\varepsilon}}}
\newcommand{\bepsnu}{\pmb{\nu}^{\varepsilon}}

\newcommand{\bepsS}{\mathbf{{S^{\varepsilon}}}}
\newcommand{\bepsT}{\mathbf{T^{\varepsilon}}}

\newcommand{\epsM}{M^{\varepsilon}}
\newcommand{\mathM}{\mathcal{M}}
\newcommand{\epsL}{L^{\varepsilon}}
\newcommand{\epsK}{K^{\varepsilon}}
\newcommand{\epsnu}{\nu^{\varepsilon}}

\newcommand{\hilbertR}{L^2\big(\R^3,\pmb{\mu}^{-\frac{1}{2}}\big)}
\newcommand{\hilbertv}{L_v^2\big(\pmb{\mu}^{-\frac{1}{2}}\big)}
\newcommand{\hilbertTR}{L^2\big(\T^3\times\R^3,\pmb{\mu}^{-\frac{1}{2}}\big)}
\newcommand{\hilbertxv}{L_{x,v}^2\big(\pmb{\mu}^{-\frac{1}{2}}\big)}

\newcommand{\sobolevTR}[1]{H^{#1}\big(\T^3\times\R^3,\pmb{\mu}^{-\frac{1}{2}}\big)}

\newcommand{\sobolevxv}[1]{H_{x,v}^{#1}\big(\pmb{\mu}^{-\frac{1}{2}}\big)}
\newcommand{\sobolevx}[1]{H_x^{#1}\big(\bar{\boldc}^{-\frac{1}{2}}\big)}

\newcommand{\spaceR}{L_v^2\big(\R^3,\langle v\rangle^{\gamma/2}\pmb{\mu}^{-1/2}\big)}
\newcommand{\spacev}{L_v^2\big(\langle v\rangle^{\gamma/2}\pmb{\mu}^{-1/2}\big)}

\newcommand{\spacexv}{L_{x,v}^2\big(\langle v\rangle^{\gamma/2}\pmb{\mu}^{-\frac{1}{2}}\big)}
\newcommand{\spacex}[1]{H_x^{#1}\big(\boldc_{\infty}^{-\frac{1}{2}}\big)}

\newcommand{\fp}{f^{\prime}}
\newcommand{\fs}{f^*}
\newcommand{\fps}{f^{\prime *}}

\newcommand{\Mip}{{M_i^{\varepsilon}}^{\prime}}

\newcommand{\Mjps}{{M_j^{\varepsilon}}^{\prime *}}

\newcommand{\multideriv}{\partial^{\beta}_v\partial^{\alpha}_x}

\def\restriction#1#2{\mathchoice
              {\setbox1\hbox{${\displaystyle #1}_{\scriptstyle #2}$}
              \restrictionaux{#1}{#2}}
              {\setbox1\hbox{${\textstyle #1}_{\scriptstyle #2}$}
              \restrictionaux{#1}{#2}}
              {\setbox1\hbox{${\scriptstyle #1}_{\scriptscriptstyle #2}$}
              \restrictionaux{#1}{#2}}
              {\setbox1\hbox{${\scriptscriptstyle #1}_{\scriptscriptstyle #2}$}
              \restrictionaux{#1}{#2}}}
\def\restrictionaux#1#2{{#1\,\smash{\vrule height .8\ht1 depth .85\dp1}}_{\,#2}}

\makeatletter
\def\namedlabel#1#2{\begingroup
    #2%
    \def\@currentlabel{#2}%
    \phantomsection\label{#1}\endgroup
}
\makeatother

\def\signmarc{\bigskip \begin{center} {
{\sc Marc Briant}\par\vspace{3mm}
Universit\'e de Paris, Universit\'e Paris Descartes\par \vspace{1mm}
Laboratoire MAP5, CNRS UMR 8145 \par \vspace{1mm}
F-75006 Paris, FRANCE \par
\vspace{3mm}
{\sc e-mail:}} \tt{briant.maths@gmail.com} \end{center}}

\def\signandrea{\bigskip \begin{center} {
{\sc Andrea Bondesan} \par\vspace{3mm}
Universit\'e de Paris, Universit\'e Paris Descartes\par \vspace{1mm}
Laboratoire MAP5, CNRS UMR 8145 \par \vspace{1mm}
F-75006 Paris, FRANCE \par

\vspace{3mm}
Sorbonne Universit\'e,\par \vspace{1mm}
Laboratoire Jacques-Louis Lions, CNRS UMR 7598 \par \vspace{1mm}
F-75005 Paris, FRANCE \par

\vspace{3mm}
{\sc e-mail:}} \tt{andrea.bondesan@parisdescartes.fr} \end{center}}

\begin{document} 

\title[Multi-species Boltzmann to Maxwell-Stefan]{Stability of the Maxwell-Stefan system\\in the diffusion asymptotics of the\\Boltzmann multi-species equation}
\author{Andrea Bondesan and Marc Briant}

\begin{abstract}
We investigate the diffusion asymptotics of the Boltzmann equation for gaseous mixtures, in the perturbative regime around a local Maxwellian vector whose fluid quantities solve a flux-incompressible Maxwell-Stefan system. Our framework is the torus and we consider hard-potential collision kernels with angular cutoff. As opposed to existing results about hydrodynamic limits in the mono-species case, the local Maxwellian we study here is not a local equilibrium of the mixture due to cross-interactions. By means of a hypocoercive formalism and introducing a suitable modified Sobolev norm, we build a Cauchy theory which is uniform with respect to the Knudsen number $\eps$. In this way, we shall prove that the Maxwell-Stefan system is stable for the Boltzmann multi-species equation, ensuring a rigorous derivation in the vanishing limit $\eps\to 0$.
\end{abstract}

\maketitle

\vspace*{10mm}



\tableofcontents

\section{Introduction}


The multi-species Boltzmann system is an extension of the standard Boltzmann mono-species equation \cite{Cer1,CerIllPul,Vil2}, adapted to the case where the particles constituting the rarefied gas are of different kinds. More precisely, a gaseous mixture composed of $N\geq 2$ different species of chemically non-reacting monoatomic particles, having atomic masses $(m_i)_{1\leq i\leq N}$ and evolving on the $3$-dimensional torus $\T^3$, can be modelled by means of a distribution function $\boldF=(F_1,\ldots,F_N)$, where $F_i=F_i(t,x,v)$ describes the evolution of the $i$-th species of the mixture and satisfies, for any $1\leq i\leq N$, the equation of Boltzmann type
\begin{equation}\label{multiBE}
\partial_t F_i + v\cdot \nabla_x F_i = Q_i(\boldF,\boldF) \quad \textrm{on } \R_+\times\T^3\times\R^3,
\end{equation}
with a given initial data
$$F_i(0,x,v) = F_i^\init(x,v),\quad x\in\T^3,\ v\in\R^3.$$
We mention in particular that one can derive this type of equations from Newtonian mechanics, at least formally, in the case of a single species gas \cite{Cer1,CerIllPul}. The rigorous derivation of the mono-species Boltzmann equation from Newtonian laws has by now only been proved locally in time (see \cite{Lan,IllPul2,IllPul3} and, more recently, \cite{GalSaiTex,PulSafSim}).  

\smallskip
Throughout the article, $N$-vectors (or vector-valued functions) will be denoted by bold letters, while the corresponding indexed letters will indicate their components. For example, $\mathbf{W}$ represents the vector or vector-valued function $(W_1,\ldots ,W_N)$.

\smallskip
The Boltzmann operator $\boldQ=(Q_1,\ldots, Q_N)$ is given for any $1\leq i\leq N$ by 
$$Q_i(\boldF,\boldF) = \sum\limits_{j=1}^N Q_{ij}(F_i,F_j),$$
where $Q_{ij}$ models the interactions between particles of either the same ($i=j$) or of different~($i\neq j$) species and is local in time and space. We focus on binary and elastic collisions, meaning that if two particles of different species of respective atomic masses $m_i$ and~$m_j$ collide with velocities $v'$ and $v'_*$, then the shape of their post-collisional velocities $v$ and $v_*$ is prescribed by the conservation of momentum and kinetic energy
\begin{equation}\label{elasticcollision}
m_iv + m_jv_*= m_iv' + m_jv'_*,\qquad \frac{1}{2}m_i |v|^2 + \frac{1}{2}m_j |v_*|^2 = \frac{1}{2}m_i |v'|^2 + \frac{1}{2}m_j|v'_*|^2.
\end{equation}
We point out that, unlike the mono-species case where $N=1$ and so $m_i=m_j=m$, we can see here an asymmetry in the role played by $v$ and $v_*$, due to the different masses of the species. This issue will be of primary interest in our work. The bi-species collision operators then read, for any $1\leq i,j\leq N$,
$$Q_{ij}(F_i,F_j)(v) =\int_{\R^3\times \mathbb{S}^{2}}B_{ij}\left(|v - v_*|,\cos\vartheta\right)\left[F_i'F_j^{\prime *} - F_iF_j^*\right]\dd v_*\dd \sigma,$$
where the shorthand notations $F_i'=F_i(v')$, $F_i=F_i(v)$, $F_j^{\prime *}=F_j(v'_*)$ and $F_j^*=F_j(v_*)$ are used with the definitions
$$\left\{ \begin{array}{rl} \displaystyle{v'} & \displaystyle{=\frac{m_iv+m_jv_*}{m_i+m_j} +  \frac{m_j}{m_i+m_j}|v-v_*|\sigma,}
\\[6mm]
\displaystyle{v' _*}&\displaystyle{=\frac{m_iv+m_jv_*}{m_i+m_j} - \frac{m_i}{m_i+m_j}|v-v_*|\sigma,} \end{array}\right.\qquad \cos\vartheta =  \frac{(v-v_*)\cdot \sigma}{\abs{v-v_*}}.$$
In particular, the cross-sections $B_{ij}$ model the physics of the binary collisions between particles. Here we shall focus on cutoff Maxwellian, hard-potential and hard-sphere collision kernels. Namely, let us make the following assumptions on each $B_{ij}$, $i$ and~$j$ being fixed.
\begin{enumerate}
\item[(H1)] It satisfies a symmetry property with respect to the interchange of both species indices~$i$ and $j$
$$ B_{ij}(|v-v_*|,\cos\vartheta)=B_{ji}(|v-v_*|,\cos\vartheta),\quad \forall v,v_*\in\R^3,\ \forall \vartheta\in\R. $$
\item[(H2)] It decomposes into the product of a kinetic part $\Phi_{ij}\geq 0$ and an angular part $b_{ij}\geq 0$, namely
$$ B_{ij}(|v-v_*|,\cos\vartheta)=\Phi_{ij}(|v-v_*|)b_{ij}(\cos\vartheta),\quad \forall v,v_*\in\R^3,\ \forall\ \vartheta\in\R. $$
\item[(H3)] The kinetic part has the form of hard or Maxwellian ($\gamma = 0$) potential, \textit{i.e.}
$$ \Phi_{ij}(|v-v_*|)=C_{ij}^{\Phi}|v-v_*|^{\gamma},\quad C_{ij}^{\Phi}>0,\quad \gamma\in [0,1],\quad \forall v,v_*\in\R^3. $$
\item[(H4)] For the angular part, we consider a strong form of Grad's angular cutoff \cite{Gra1}. We assume that there exists a constant $C >0$ such that
$$ 0 < b_{ij}(\cos\vartheta)\leq C|\sin\vartheta||\cos\vartheta|,\qquad b_{ij}^{\prime}(\cos\vartheta)\leq C,\qquad \vartheta\in [0,\pi]. $$
Furthermore, we assume that
$$\inf_{\sigma_1,\sigma_2\in\Sf}\int_{\Sf}\min\{b_{ii}(\sigma_1\cdot\sigma_3),b_{ii}(\sigma_2\cdot\sigma_3)\}\dd \sigma_3 >0.$$
\end{enumerate}
Note that the above hypotheses on the collision kernels are standard in both the multi-species setting \cite{DauJunMouZam,BriDau} and the mono-species one \cite{BarMou,Mou1}, and are assumed to hold in order to derive suitable spectral properties on the linear operator. Assumption (H1) translates the idea that the collisions are micro-reversible. Assumption (H2) is only made for the sake of clarity (even though it is commonly used in a lot of physical models) and could probably be relaxed at the price of technicalities. Assumption (H3) is proper to collision kernels that come from interaction potentials behaving like power-laws. Finally, the positivity assumption on the integrals appearing in (H4) is satisfied by most physical models and is required to obtain an explicit spectral gap estimate in the mono-species case \cite{BarMou,Mou1}, thus becoming a prerequisite to establish explicit computations of the spectral gap also in the multi-species setting \cite{BriDau}.

\smallskip
The first \textit{a priori} laws one can extract \cite{DesMonSal,BouGreSal2,DauJunMouZam} from $\eqref{multiBE}$ are the conservation, over time $t\geq 0$, of the quantities 
\begin{equation}\label{conservationlaws}
\begin{split}
\ & c_{i,\infty} = \int_{\T^3\times\R^3} F_i(t,x,v)\dd x\dd v,
\\[2mm] \ & \rho_\infty u_{\infty} = \sum\limits_{i=1}^N\int_{\T^3\times\R^3} m_ivF_i(t,x,v)\dd x\dd v,
\\[2mm]\ & \frac{3}{2} \rho_\infty \theta_{\infty} = \sum\limits_{i=1}^N\int_{\T^3\times\R^3} \frac{m_i}{2}\abs{v-u_\infty}^2F_i(t,x,v)\dd x\dd v,
\end{split}
\end{equation}
where $c_{i,\infty}$ stands for the number of particles of species $i$, and we have denoted the total mass of the mixture $\rho_\infty = \sum_{i=1}^N m_i c_{i,\infty}$, its total momentum $\rho_\infty u_\infty$ and its total energy~$\frac{3}{2}\rho_\infty \theta_\infty$. The lack of symmetry mentioned before clearly appears here, since only the total momentum and energy of the gas are preserved, as opposed to a preservation of the momentum and energy of each single species.

The second important feature is that the operator $\boldQ=(Q_1,\dots,Q_N)$ also satisfies a multi-species version of the classical $H$-theorem \cite{DesMonSal}, from which one deduces that the only distribution functions satisfying $\boldQ(\boldF,\boldF)=0$ are given by the local Maxwellian vectors~$\pa{M_{(c_i,u,\theta)}}_{1\leq i \leq N}$, having the specific form
\begin{equation}\label{local equilibrium}
M_{(c_i,u,\theta)}(t,x,v) = c_i(t,x)\pa{\frac{m_i}{2\pi \theta(t,x)}}^{3/2}e^{-m_i\frac{\abs{v-u(t,x)}^2}{2\theta(t,x)}},\quad t\geq 0,\ x\in\T^3,\ v\in\R^3,
\end{equation}
for some functions $\mathbf{c}=(c_i)_{1\leq i\leq N}$, $u$ and $\theta$. These are the fluid quantities associated to $\mathbf{M}$ and satisfy, for any $(t,x)\in\R_+\times\T^3$, the relations
\begin{eqnarray*}
c_i(t,x) &=& \int_{\R^3} M_{(c_i,u,\theta)}(t,x,v) \dd v,\quad 1\leq i\leq N,
\\[2mm] \sum\limits_{i=1}^N m_i c_i(t,x) u(t,x) &=& \sum\limits_{i=1}^N\int_{\R^3} m_i v M_{(c_i,u,\theta)}(t,x,v)\dd v
\\[2mm] \frac{3}{2}\sum\limits_{i=1}^N m_i c_i(t,x) \theta(t,x) &=& \sum\limits_{i=1}^N\int_{\R^3} \frac{m_i}{2} \abs{v-u(t,x)}^2 M_{(c_i,u,\theta)}(t,x,v)\dd v.
\end{eqnarray*}
In particular we stress again the fact that, in contrast with the case where only one gas is considered, in the multi-species framework not all local Maxwellian vectors are local equilibrium states for the mixture, since the relation $\boldQ(\boldF,\boldF)=0$ is satisfied if and only if the local bulk velocity $u(t,x)$ and temperature $\theta(t,x)$ are the same for each species.

\smallskip
Going further, since we work in $\T^3$, one can also prove that the only global equilibrium of the mixture, \textit{i.e.} the unique stationary solution $\boldF$ to $\eqref{multiBE}$, is given by the global Maxwellian vector $\pa{M_{(c_{i,\infty},u_\infty,\theta_\infty)}}_{1\leq i\leq N}$ whose fluid quantities satisfy the relations $\eqref{conservationlaws}$. It is defined, for any $1\leq i\leq N$, by
\begin{equation*}
M_{(c_{i,\infty},u_\infty,\theta_\infty)}(v) = c_{i,\infty}\pa{\frac{m_i}{2\pi \theta_\infty}}^{3/2}e^{-m_i\frac{\abs{v-u_\infty}^2}{2\theta_\infty}},\quad \forall v\in\R^3.
\end{equation*}
In particular, without loss of generality, in what follows we shall consider as unique global equilibrium of the mixture the $N$-vector $\boldmu=(\mu_1,\ldots,\mu_N)$, defined componentwise as
\begin{equation}\label{mu}
\mu_i(v) = c_{i,\infty}\pa{\frac{m_i}{2\pi}}^{3/2}e^{-m_i\frac{\abs{v}^2}{2}},\quad \forall v\in\R^3,
\end{equation}
and obtained by a translation and a dilation of the coordinate system, which allow to choose~$u_\infty=0$ and $\theta_\infty=1$.

\smallskip
The Cauchy theory and the trend to equilibrium for solutions of equation $\eqref{multiBE}$ studied in a perturbative setting around the global Maxwellian state $\eqref{mu}$ have been shown in $L^\infty_x\big(\T^3;L^1_v(\R^3)\big)$ with polynomial weights \cite{BriDau} and in $L^\infty_{x,v}\big(\T^3\times\R^3\big)$ with exponential and polynomial weights \cite{Bri2}. In the present work we are interested in studying the diffusion limit of equation $\eqref{multiBE}$. More precisely, we consider the following scaled version of $\eqref{multiBE}$, given for any $1\leq i\leq N$ by
\begin{equation}\label{multi BE scaled}
\partial_t F_i^\eps + \frac{1}{\eps}v\cdot\nabla_xF_i^\eps = \frac{1}{\eps^2}Q_i(\boldF^\eps,\boldF^\eps) \quad \textrm{on }\R_+\times\T^3\times\R^3,
\end{equation}
and we investigate the behaviour of the fluid quantities $c_i^\eps(t,x)$, $u_i^\eps(t,x)$ and $\theta_i^\eps(t,x)$ associated to each distribution $F_i^\eps$, when the scaling parameter $\eps>0$ vanishes. A first formal derivation \cite{BouGreSal2} showed that, in the case where $\mathbf{F^\eps} = \pab{M_{(c_i^\eps,u_i^\eps,T)}}_{1\leq i \leq N}$ is a local Maxwellian vector with constant temperature $\bar{\theta} >0$ and $\pa{c^\eps_i,u^\eps_i}_{1\leq i \leq N}$ converge towards $\pa{c_i,u_i}_{1\leq i \leq N}$, the limit macroscopic quantities are solutions, on $\R_+\times\T^3$, to the flux-incompressible Maxwell-Stefan system
\begin{gather}
\partial_t c_i + \nabla_x \cdot \pa{c_i u_i} = 0, \label{MS mass}
\\[2mm] - \nabla_x c_i = \sum\limits_{j=1}^N c_ic_j\frac{u_i - u_j}{\Delta_{ij}}, \label{MS momentum}
\\[2mm]\quad \partial_t \pa{\sum\limits_{i=1}^N c_i} = 0, \quad \nabla_x\cdot \pa{\sum\limits_{i=1}^N c_i u_i} =0,\label{MS incompressibility}
\end{gather}
where $\Delta_{ij}$ are symmetric (with respect to the species indices) positive constants only depending on $\bar{\theta}$, the masses $(m_i)_{1\leq i\leq N}$ and the collision kernels $(B_{ij})_{1\leq i,j\leq N}$. Note that the incompressibility is to be understood in the whole and not for each species. At first sight, the structure of the limit equations looks rather different from the Navier-Stokes limit of the Boltzmann mono-species equation ($N=1$). However, as showed in \cite{BouGreSal2}, the momentum balance equation appears partly at order $\eps^2$ under the explicit form
$$\eps^2\frac{m_i}{\bar{\theta}}\cro{\partial_t\pa{c_i^\eps u_i^\eps} + \nabla_x\cdot \pa{c_i^\eps u_i^\eps \otimes u_i^\eps}} + \nabla_x c_i^\eps = \sum\limits_{j=1}^N c_i^\eps c_j^\eps \frac{u_i^\eps - u_j^\eps}{\Delta_{ij}}, $$
and we actually see the Navier-Stokes structure showing up, without viscous terms because the solution is supposed to be a local Maxwellian and therefore the interactions between the microscopic part and the fluid quantities are not taken into account. 

In particular, the Maxwell-Stefan system \eqref{MS mass}--\eqref{MS momentum}--\eqref{MS incompressibility} is of core importance in physics and biology, since it is used to model the evolution of diffusive phenomena in mixtures \cite{Max,Ste,TDBCH,Cha,BouGotGre}. As such, its derivation from the kinetic equations is of great interest from both a mathematical and a physical point of view. 

As already underlined, the problem at hand is reminiscent of the hydrodynamic limit of the mono-species Boltzmann equation, towards the incompressible Navier-Stokes system. Therefore, let us first give a brief description of the strategies which have been developed in this context. We emphasize that the list is not exhaustive, does not concern any other type of hydrodynamic limit (such as Euler equations or acoustics) and we refer to \cite{StRay,Gol1} for more references and discussions.
\par The formal derivation of the Navier-Stokes limit from the mono-species Boltzmann equation came from a series of articles called the Bardos-Golse-Levermore (BGL) program \cite{BGL1,BGL2}. It has been made rigorous in various ways, always with the help of a Taylor expansion of the solution with respect to the parameter $\eps$ \cite{Gra}. One could look at perturbative solutions around a global Maxwellian $\mu(v) = c_{\infty}/(2\pi \theta_{\infty})^{3/2} \exp\br{-\frac{\abs{v}^2}{2\theta_{\infty}}}$, that is to say solutions of the form $F^\eps = \mu + \eps f^\eps$, and study their limit when $\eps$ tends to 0. This has been done by describing the spectrum of the linear operator $L(f^\eps) = \eps^{-2}\pab{Q(\mu,f^\eps) + Q(f^\eps,\mu)}$ \cite{EllPin,BarUka}, by directly tackling the Cauchy problem for $f^\eps$ in Sobolev spaces to get convergence results on the fluid quantities of $f^\eps$ in the setting of renormalised solutions \cite{GolStRay1,GolStRay2,LevMas} or by using hypocoercivity techniques \cite{Guo3,Bri1}. Another strategy is to investigate the stability of the Boltzmann equation not around a global equilibrium, but around a local Maxwellian whose fluid quantities solve the limit macroscopic system. In this spirit, De Masi, Esposito and Lebowitz \cite{DMEL} studied solutions of the form $F^\eps = M_{(c_{\infty},\eps u,\theta_{\infty})}+ \eps f^\eps$ where $u=u(t,x)$ is a smooth solution of the perturbed incompressible Navier-Stokes equation, around the constant macroscopic equilibrium $(c_{\infty},0,\theta_{\infty})$. The interest of this method, initiated by Caflisch \cite{Caf1} for the Euler limit of the Boltzmann equation, is that the main order term already encodes the limit system and one can thus study the fluid perturbations explicitly. 

\smallskip
For the time being, only formal derivations of the Maxwell-Stefan system from the Boltzmann multi-species equation have been obtained \cite{BouGreSal2,HutSal1,BouGrePav} and, unlike the BGL program \cite{BGL2} for the mono-species case, they do not provide any convergence of the fluid quantities given an \textit{a priori} convergence in distribution. Note moreover that in the context of mixtures other formal hydrodynamic limits have been also recently derived in a non-dissipative regime \cite{BasEspLebMar, BisDes2,BiaDog2,BarBisBruDes} and in a stationary regime for two species in a slab \cite{Bru1,Bru2} but, up to our knowledge, no rigorous convergences have been provided yet. This work aims at filling this gap.

\smallskip
\par The strategy of studying close-to-global equilibrium solutions offers bounds (and thus compactness) in Sobolev spaces, but has to be transferred into a convergence of the nonlinear moments. We therefore chose to investigate the hydrodynamic limit of the multi-species Boltzmann equation in the spirit of \cite{Caf1,DMEL}. More precisely, we construct solutions in Sobolev spaces to the perturbed system $\eqref{multi BE scaled}$ of the form $\boldF^\eps =  \bepsM+ \eps \boldf^\eps$, where the local Maxwellian state $\bepsM=(\epsM_1,\ldots,\epsM_N)$ is given for any $1\leq i\leq N$ by
\begin{equation}\label{perturbed solution}
\epsM_i(t,x,v) = c_i(t,x) \pa{\frac{m_i}{2\pi \bar{\theta}}}^{3/2} \exp\br{-m_i\frac{|v-\eps u_i(t,x)|^2}{2 \bar{\theta}}},\quad t\geq 0,\ x\in\T^3,\ v\in\R^3,
\end{equation}
and its fluid quantities $(c_i,u_i)_{1\leq i \leq N}$ are perturbative solutions of the incompressible Maxwell-Stefan system \eqref{MS mass}--\eqref{MS momentum}--\eqref{MS incompressibility}, whose global existence and exponential relaxation to equilibrium have been recently obtained by the authors in \cite{BonBri}.

We however differ from \cite{DMEL} in several ways. First of all, the local Maxwellian $\eqref{perturbed solution}$ is not an equilibrium state for the mixture, meaning in particular that the relation $\boldQ(\bepsM,\bepsM)=0$ does not hold true in our setting. Secondly, even though we only deal with the case of a fixed temperature $\bar{\theta}$, we do not ask~$c_i(t,x)$ to be constant, but we shall consider perturbations of a constant state \cite{BonBri}. At last, we shall not use higher order fluid expansions, but we are able to develop a hypocoercive strategy in the spirit of \cite{MouNeu,Bri1}, by separating the leading order of the limit Maxwell-Stefan components from the full microscopic and fluid part perturbations. The idea is indeed natural, remembering that the Maxwell-Stefan system is obtained when one only looks at the interactions at the first order in $\eps$. 

\smallskip
\par We plug the perturbation $\boldF^\eps =  \bepsM+ \eps \boldf^\eps$ into the rescaled Boltzmann system $\eqref{multi BE scaled}$ and we obtain the perturbed equation
\begin{equation}\label{perturbed BE}
\partial_t \boldf^\eps + \frac{1}{\eps}v\cdot \nabla_x \boldf^\eps = \frac{1}{\eps^2}\mathbf{L^\eps}(\boldf^\eps) + \frac{1}{\eps}\mathbf{Q}(\boldf^\eps,\boldf^\eps) + \bepsS,
\end{equation}
where the linear operator $\bepsL=(\epsL_1,\ldots,\epsL_N)$ is given for any $1\leq i\leq N$ by
\begin{equation}\label{Leps}
L^\eps_i(\boldf^\eps) = \sum\limits_{j=1}^N \pabb{Q_{ij}(\epsM_i,f_j^\eps) + Q_{ij}(f_i^\eps,\epsM_j)},
\end{equation}
and the source term $\bepsS$ encodes the distance between the Maxwell-Stefan system and the fluid part of the perturbed Boltzmann equation, and is defined through the relation
\begin{equation}\label{Seps}
\bepsS = \frac{1}{\eps^3}\boldQ(\bepsM,\bepsM) - \frac{1}{\eps} \partial_t \bepsM- \frac{1}{\eps^2}v\cdot \nabla_x\bepsM.
\end{equation}
The pertubative setting \eqref{perturbed BE}--\eqref{Leps}--\eqref{Seps} is classical in the case of the Boltzmann mono-species equation. Usually, the strategy consists in proving that the linear operator $\mathbf{L^\eps}$ is self-adjoint in some space, with non-trivial null space, and features a spectral gap. In fact, the spectral gap allows to get strong negative feedback from the microscopic part~(orthogonal to~$\Ker\bepsL$) of the solution and therefore one is only left with finding a way to control the kernel part of the solution, in order to close the energy estimates. This can be achieved either by investigating the fluid equations \cite{Guo3,BriDau} or by using hypocoercive norms \cite{MouNeu,Vil4,Bri1} which generate a complete negative return in higher Sobolev norms, \textit{via} the commutator $[v\cdot \nabla_x , \nabla_v] = -\nabla_x$.
\par Unfortunately, the central spectral gap property essentially comes from the fact that, in the mono-species case, any local Maxwellian is a local equilibrium. As soon as $N >1$ we have seen that a local Maxwellian vector is an equilibrium of the system of equations $\eqref{multi BE scaled}$ only if each component shares the same velocity and temperature $\eqref{local equilibrium}$. This is not the case for the local Maxwellian state $\eqref{perturbed solution}$ we consider, where each species evolves at its own speed~$\eps u_i$. Not only one looses the self-adjointness of the linear operator $\bepsL$ in the usual $L^2$ space weighted by the local Maxwellian $\bepsM$, but also the spectral gap property. However, it has been recently showed in \cite{BonBouBriGre} that if one does not have a spectral gap, it is still possible to recover a negative return on the fluid part by adding an error term of order $\eps$, as long as the velocities $(\eps u_i)_{1\leq i\leq N}$ remain small, of order $\eps$. Our contribution aims at extending this perturbative result to the case of more general densities $c_i(t,x)$, in particular to the case of perturbative solutions of the Maxwell-Stefan system, where the loss of the spectral gap remains at a lower order of magnitude. We then construct a new modified Sobolev norm which allows to recover a coercivity property and to close the energy estimates on the nonlinear terms. This is an extension and an adaptation of the hypocoercivity methods developed in \cite{MouNeu,Bri1}, where fine controls over the $\eps$-dependencies of the coefficients act to counter-balance the out-of-equilibrium property of our system, as well as the interactions with the non constant fluid quantities appearing in the weight. We emphasize again that this is achieved by looking at the perturbative solutions of the Maxwell-Stefan system constructed in \cite{BonBri}. Such a setting also appears in standard studies in the mono-species case, where one only recovers perturbative Leray solutions of the incompressible Navier-Stokes equation. Besides, the perturbative setting for the Maxwell-Stefan system proves itself to be exactly the one needed to control the problematic source term $\bepsS$. In particular, we underline once more that the perturbative setting we look at extends, to a non constant density and a non-zero momentum macroscopic equilibria (even though the perturbations we consider are small), the previous results of Caflisch \cite{Caf1} and De Masi, Esposito and Lebowitz \cite{DMEL} in the mono-species case, and provides a hypocoercive approach to their framework. 


The rest of the paper is structured as follows. In Section \ref{sec:Main results} we introduce the notations, we state our main theorem and we also provide a thorough description of the methods and strategies built up to prove it. Using the expansion $\boldF^\eps=\bepsM +\eps\boldf^\eps$ around the local Maxwellian vector $\bepsM$ given by $\eqref{perturbed solution}$ we establish in Section \ref{sec:perturbed BE} global existence and uniqueness of the fluctuations $\boldf^\eps$, which are solutions to the perturbed Boltzmann equation $\eqref{perturbed BE}$. The uniform in $\eps$ \textit{a priori} estimates eventually allow us to prove a stability result which tells that the distribution function $\boldF^\eps$ remains close to the local Maxwellian $\bepsM$ up to an order $\eps$, thus ensuring a rigorous derivation of the Maxwell-Stefan system \eqref{MS mass}--\eqref{MS momentum}--\eqref{MS incompressibility} from the Boltzmann multi-species equation $\eqref{perturbed BE}$. In particular, our strategy exploits the hypocoercive structure of $\eqref{perturbed BE}$, following the works \cite{MouNeu,Bri1}. This consists in proving a series of \textit{a priori} estimates satisfied by the operators~$\bepsL$ and $\boldQ$, and by the source term $\bepsS$. While in Section \ref{sec:perturbed BE} we only introduce these fundamental results, all their technical proofs are finally collected in Section \ref{sec:technical proofs}.
\bigskip

\section{Main result}\label{sec:Main results}

\subsection{Notations and Conventions}
Let us begin by detailing all the notations that we use throughout the article. Recalling that we denote $\mathbf{W}=(W_1,\ldots,W_N)$ any vector or vector-valued operator belonging to $\R^N$, we shall use the symbol $\mathbf{1}$ to name the specific vector $(1,\dots,1)$. Henceforth, the multiplication of $N$-vectors has to be understood in a component by component way, so that for any $\mathbf{w}, \mathbf{W}\in\R^N$ and any $q\in\mathbb{Q}$ we have
\begin{equation*}
\mathbf{w}\mathbf{W} = (w_i W_i)_{1\leq i\leq N},\quad \mathbf{W}^q=(W_i^q)_{1\leq i\leq N}.
\end{equation*}
Moreover, we introduce the Euclidean scalar product in $\R^N$, weighted by a vector $\mathbf{W}$, which is defined as
$$\langle \mathbf{f},\mathbf{g}\rangle_{\mathbf{W}}= \sum\limits_{i=1}^Nf_ig_i W_i,$$
and induces the norm $\norm{\boldf}_{\mathbf{W}}^2=\scalprod{\boldf,\boldf}_{\mathbf{W}}$. In particular, when $\mathbf{W}=\mathbf{1}$, the index $\mathbf{1}$ will be dropped in both the notations for the scalar product and the norm. At last, note that we shall also make use of the following shorthand notation
$$\langle v \rangle = \sqrt{1+\abs{v}^2},$$
not to be confused with the scalar product defined above.

\smallskip
The convention we choose for the functional spaces is to index the space by the name of the concerned variable. For $p\in [1,+\infty]$ we have
$$L^p_{t} = L^p (0,+\infty),\quad L^p_x = L^p\left(\T^3\right), \quad L^p_v = L^p\left(\R^3\right).$$
Consider now some positive measurable vector-valued functions $\mathbf{w}:\T^3\to(\R_+^*)^N$ in the variable $x$ and $\func{\mathbf{W}}{\R^3}{(\R_+^*)^N}$ in the variable $v$. For any $1\leq i\leq N$, we define the weighted Hilbert spaces $L^2(\T^3,w_i)$ and $L^2(\R^3, W_i)$ by introducing the respective scalar products and norms
\begin{equation*}
\begin{array}{lll}
\displaystyle \langle c_i,d_i\rangle_{L_x^2(w_i)}=\int_{\T^3}c_i d_i w_i^2\dd x,  &\ \  \| c_i\|^2_{L_x^2(w_i)}=\langle c_i,c_i\rangle_{L_x^2(w_i)}, &\ \  \forall c_i,d_i\in L^2(\T^3,w_i),
\\ \\  \displaystyle \langle f_i,g_i\rangle_{L_v^2(W_i)}=\int_{\R^3}f_i g_i W_i^2\dd v,  &\ \  \| f_i\|^2_{L_v^2(W_i)}=\langle f_i,f_i\rangle_{L_v^2(W_i)},  &\ \  \forall f_i,g_i\in L^2(\R^3,W_i).
\end{array}
\end{equation*}
With these definitions, we say that $\boldc:\T^3\to\R^N$ belongs to $L^2(\T^3,\mathbf{w})$ and that $\boldf:\R^3\to\R^N$ belongs to $L^2(\R^3, \mathbf{W})$ if and only if $c_i:\T^3\to\R\in L^2(\T^3,w_i)$ and $f_i :\R^3\to\R\in L^2(\R^3,W_i)$ for any $1\leq i\leq N$. The weighted Hilbert spaces $L^2(\T^3,\mathbf{w})$ and $L^2(\R^3,\mathbf{W})$ are therefore endowed with the induced scalar products and norms
\begin{gather*}
\scalprod{\boldc,\mathbf{d}}_{L^2_x(\mathbf{w})} = \sum_{i=1}^N\langle c_i,d_i\rangle_{L_x^2(w_i)},\qquad \|\boldc\|_{L_x^2(\mathbf{w})}=\left(\sum_{i=1}^N\| c_i\|^2_{L_x^2(w_i)}\right)^{1/2},
\\[3mm]  \langle\boldf,\mathbf{g}\rangle_{L_v^2(\mathbf{W})}=\sum_{i=1}^N\langle f_i,g_i\rangle_{L_v^2(W_i)},\qquad \|\boldf\|_{L_v^2(\mathbf{W})}=\left(\sum_{i=1}^N\| f_i\|^2_{L_v^2(W_i)}\right)^{1/2},
\end{gather*}
Note that in the specific case of positive measurable functions $\mathbf{W}:\R^3\to(\R_+^*)^N$ in the sole variable $v$, without risk of confusion we shall also consider the weighted Hilbert space~${L^2(\T^3\times\R^3,\mathbf{W})}$, defined similarly by the natural scalar product and norm 
\begin{equation*}
\begin{split}
\scalprod{\boldf,\boldg}_{L^2_{x,v}(\mathbf{W})}  & = \sum_{i=1}^N \int_{\T^3\times\R^3} f_i(x,v) g_i(x,v) W_i^2(v)\dd x\dd v,
\\[3mm] \norm{\boldf}_{L^2_{x,v}(\mathbf{W})}^2  & = \sum_{i=1}^N\int_{\T^3\times\R^3} f_i^2(x,v) W_i^2(v)\dd x\dd v,
\end{split}
\end{equation*}
for any $\boldf,\boldg\in L^2_{x,v}(\T^3\times\R^3,\mathbf{W})$.

Finally, in the same way we can introduce the corresponding weighted Sobolev spaces. Consider two multi-indices $\alpha,\beta\in\N^3$, of lengths $|\alpha|=\sum_{k=1}^3 \alpha_k$ and~$|\beta|=\sum_{k=1}^3 \beta_k$ respectively . We shall use the convention that $\alpha$ always refers to $x$-derivatives while $\beta$ refers to~$v$-derivatives. Note in particular that we shall use the standard notation of the canonical basis in $\R^3$ to name the specific multi-indices having one component equal to 1 and the others equal to 0, so for example $e_1=(1,0,0)$.

For any $s\in\N$ and any vector-valued functions $\boldc\in H^s(\T^3,\mathbf{w})$ and $\boldf\in H^s(\mathcal{S},\mathbf{W})$, where either $\mathcal{S}=\R^3$ or $\mathcal{S}=\T^3\times\R^3$, we define the norms
\begin{equation*}
\begin{split}
\norm{\boldc}_{H^s_x(\mathbf{w})} & = \left(\sum\limits_{i=1}^N \sum\limits_{|\alpha|\leq s}\norm{\partial_x^\alpha c_i}^2_{L^2_{x}(w_i)}\right)^{1/2},\\[3mm]
\norm{\mathbf{f}}_{H^s_{v}(\mathbf{W})} & = \left(\sum\limits_{i=1}^N \sum\limits_{|\beta| \leq s}\norm{\partial_v^\beta f_i}^2_{L^2_{v}(W_i)}\right)^{1/2},\\[3mm] 
\norm{\mathbf{f}}_{H^s_{x,v}(\mathbf{W})} & = \left(\sum\limits_{i=1}^N \sum\limits_{\abs{\alpha}+\abs{\beta}\leq s}\norm{\multideriv f_i}^2_{L^2_{x,v}(W_i)}\right)^{1/2}.
\end{split}
 \end{equation*}


\smallskip
\subsection{Statement of the result and strategy}

Thanks to the Cauchy theory built up by the authors \cite{BonBri}, we can construct the local Maxwellian $\bepsM$ whose fluid quantities are perturbative solutions of the Maxwell-Stefan system \eqref{MS mass}--\eqref{MS momentum}--\eqref{MS incompressibility}. For the reader convenience, let us first recall the result obtained in \cite{BonBri}.

\begin{theorem}\label{theo:Cauchy MS}
Let $s> 3$ be an integer, $\func{\bar{u}}{\R_+\times\T^3}{\R^3}$ be in $L^\infty\pab{\R_+; H^{s}(\T^3)}$ with~${\nabla_x\cdot\bar{u}=0}$, and consider $\mathbf{\bar{c}} > 0$. There exist $\delta_\ms$, $C_\ms$, $C'_\ms$, $\lambda_\ms >0$ such that for all~${\eps\in (0,1]}$ and for any initial datum $(\tilde{\mathbf{c}}^\init, \tilde{\mathbf{u}}^\init)\in H^s(\T^3)\times H^{s-1}(\T^3)$ satisfying, for almost any $x\in\T^3$ and for any $1\leq i\leq N$,
\begin{itemize}
\item[(i)]\textbf{Mass compatibility: } $\disp{\sum_{i=1}^N \tilde{c}^\init_i(x) = 0 \quad\mbox{and}\quad \int_{\T^3}\tilde{c}_i^\init(x) \dd x = 0}$, \\[-0.5mm]
\item[(ii)]\textbf{Mass positivity: } $\disp{\bar{c}_i +\eps \tilde{c}_i^\init (x) > 0}$, \\[-1.5mm]
\item[(iii)]\textbf{Moment compatibility: } $\disp{\nabla_x \tilde{c}^\init_i = \sum_{j\neq i}\frac{c_i^\init c_j^\init}{\Delta_{ij}}\pa{\tilde{u}_j^\init - \tilde{u}_i^\init}}$,\\[-0.5mm]
\item[(iv)] \textbf{Smallness assumptions: } $\disp{\norm{\tilde{\mathbf{c}}^\init}_ {H^s_x}\leq \delta_\ms}\quad$ and $\quad\disp{\norm{\bar{u}}_{L^\infty_t H^{s}_x}\leq \delta_\ms}$,
\end{itemize}
there exists a unique weak solution 
$$\pa{\mathbf{c},\mathbf{u}} = \pab{\bar{\mathbf{c}}+\eps \tilde{\mathbf{c}}, \bar{\mathbf{u}}+\eps\tilde{\mathbf{u}}}$$
in $L^\infty\pab{\R_+;H^s(\T^3)} \times L^\infty\pab{\R_+;H^{s-1}(\T^3)}$ to the incompressible Maxwell-Stefan system \eqref{MS mass}--\eqref{MS momentum}--\eqref{MS incompressibility}, such that initially~${\restriction{\pa{\tilde{\mathbf{c}},\tilde{\mathbf{u}}}}{t=0} = \pa{\tilde{\mathbf{c}}^\init,\tilde{\mathbf{u}}^\init}}$ a.e. on~$\T^3$. In particular, if $s>4$ and $\bar{u}\in C^0\big(\R_+; H^s(\T^3)\big)$, then the couple $(\mathbf{c},\mathbf{u})$ also belongs to $C^0\pab{\R_+;H^{s-1}(\T^3)} \times C^0\pab{\R_+;H^{s-2}(\T^3)}$.

Moreover, $\mathbf{c}$ is positive and the following relations hold a.e. on~${\R_+\times\T^3}$:
\begin{equation}\label{equimolar vectorial}
\scalprod{\mathbf{c},\tilde{\mathbf{u}}}=\sum_{i=1}^N c_i(t,x) \tilde{u}_i(t,x)=0 \quad\mbox{ and }\quad \int_{\T^3}\tilde{c}_i(t,x) \dd x = 0.
\end{equation}
Finally, for almost any time $t\geq 0$
\begin{eqnarray*}
\norm{\tilde{\mathbf{c}}}_{\sobolevx{s}} &\leq&   e^{- t \lambda_\ms}\norm{\tilde{\mathbf{c}}^\init}_{\sobolevx{s}},
\\[4mm]    \norm{\tilde{\mathbf{u}}}_{H^{s-1}_x} &\leq&  C_\ms e^{- t \lambda_\ms}\norm{\tilde{\mathbf{c}}^\init}_{\sobolevx{s}},
\\[3mm]    \int_0^t e^{2 (t-\tau) \lambda_\ms}\norm{\tilde{\mathbf{u}}(\tau)}^2_{H^{s}_x}\dd \tau &\leq & C'_\ms \norm{\tilde{\mathbf{c}}^\init}^2_{\sobolevx{s}}.
\end{eqnarray*}
The constants $\delta_\ms$, $\lambda_\ms$, $C_\ms$ and $C'_\ms$ are constructive and only depend on $s$, the number of species $N$, the diffusion coefficients $(\Delta_{ij})_{1\leq i,j\leq N}$ and the constant vector~$\mathbf{\bar{c}}$. In particular, they are independent of the parameter $\eps$.
\end{theorem}

\medskip
Now, recall that we can select as unique global equilibrium of the mixture the global Maxwellian state $\boldmu$ defined by $\eqref{mu}$, where $u_\infty=0$ and $\theta_\infty=1$. Thanks to the above result, we can choose the macroscopic equilibrium state $(\bar{c}_i,\bar{u})_{1\leq i\leq N}$ to be compatible with the global (kinetic) equilibrium of the mixture $\boldmu$, by taking $\bar{c}_i=c_{i,\infty}$ for any $1\leq i\leq N$. Therefore, supposing the uniform (in space) and constant (in time) temperature $\bar{\theta}$ to be equal to 1 for simplicity, we introduce the local Maxwellian vector $\bepsM=\mathbf{M}_{(\boldc,\eps\boldu,\mathbf{1})}=(\epsM_1,\ldots,\epsM_N)$, given for any $1\leq i\leq N$ by
\begin{equation}\label{Local Maxwellian eps}
\epsM_i(t,x,v)=c_i(t,x)\pa{\frac{m_i}{2\pi}}^{3/2}\exp\br{-m_i\frac{|v-\eps u_i(t,x)|^2}{2}},\quad  t\geq 0,\ x\in\T^3,\ v\in\R^3,
\end{equation}
where $(\boldc,\boldu)$ is the unique weak solution of the Maxwell-Stefan system \eqref{MS mass}--\eqref{MS momentum}--\eqref{MS incompressibility}, perturbed around the macroscopic equilibrium state $(\boldc_{\infty},\bar{\boldu})$. More precisely, the fluid quantities of $\bepsM$ take the form
\begin{equation}\label{Meps fluid quantities}
\left\{\begin{array}{l}
c_i(t,x) = c_{i,\infty} + \eps \tilde{c}_i(t,x),\\[5mm]
u_i(t,x) = \bar{u}(t,x) + \eps \tilde{u}_i(t,x),\ \ \nabla_x\cdot \bar{u}(t,x) = 0,
\end{array}\right.\qquad t\geq 0,\ x\in\T^3,
\end{equation}
for any $1\leq i\leq N$. In particular, we notice that under the assumptions of Theorem \ref{theo:Cauchy MS} the following fundamental properties are verified by the couple $(\boldc_{\infty}+\eps\tilde{\boldc},\bar{\boldu}+\eps\tilde{\boldu})$:
\begin{itemize}
\item[(i)] $\disp \inf_{\R^+\times\T^3}\min_{1\leq i\leq N} \pab{c_{i,\infty}+\eps \tilde{c}_i(t,x)}  > 0,$\\[1mm]
\item[(ii)] $\disp \scalprod{\boldc_{\infty}+\eps\tilde{\boldc}(t,x),\mathbf{1}} = C_0 >0\quad \textrm{a.e. on }\ \R_+\times\T^3,$\\[1mm]
\item[(iii)] $\disp \norm{\tilde{\boldc}}_{L^\infty_t \spacex{s}}\leq \delta_\ms,$\\[1mm]
\item[(iv)] $\disp \norm{\bar{u}}_{L^\infty_t H^s_x}\leq\delta_\ms\qquad \textrm{ and }\qquad \norm{\tilde{\boldu}}_{L^\infty_t H^s_x}\leq \delta_\ms C_\ms.$
\end{itemize}

Starting from this choice of the local Maxwellian, we consider solutions to the Boltzmann multi-species equation $\eqref{multi BE scaled}$ of the form $\boldF^\eps=\bepsM+\eps\boldf$, where the fluctuations $\boldf$ satisfy the perturbed system $\eqref{perturbed BE}$.

\smallskip
Our strategy for developing a uniform Cauchy theory for equation $\eqref{perturbed BE}$ is inspired by the works \cite{MouNeu, Bri1} and aims at building a suitable Sobolev-equivalent norm which satisfies a Gr\"{o}nwall-type inequality among solutions of $\eqref{perturbed BE}$. The idea of the method originates from the hypocoercive behaviour shown by some classical kinetic equations in the mono-species framework, where the interaction of a degenerate coercive operator with a conservative operator may induce global dissipation in all variables, and consequently relaxation towards equilibrium. A typical example is precisely the inhomogeneous Boltzmann mono-species equation. In fact, the mono-species Boltzmann operator $L$ linearized around a global equilibrium of the gaz exhibits a spectral gap which translates into a negative return in some Hilbert space depending on the sole velocity variable $v$. In particular, $L$ is degenerate in the sense that its kernel is much larger than the set of global equilibria. Here, precisely comes into play the effect of the conservative transport operator $v\cdot\nabla_x$, that introduces a dependence on the space variable $x$ which at first sight cannot be handled using the dissipation of $L$ in $v$. Nevertheless, it is possible to prove \cite{Vil4,MouNeu} that the association of these operators can actually produce a global negative return in both $x$ and $v$, if one considers a well-designed Lyapunov functional that is able to transfer the dissipation of $L$ into a (hypo) dissipation of $T=L-v\cdot\nabla_x$. One possibility is for example to introduce \cite{MouNeu} a modified Sobolev norm where we add to the usual $H^s_{x,v}$ norm, new suitable terms based on commutators of higher derivatives, such as the well-known $[v\cdot\nabla_x,\nabla_v]=-\nabla_x$. In this way, the linear part is dealt with, and the full nonlinear Boltzmann equation (close to equilibrium) can be tackled \cite{MouNeu} and its solution can be proved to relax towards a global equilibrium, with an exponential decay rate of convergence.

As already mentioned in the introduction, a similar strategy has been developed in \cite{Bri1} for the study of the hydrodynamic limit of the Boltzmann mono-species equation set under a standard diffusive scaling, leading to the same conclusions as \cite{MouNeu}. 
In particular, a global dissipation can be obtained if one introduces a new Sobolev norm which is adapted from the one used in \cite{MouNeu}, by adding a dependence on $\eps$ in its terms. In such a way, the operator $T^\eps = \eps^{-2} L - \eps^{-1}v\cdot \nabla_x$ exhibits the required hypocoercive behaviour, which provides a control on the nonlinear stiff term $\eps^{-1} Q$ and allows to ensure a global negative return in both $x$ and $v$, and the expected convergence to equilibrium for the solution.

In the multi-species setting we consider here, even if the linearization around the non-equilibrium Maxwellian $\bepsM$ gives rise to a linearized Boltzmann equation $\eqref{perturbed BE}$ involving the new operators $\bepsL$ and $\bepsS$, we wish to prove that the underlying hypocoercive structure of our model is in fact almost formally identical to the one featured by the mono-species equation studied in \cite{Bri1}. 
The approach that will guide our presentation from now on will consist in adapting the tools and computations developed in \cite{Bri1}, in order to study equation $\eqref{perturbed BE}$. However, we stress the fact that this is not a straightforward extension. Indeed, a full spectral gap property for $\bepsL$ is missing \cite{BonBouBriGre} when we settle our analysis in an Hilbert space weighted by the global equilibrium $\boldmu$, which is incompatible with the operator $\bepsL$, linearized around the local Maxwellian $\bepsM$. In particular, we recall that we see the appearance of an extra positive term of order $\eps$, which contains the projected part~$\pi_\boldL(\boldf)$ and precisely takes into account this incompatibility. It is therefore fundamental to ensure that this loss in the spectral gap 
does not affect too much the computations derived in \cite{Bri1} and that $\bepsL$ can still provide a complete negative return in combination with the transport operator $v\cdot \nabla_x$. This feature is indeed characteristic of our model and will appear any time the linearized operator $\bepsL$ (and more specifically its higher derivatives in $x$) is involved. It is at this point that the specific form of the fluid quantities $\eqref{Meps fluid quantities}$ prominently comes into play in order to gain a lower order factor in $\eps$ or in $\delta_\ms$. In particular, we underline the importance of having $\delta_\ms$ as a free parameter of the problem, since its choice shall be crucial in order to obtain a careful control on some of the extra terms produced by $\bepsL$, where the presence of the factor $\eps$ will not be enough to close the estimates as in \cite{Bri1}. Finally, we remind the reader the additional difficulties caused by the source term $\bepsS$, which displays an intricate dependence on the local Maxwellian $\bepsM$, and is at first glance of order $\mathcal{O}(\eps^{-3})$. Again, the particular form of the solution $(\boldc,\boldu)$ is at the core of our derivation of the correct estimates to control~$\bepsS$, and the perturbative setting we consider for the Maxwell-Stefan system will prove itself sufficient in order to handle the negative powers of $\eps$.


\smallskip
\noindent \textbf{Choice of the functional spaces.} It has been pointed out in \cite{BonBouBriGre} that even if $\bepsL$ exhibits no clear self-adjointness in the usual space of linearization $L^2\big((\bepsM)^{-\frac{1}{2}},\R^3\big)$ it is still possible to recover a partial spectral gap property for $\bepsL$ by linking it to the Boltzmann operator $\boldL$, linearized around the global equilibrium $\boldmu$. This connection is obtained by choosing $\hilbertR$ as the space of interest. Therefore, in order to exploit the result of \cite{BonBouBriGre}, from now on we set our study in the corresponding weighted Sobolev spaces~$\sobolevTR{s}$, defined for any $s\in\N$.

\smallskip
\noindent\textbf{Choice of a referent hypocoercive operator.} Next, we must identify the dissipative and the conservative operators that play a central role in the theory of hypocoercivity. Since it is crucial to determine the explicit expression of the equilibria of the mixture, $\bepsL$ is not a valuable choice as we possess no information about the shape of its kernel. Therefore, as in the mono-species case \cite{MouNeu,Bri1}, we select the hypocoercive operator of interest to be defined for any $\eps\in (0,1]$ by $\boldT^\eps = \eps^{-2}\boldL-\eps^{-1}v\cdot\nabla_x$, acting on $\sobolevTR{s}$. In particular, recall that the dissipative operator $\boldL=(L_1,\ldots,L_N)$ is defined for any $1\leq i\leq N$ by
$$L_i(\boldf)=\sum_{j=1}^N \Big(   Q_{ij}(\mu_i,f_j)+Q_{ij}(f_i,\mu_j)   \Big),$$
and is a closed self-adjoint operator in the space $\hilbertR$. Moreover, its kernel is described by an orthonormal basis $\big(\pmb{\phi}^{(k)}\big)_{1\leq k\leq N+4}$ in $\hilbertR$, that is
$$\Ker \boldL = \Span\Big(\pmb{\phi}^{(1)},\ldots,\pmb{\phi}^{(N+4)}\Big),$$
where
$$\left\{\begin{array}{ll}
\displaystyle \pmb{\phi}^{(i)}=\frac{1}{\sqrt{c_{i,\infty}}}\mu_i\pab{\delta_{ij}}_{1\leq j\leq N} = \frac{1}{\sqrt{c_{i,\infty}}}\mu_i\mathbf{e}^{\text{\tiny (i)}}, \quad 1\leq i\leq N, \\[9mm]
\displaystyle \pmb{\phi}^{(N+\ell)}=\frac{v_\ell}{\left(\sum_{j=1}^N m_j c_{j,\infty}\right)^{1/2}}\big(m_i \mu_i\big)_{1\leq i\leq N}, \quad 1\leq \ell\leq 3,  \\[9mm]
\displaystyle \pmb{\phi}^{(N+4)}=\frac{1}{\left(\sum_{j=1}^N c_{j,\infty}\right)^{1/2}}\left(\frac{m_i |v|^2-3}{\sqrt{6}} \mu_i\right)_{1\leq i\leq N}.
\end{array}\right.$$
In this way, we can write the orthogonal projection onto $\Ker \boldL$ in $\hilbertR$ as
\begin{equation}\label{Projection}
\pi_{\boldL}(\boldf)(v)=\sum_{k=1}^{N+4}\big\langle\boldf,\pmb{\phi}^{(k)}\big\rangle_{\hilbertv}\ \pmb{\phi}^{(k)}(v),\quad\forall\boldf\in \hilbertR.
\end{equation}
In particular, its explicit expression is given by
\begin{equation}\label{Projection explicit}
\begin{split}
\pi_{\boldL}(\boldf) = &\sum_{i=1}^N \cro{\frac{1}{c_{i,\infty}}\int_{\R^3}f_i\dd v}\mu_i \mathbf{e}^{\text{\tiny (i)}}
\\[3mm]  &  +\sum_{k=1}^3\frac{v_k}{\sum_{i=1}^N m_i c_{i,\infty}}\cro{\sum_{i=1}^N\int_{\R^3}m_i v_k f_i\dd v}\pab{m_i\mu_i}_{1\leq i\leq N} 
\\[3mm]  &  +\frac{1}{\sum_{i=1}^N c_{i,\infty}}\cro{\sum_{i=1}^N\int_{\R^3}\frac{m_i|v|^2-3}{\sqrt{6}}f_i\dd v}\pa{\frac{m_i|v|^2-3}{\sqrt{6}}\mu_i}_{1\leq i\leq N}.
\end{split}
\end{equation}

\smallskip
\noindent\textbf{Choice of the norm.} With the aim of deriving similar \textit{a priori} estimates to \cite{Bri1}, we choose the same Sobolev-equivalent norm. For any $s\in\N^*$ and any $\eps\in (0,1]$, we introduce the following functional, defined on the space $\sobolevTR{s}$ by
\begin{multline}\label{Sobolev norm}
\norm{\cdot}_{\mathcal{H}^s_\eps} = \left[\sum_{|\alpha|\leq s} a^{(s)}_{\alpha}\norm{\partial^\alpha_x\cdot}_{\hilbertxv}^2  + \eps \sum_{\substack{|\alpha|\leq s   \\[0.2mm]   k,\ \alpha_k >0}} b^{(s)}_{\alpha,k}\scalprod{\partial^{\alpha}_x\cdot,\partial^{\ee_k}_{v}\partial^{\alpha-\ee_k}_x\cdot}_{\hilbertxv} \right.
\\[-2mm]    \left.+\ \eps^2\sum_{\substack{|\alpha|+|\beta|\leq s  \\[0.2mm]  |\beta|\geq 1}} d^{(s)}_{\alpha,\beta}\norm{\multideriv \cdot}_{\hilbertxv}^2\right]^{1/2},
\end{multline}
for some positive constants $\pab{a^{(s)}_{\alpha}}_{\alpha}^s$, $\pab{b^{(s)}_{\alpha,k}}_{\alpha,k}^s$ and $\pab{d^{(s)}_{\alpha,\beta}}_{\alpha,\beta}^s$ to be appropriately fixed later.

\smallskip
Thanks to these choices, we can now establish our main result.

\begin{theorem}\label{theo:Cauchy BE}
Let the collision kernels $B_{ij}$ satisfy assumptions $\mathrm{(H1)}$--$\mathrm{(H2)}$--$\mathrm{(H3)}$--$\mathrm{(H4)}$, and consider the local Maxwellian $\bepsM$ defined by \eqref{Local Maxwellian eps}--\eqref{Meps fluid quantities}. There exist $s_0\in\N^*$, $\bar{\delta}_\ms>0$ and~${\eps_0\in (0,1]}$ such that the following statements hold for any integer $s \geq s_0$.
\begin{itemize}
\item[(i)] There exist three sets of positive constants $\pab{a^{(s)}_{\alpha}}_{\alpha}^s$, $\pab{b^{(s)}_{\alpha,k}}_{\alpha,k}^s$ and $\pab{d^{(s)}_{\alpha,\beta}}_{\alpha,\beta}^s$ such that, for all $\eps\in (0,\eps_0]$, both following norms are equivalent
\begin{equation*}
\quad \norm{\cdot}_{\mathcal{H}^s_\eps} \sim  \pa{\norm{\cdot}_{\hilbertxv}^2 + \sum_{|\alpha| \leq s}\norm{\partial^\alpha_x\cdot}_{\hilbertxv}^2 + \eps^2\sum_{\substack{|\alpha|+|\beta|\leq s   \\[0.2mm]   |\beta|\geq 1}}\norm{\multideriv\cdot}_{\hilbertxv}^2}^{1/2}.
\end{equation*}
\item[(ii)] There exists $\delta_\bb >0$ such that, for all $\eps\in (0,\eps_0]$, for all $\delta_\ms\in [0,\bar{\delta}_\ms]$ and for any initial datum $\boldf^\init$ in~$\sobolevTR{s}$ with
\begin{equation*}
\norm{\boldf^\init}_{\mathcal{H}^s_\eps}\leq \delta_\bb,\qquad \norm{\pi_\bepsT(\boldf^\init)}_{\hilbertxv}\leq C\delta_\ms,
\end{equation*}
for some positive constant $C >0$ independent of the parameters $\eps$ and $\delta_\ms$, there exists a unique $\boldf\in C^0\pa{\R_+;\sobolevTR{s}}$ such that $\boldF^\eps=\bepsM+\eps\boldf$ is the unique weak solution of the Boltzmann multi-species equation $\eqref{perturbed BE}$. Moreover, if~${\boldF^{\eps,\init}=\mathbf{M}^{\eps,\init}+\eps\boldf^\init\geq 0}$, then $\boldF^\eps(t,x,v)\geq 0$ almost everywhere on $\R_+\times\T^3\times\R^3$. Finally, for any time $t\geq 0$, $\boldF^\eps$ satisfies the stability property
\begin{equation*}
\norm{\boldF^\eps-\bepsM}_{\mathcal{H}^s_\eps}\leq \eps \delta_\bb, \quad \forall\eps\in (0,\eps_0].
\end{equation*}
\end{itemize}
The constant $\delta_\bb$ is explicit and only depend on the number of species $N$, on the atomic masses $(m_i)_{1\leq i\leq N}$, and on the cross sections $(B_{ij})_{1\leq i,j\leq N}$. In particular, it is independent of the parameters $\eps$ and $\delta_\ms$.
\end{theorem}

\begin{remark}
Let us make a few comments that help in clarifying some important features of the above result.
\begin{itemize}
\item[(i)] First of all, the uniqueness has to be understood in a perturbative regime, that is among solutions of the form $\bepsM+\eps\boldf$ with small $\boldf$. Moreover, due to the presence of the source term $\bepsS$ we do not obtain an exponential decay to equilibrium for the perturbation $\boldf$, but we expect to recover it in the case where $\bar{\mathbf{u}}=0$ (see Remark \ref{rem:expo decay}).
\item[(ii)] We also emphasize that the Maxwell-Stefan model we consider is not properly the hydrodynamic limit of the Boltzmann multi-species equation, since the fluid quantities of the perturbation $\boldf$ have the same order of magnitude as the ones of the local Maxwellian $\bepsM$. However, we here prove that the system \eqref{MS mass}--\eqref{MS momentum}--\eqref{MS incompressibility} is stable in the limit $\eps\to 0$ of the perturbed equation $\eqref{perturbed BE}$. In other words, our result shows that the arising Maxwell-Stefan system is actually a stable perturbation of Fick's multi-species model, whose derivation from equation $\eqref{multi BE scaled}$ has been recently proved in \cite{BriGre}.
\item[(iii)] At last, it is important to note that since the proof of Theorem \ref{theo:Cauchy BE} is based on the hypocoercive formalism, its procedure can be actually applied to a wider class of limit macroscopic models which satisfy some minimal abstract hypotheses. In particular, our theorem could be easily generalized to cover a lot of previous results regarding the diffusive limit of the Boltzmann mono-species equation (and their counterparts in the multi-species setting), like the rigorous derivation of the Navier-Stokes equations obtained via Hilbert expansion \cite{DMEL} or by means of hypocoercivity techniques \cite{Bri1}. 
We shall not give any details of its proof since it could be recovered with a slight adjustment to our next computations, which is out of the principal scope of this work.
\end{itemize}
\end{remark}

\begin{theorem}[More general result]
Let the collision kernels $B_{ij}$ satisfy assumptions $\mathrm{(H1)}$--$\mathrm{(H2)}$--$\mathrm{(H3)}$--$\mathrm{(H4)}$, and consider some general fluid system of equations $\mathcal{P}(\boldc,\boldu) = 0$. Let the system possess a unique perturbative solution $(\boldc,\boldu)=(\bar{\boldc} + \eps\tilde{\boldc},\bar{\boldu} + \eps\tilde{\boldu})$ around any macroscopic equilibrium $(\bar{\boldc},\bar{\boldu})$ of $\mathcal{P}$, with $\bar{\boldc}\in (\R^*_+)^N$ and $\bar{\boldu} = (\bar{u},\ldots,\bar{u})$. Construct the local Maxwellian $\bepsM$ defined by \eqref{Local Maxwellian eps}, using the fluid quantities $(\boldc,\boldu)$ and suppose that there exists a constant $\delta_{\mathrm{fluid}} >0$ verifying
\vspace*{1mm}
\begin{itemize}
\item[(i)] $\displaystyle \tilde{\boldc}\in L^\infty\pab{\R_+;H^s(\T^3)}\quad\textrm{and}\quad \norm{\tilde{\boldc}}_{L^\infty_t H^s_x} \leq \delta_{\mathrm{fluid}},$\\[2mm]
\item[(ii)] $\displaystyle \bar{\boldu},\ \tilde{\boldu}\in L^\infty\pab{\R_+;H^{s-1}(\T^3)}\quad\textrm{and}\quad \norm{\bar{\boldu}}_{L^\infty_t H^{s-1}_x},\ \norm{\tilde{\boldu}}_{L^\infty_t H^{s-1}_x} \leq \delta_{\mathrm{fluid}},$\\[2mm]
\item[(iii)] The microscopic and the fluid parts of the source term $\bepsS$ are controlled in the $\mathcal{H}^s_\eps$ norm as
$$\norm{\mathbf{\pi_L^\bot}\pa{\bepsS}}_{\mathcal{H}^s_\eps} = \mathcal{O}\pa{\frac{\delta_{\rm fluid}}{\eps}} \quad\mbox{ and }\quad \norm{\mathbf{\pi_L}\pa{\bepsS}}_{\mathcal{H}^s_\eps} = \mathcal{O}\pa{\delta_{\rm fluid}}.$$
\end{itemize}
Then, there exist $s_0\in\N^*$, $\bar{\delta}_{\mathrm{fluid}} >0$ and~${\eps_0\in (0,1]}$ such that, for any integer $s \geq s_0$,
there exists an explicit $\delta_\bb >0$ such that, for all $\eps\in (0,\eps_0]$, for all $\delta_{\mathrm{fluid}}\in [0,\bar{\delta}_{\mathrm{fluid}}]$ and for any initial datum $\boldf^\init$ in~$\sobolevTR{s}$ with
\begin{equation*}
\norm{\boldf^\init}_{\mathcal{H}^s_\eps}\leq \delta_\bb,\qquad \norm{\pi_\bepsT(\boldf^\init)}_{\hilbertxv} = \mathcal{O}(\delta_{\mathrm{fluid}}),
\end{equation*}
there exists a unique $\boldf\in C^0\pa{\R_+;\sobolevTR{s}}$ such that $\boldF^\eps=\bepsM+\eps\boldf$ is the unique weak solution of the Boltzmann multi-species equation $\eqref{perturbed BE}$. Moreover, if~${\boldF^{\eps,\init}=\mathbf{M}^{\eps,\init}+\eps\boldf^\init\geq 0}$, then $\boldF^\eps(t,x,v)\geq 0$ almost everywhere on $\R_+\times\T^3\times\R^3$. Finally, for any time $t\geq 0$, $\boldF^\eps$ satisfies the stability property
\begin{equation*}
\norm{\boldF^\eps-\bepsM}_{\mathcal{H}^s_\eps}\leq \eps \delta_\bb, \quad \forall\eps\in (0,\eps_0].
\end{equation*}
\end{theorem}
\bigskip

\section{Perturbative Cauchy theory for the Boltzmann multi-species equation}\label{sec:perturbed BE}

\noindent In this section we establish a Cauchy theory for the Boltzmann multi-species equation $\eqref{perturbed BE}$, perturbed around the local Maxwellian state \eqref{Local Maxwellian eps}--\eqref{Meps fluid quantities}. In the first part we present the hypocoercive properties satisfied by the Boltzmann operators $\bepsL$ and $\boldQ$, which allow to connect our analysis to the formalism used in \cite{MouNeu,Bri1}. The proofs of these lemmata are postponed to Section \ref{sec:technical proofs} for clarity purposes. In the second part we prove Theorem \ref{theo:Cauchy BE}: 
we 
derive suitable uniform (in $\eps$) bounds for the problematic source term $\bepsS$; then list the \textit{a priori} energy estimates that we are able to recover for any solution $\boldf$ of $\eqref{perturbed BE}$; we at last notice that our estimates actually coincide (up to lower orders in $\eps$ or in $\delta_\ms$) to the ones obtained in \cite{Bri1}. 

\subsection{Hypocoercive formalism}\label{Hypocoercivity} 

The initial step consists in verifying some structural properties on $\bepsL$ and $\boldQ$, necessary to recover the basic \textit{a priori} energy estimates that will be presented in the following part. We shall see that the extension of the methods in \cite{Bri1} is absolutely nontrivial and many new issues come up in the analysis. For this reason, we choose to keep the same order of presentation of the technical assumptions adopted in \cite{Bri1}. This will allow us to enlighten the similarities and especially the main differences between our strategies. 

\smallskip
Before stating our lemmata, we here provide a brief description of the main features of the linearized Boltzmann operator $\bepsL$. We recall that $\bepsL$ can be written under the form
$$\bepsL=\bepsK-\bepsnu,$$
where $\bepsK=(\epsK_1,\ldots,\epsK_N)$ is defined, for any $1\leq i\leq N$, by
\begin{equation}\label{OperatorK}
\epsK_i(\boldf)(v)=\sum_{j=1}^N\int_{\R^3\times\Sf}B_{ij}(|v-v_*|,\cos\vartheta)\Big(\Mip \fps_j+\Mjps \fp_i-\epsM_i \fs_j\Big)\dd v_* \dd\sigma,
\end{equation}
and $\bepsnu=(\epsnu_1,\ldots ,\epsnu_N)$ acts like a multiplicative operator, namely, for any $1\leq i\leq N$,
$$\epsnu_i(\boldf)(v)=\sum_{j=1}^N \epsnu_{ij}(v)f_i(v),$$
\begin{equation}\label{CollisionFrequency}
\epsnu_{ij}(v)=\int_{\R^3\times\Sf} B_{ij}\big(|v-v_*|,\cos\vartheta\big)\epsM_j(v_*)\dd v_* \dd\sigma,\quad \forall v\in\R^3.
\end{equation}

\begin{remark}
Note that throughout this initial presentation we shall suppose that the solution $(\boldc,\boldu)$ of the Maxwell-Stefan system we consider is always smooth enough to perform all the estimates, as stated in Theorem \ref{theo:Cauchy MS}. The required regularity will be fixed afterwards and will allow to apply all the results obtained in this section. As such, the quantities $\min_i c_i >0$, $C_0=\scalprod{\boldc,\mathbf{1}}$, $\norm{\boldc}_{L^\infty_t \spacex{s}}$ and $\norm{\boldu}_{L^\infty_t H_x^s}$ are considered as constants and we only keep their track when the explicit order of magnitude $\mathcal{O}(\delta_\ms)$ is needed in order to close the energy estimates at the kinetic level. Also, each time a constant is written, it is independent of $\eps$. This is only done for a sake of simplicity and clarity, but nevertheless each of these constants will be computed explicitly inside the proofs.
\end{remark}

We begin with a lemma which establishes some general controls on $\bepsnu$ and $\bepsL$ (stating in particular the coercivity of $\bepsnu$). 

\begin{lemma}[Estimates on $\bepsnu$ and $\bepsL$] \label{lemma:Properties 1-2}
For any $\eps \in (0,1]$, the linearized operator $\bepsL$ and its multiplicative part $\bepsnu$ satisfy the following explicit bounds. There exist some explicit positive constants $C^{\NuNu}_1$, $C^{\NuNu}_2$ and $C^\LL_1$ such that, for any $\boldf,\boldg\in \hilbertR$, we have

\begin{gather}
C^{\NuNu}_1\norm{\boldf}_{\hilbertv}^2 \leq C^{\NuNu}_1\norm{\boldf}_{\spacev}^2\leq \scalprod{\bepsnu(\boldf),\boldf}_{\hilbertv}\leq C^{\NuNu}_2\norm{\boldf}_{\spacev}^2,  \label{Property1} \\[8mm]   
\abs{\scalprod{\bepsL(\boldf),\boldg}_{\hilbertv}} \leq C^\LL_1 \norm{\boldf}_{\spacev}\norm{\boldg}_{\spacev}.   \label{Property2} 
\end{gather}
\end{lemma}

\smallskip
The constant $C^{\NuNu}_1$ in $\eqref{Property1}$ translates into a coercivity property for $\bepsnu$. Our next result shows that $\bepsnu$ also exhibits a defect of coercivity along its $v$-derivatives.

\begin{lemma}[Defect of coercivity for $\multideriv \bepsnu$]\label{LemmaNu}
For any $s\in\N^*$, and for all multi-indices $\alpha, \beta$ such that $|\alpha|+|\beta|=s$ and $|\beta|\geq 1$, there exist some positive constants $C^\NuNu_k$, with~$k\in\br{3,\ldots, 7}$, such that, for any $\boldf\in\sobolevTR{s}$,
\begin{multline}\label{Property3}
\scalprod{\multideriv \bepsnu(\boldf),\multideriv\boldf}_{\hilbertxv}\geq 
\\[3mm]    \pab{C^{\NuNu}_3-\eps\mathds{1}_{\br{|\alpha|\geq1}}C^{\NuNu}_4}\norm{\multideriv \boldf}_{\spacexv}^2  \qquad\qquad
\\[3mm]-\pab{C^{\NuNu}_5+\eps\mathds{1}_{\br{|\alpha|\geq1}}C^\NuNu_6}\norm{\boldf}_{\sobolevxv{s-1}}^2 \qquad\ \ 
\\[2mm]   - \eps\mathds{1}_{\br{|\alpha|\geq1}}C^{\NuNu}_7\sum_{0<|\alpha^\prime|+|\beta^\prime|\leq s-1}\norm{\partial^{\beta^\prime}_v\partial^{\alpha^\prime}_x \boldf}_{\spacexv}^2.
\end{multline}
\end{lemma}

We here notice for the first time the main difference with the mono-species case treated in \cite{Bri1}, where the Boltzmann equation is perturbed around a global equilibrium. Since the Maxwellian $\bepsM$ we consider depends on the space variable, we see the appearance of several new terms in the estimates, coming precisely from the $x$-derivatives of $\bepsM$. Nevertheless, each of these correction terms is at a lower order in $\eps$ and they are thus likely to be controlled by the main terms, which correspond to the ones derived in \cite{Bri1}. 

\smallskip
Let us now turn to the analysis of $\bepsK$. Following again the strategy of Mouhot and Neumann \cite{MouNeu}, we want to prove that this operator exhibits a regularizing behaviour whenever we consider some $v$-derivatives. More precisely, we shall see that $\multideriv \bepsK$ is controlled by a main term only depending on lower derivatives of $\boldf$ plus a correction term~(which can be made arbitrary small) involving the same order of derivatives $\multideriv \boldf$.  We can think of this property as the counterpart of the defect of coercivity shown previously on $\multideriv \bepsnu$, in the sense that the correction term appearing in the bound of $\multideriv \bepsK$ will be controlled by the negative contribution coming from $\multideriv \bepsnu$. 


Compared to mono-species case, we shall dive deeper into the structure of the operator, encountering two major issues. 
The first issue is that we 
need to recover a full explicit expression of $\bepsK$ in a kernel form, which is not just a simple extension of the mono-species case. This is done in Appendix \ref{App:Carleman} by deriving a precise Carleman representation of $\bepsK$, following the work of \cite{BriDau}. 
The second issue arises because the kernel operator strongly depends on $|v-v_*|$ and we need to properly treat the $v$-derivatives in the small region where the relative velocity is close to zero. 
This last step will in particular involve dealing with the different decay rates of the global equilibria $\mu_i$, a strategy developed in \cite{BonBouBriGre}.

\begin{lemma}[Mixing properties in velocity for $\multideriv \bepsK$]\label{LemmaK}
Let $s\in\N^*$ and consider two multi-indices $\alpha, \beta$ such that $|\alpha|+|\beta|=s$ and $|\beta|\geq 1$. Then, for any $\xi \in (0,1)$, there exist two constants $C^\KK_1(\xi)$, $C^\KK_2 >0$ such that, for all $\boldf\in \sobolevxv{s}$,
\begin{equation}\label{Property4}  
 \scalprod{\multideriv \bepsK(\boldf),\multideriv \boldf}_{\hilbertxv} \leq C^\KK_1(\xi)\norm{\boldf}_{\sobolevxv{s-1}}^2+\xi C^\KK_2 \norm{\multideriv \boldf}_{\hilbertxv}^2.
\end{equation}
The constants $C^\KK_1(\xi)$ and $C^\KK_2$ are explicit, and in particular $C^\KK_2$ does not depend on the parameter $\xi$.
\end{lemma}

The next crucial step in establishing a solid hypocoercivity framework is to determine whether $\bepsL$ possesses a spectral gap in the space of interest $\hilbertR$. 
The main issue is linked to the fact that $\bepsL$ is not self-adjoint in our setting and thus all the tools from classical spectral theory, as well as methods from \cite{DauJunMouZam, BriDau}, cannot be applied directly. Nevertheless, a partial answer to this problem has been already provided in \cite{BonBouBriGre}, where we were able to recover a quantitative upper bound for the Dirichlet form of the linearized operator $\bepsL$ in the space $\hilbertR$. We recall that the idea is to look for a penalization of type $\bepsL=\boldL+(\bepsL-\boldL)$, where $\boldL$ is the Boltzmann operator linearized around the global equilibrium $\boldmu$. 
In this way, the Dirichlet form of $\bepsL$ can be upper bounded by the usual spectral gap plus two correction terms of order $\eps$. The first term gives the distance between the non-equilibrium Maxwellian $\bepsM$ (of the linearization) and the global equilibrium $\boldmu$, while the second one translates the fact that the spectral projection is taken with respect to the space of equilibria of $\boldL$, which does not correspond to the space of equilibria of $\bepsL$. We present an extended and slightly different version of the estimate obtained in \cite{BonBouBriGre}, since our study requires more control on the term accounting for the projection onto $\Ker \boldL$ (say a control of order $\eps^2$). The reason is that we need to handle the factor $1/\eps^2$ in front of $\bepsL$. 
At last, we stress the fact that it is also crucial at this point to keep track of the quantity $\delta_\ms$ that uniformly bounds the $L^\infty_x$ norms of $\tilde{\boldc}$ and $\boldu$. Indeed, in what follows we shall eventually need the freedom in the choice of $\delta_\ms$, in order to take it small enough to close the \textit{a priori} estimates and recover the correct negative return in the Sobolev norm $\norm{\cdot}_{\mathcal{H}^s_\eps}$.

\begin{lemma}[Local coercivity of $\bepsL$]\label{LemmaL}
Let the collision kernels $B_{ij}$ satisfy assumptions~$\mathrm{(H1)}$--$\mathrm{(H2)}$--$\mathrm{(H3)}$--$\mathrm{(H4)}$, and let $\lambda_{\boldL}>0$ be the spectral gap in $\hilbertR$ of the operator~$\boldL$. There exists an explicit constant $C^\LL_2 >0$ such that, for all $\eps \in (0,1]$ and for any~$\eta_1 >0$, the operator $\bepsL$ satisfies the following estimate. For any $\boldf\in\hilbertR$
\begin{multline}\label{Property5} 
\scalprod{\bepsL(\boldf), \boldf}_{\hilbertv}\leq  -\pabb{\lambda_{\boldL}-(\eps +\eta_1) C^\LL_2}\norm{\boldf-\pi_{\boldL}(\boldf)}_{\spacev}^2 
\\[2mm]    + \eps^2 \delta_\ms\frac{C^\LL_2}{\eta_1}\norm{\pi_{\boldL}(\boldf)}_{\spacev}^2,\hspace{1cm}
\end{multline}
where $\pi_{\boldL}$ is the orthogonal projection onto $\Ker\boldL$ in $\hilbertR$, and we recall that $\delta_\ms >0$ can be chosen as small as desired from Theorem \ref{theo:Cauchy MS}.
\end{lemma}

We now turn to the study of the full nonlinear Boltzmann operator $\boldQ$. 


\begin{lemma}[Orthogonality to $\Ker\boldL$ and general controls on $\boldQ$]\label{LemmaQ}
The linear operator $\bepsL$ and the bilinear operator $\boldQ$ are othogonal to the kernel of $\boldL$, namely
\begin{equation}\label{Property6}
\pi_{\boldL}\pab{\bepsL(\boldf)}=\pi_{\boldL}\pab{\boldQ(\boldg,\boldh)}=0,\quad\forall \boldf,\boldg,\boldh\in\hilbertR.
\end{equation}
Moreover, $\boldQ$ satisfies the following estimate. For any $s\in\N$ and for all multi-indices~$\alpha, \beta$ such that $|\alpha|+|\beta|=s$, there exist two nonnegative functionals $\mathcal{G}_x^s$ and $\mathcal{G}_{x,v}^s$ satisfying~${\mathcal{G}_x^{s+1}\leq \mathcal{G}_x^s}$, ${\mathcal{G}_{x,v}^{s+1}\leq \mathcal{G}_{x,v}^s}$, and such that
\begin{equation}\label{Property7}
\abs{\big\langle \multideriv \boldQ(\boldg,\boldh),\boldf\big\rangle_{\hilbertxv}}\leq \left\{\begin{array}{cc}
\mathcal{G}_x^s(\boldg,\boldh)\norm{\boldf}_{\spacexv} & \textrm{if }\ |\beta|=0,\\[7mm]
\mathcal{G}_{x,v}^s(\boldg,\boldh)\norm{\boldf}_{\spacexv} & \textrm{if }\ |\beta|\geq 1.
\end{array}\right.
\end{equation}
In particular, there exists $s_0\in\N^*$ such that, for any integer $s \geq s_0$, there exists an explicit constant $C^\QQ_s>0$ verifying
\begin{equation}\label{Property8}
\begin{split}
& \mathcal{G}_x^s(\boldg,\boldh) \leq  C^\QQ_s\pa{\norm{\boldg}_{H_x^s\hilbertv}\norm{\boldh}_{H^s_x L^2_v\pab{\langle v\rangle^{\frac{\gamma}{2}}\boldsymbol\mu^{-\frac{1}{2}}}}+\norm{\boldh}_{H_x^s\hilbertv}\norm{\boldg}_{H^s_x L^2_v\pab{\langle v\rangle^{\frac{\gamma}{2}}\boldsymbol\mu^{-\frac{1}{2}}}}},\\[7mm]
& \mathcal{G}_{x,v}^s(\boldg,\boldh)\leq  C^\QQ_s\pa{\norm{\boldg}_{\sobolevxv{s}}\norm{\boldh}_{H^s_{x,v}\pab{\langle v\rangle^{\frac{\gamma}{2}}\boldsymbol\mu^{-\frac{1}{2}}}}+\norm{\boldh}_{\sobolevxv{s}}\norm{\boldg}_{H^s_{x,v}\pab{\langle v\rangle^{\frac{\gamma}{2}}\boldsymbol\mu^{-\frac{1}{2}}}}}.
\end{split}
\end{equation}
\end{lemma}

This last lemma completes the investigation of the hypocoercivity properties satisfied by our kinetic model. We can finally step forward to the actual proof of Theorem \ref{theo:Cauchy BE}.

\subsection{Proof of Theorem \ref{theo:Cauchy BE}}


For the sake of clarity, we divide our presentation into several steps. We begin by studying the source term $\bepsS$, which constitutes the main novelty of this work. Its correct estimate can in fact be seen as the last feature in providing a satisfactory hypocoercive framework. We then present some basic properties of the macroscopic projector $\pi_\boldL$ which is used in the next step in order to derive the \textit{a priori} energy estimates satisfied by the perturbations $\boldf$. In the last step, we adapt the computations carried out in \cite{Bri1} to recover the conclusions of~Theorem \ref{theo:Cauchy BE}.

\bigskip
\noindent \textbf{Step 1 -- Estimates on the source term $\bepsS$.} We here provide the study of the source term
$$\mathbf{S}^\eps = \frac{1}{\eps^3}\mathbf{Q}(\bepsM,\bepsM) - \frac{1}{\eps} \partial_t \mathbf{M^\eps} - \frac{1}{\eps^2}v\cdot \nabla_x\mathbf{M^\eps}.$$
This term gives the distance between the Maxwell-Stefan system and the fluid part of the Boltzmann equation, and it represents one of the main differences from the model considered in \cite{Bri1}, and also one of the main drawbacks. In fact, since $\bepsS$ strongly depends on inverse powers of $\eps$ and on $\bepsM$ and its derivatives, it is crucial to determine its leading order of magnitude with respect to $\eps$ and to understand the role played by the macroscopic quantities $\boldc$ and $\boldu$. 

The idea at the basis of our analysis is simple. The Lyapunov functional $\eqref{Sobolev norm}$ we chose is essentially made up of three terms: the $\hilbertxv$ norm which accounts for pure spatial derivatives and for mixed derivatives, and the $\hilbertxv$ scalar product which corresponds to a higher order commutator. In particular, as soon as one derivative in velocity is considered, the weights $\eps$ and $\eps^2$ are used in order to balance out the energy estimates by cancelling the stiffest terms. 

\smallskip
With this idea in mind, we consider three cases, each referring to one of the three sums appearing in the Sobolev norm $\eqref{Sobolev norm}$. Their investigation is similar and essentially based on the separated study of the linear part and of the nonlinear term $\boldQ(\bepsM,\bepsM)$. On the one side, the linear part accounts for an order $\eps^{-1}$ since we can check that $\nabla_x\bepsM=\mathcal{O}(\eps)$. On the other hand, we handle the nonlinear term by exploiting the particular form of $\boldu=\bar{\boldu}+\eps \tilde{\boldu}$. More precisely, we show that~$\bepsM$ is close to a local equilibrium of the mixture (with common macroscopic velocity~$\eps\bar{\boldu}$) up to an order~$\eps^2$, allowing to prove that $\multideriv \boldQ(\bepsM,\bepsM)=\mathcal{O}(\eps^2)$ for any multi-indices $\alpha,\beta\geq 0$. In such a way, we are able to deduce that the general leading order of the source term (and of its derivatives) is actually $\mathcal{O}(\eps^{-1})$, which can be easily handled in the $\mathcal{H}^s_\eps$ norm by the factors $\eps$ and $\eps^2$, as soon as one derivative in $v$ is considered. Unfortunately, since $x$-derivatives are instead controlled by a mere factor of order 1, the $\hilbertxv$ norm of $\partial^\alpha_x\bepsS$ could blow up at a rate of $\mathcal{O}(\eps^{-1})$.

That is why we derive a more precise estimate for the source term, in the case where only $x$-derivatives are taken into account. In particular, the strategy consists in decomposing $\partial^\alpha_x \bepsS$ into a fluid and a microscopic part as $\partial^\alpha_x\bepsS=\pi_{\boldL}(\partial^\alpha_x\bepsS)+\partial^\alpha_x \bepsS^\perp$, in order to show that the stiff problematic terms are actually concentrated in the sole orthogonal component $\partial^\alpha_x \bepsS^\perp$. 
This, together with the fact that $\boldc$ and $\boldu$ satisfy the Maxwell-Stefan system \eqref{MS mass}--\eqref{MS momentum}--\eqref{MS incompressibility}, will be sufficient to prove that actually $\pi_{\boldL}(\partial^\alpha_x\bepsS)=\mathcal{O}(1)$. Thanks to the previous considerations, we shall eventually recover the leading order $\partial^\alpha_x \bepsS^\perp=\mathcal{O}(\eps^{-1})$, but in this case it will no more constitute an issue since the spectral gap of $\partial^\alpha_x\bepsL$ will provide the needed negative return.


\begin{lemma}[Estimates on $\bepsS$]\label{LemmaS}
Let $s\in\N$ and $\boldf\in\sobolevTR{s}$, and consider two multi-indices $\alpha,\beta$ such that $|\alpha|+|\beta|=s$. If $|\beta|\geq 1$, there exist a positive constant~$C_{\alpha,\beta}$, such that, for any $\eps\in(0,1]$ and any $\eta_2 >0$,
\begin{equation}\label{Property9}
\abs{\scalprod{\multideriv\bepsS,\multideriv \boldf}_{\hilbertxv}} \leq  \frac{\delta_\ms^2 C_{\alpha,\beta}}{\eta_2} +\frac{\eta_2}{\eps^2}\norm{\multideriv \boldf}_{\spacexv}^2. 
\end{equation}

\vspace{1mm}
\noindent In particular, if $\alpha_k >0$ for some $k\in\br{1,2,3}$ and $\beta=e_k$, there exist $C_{\alpha, k}>0$ such that, for any $\eps\in(0,1]$ and any $\eta_3 >0$,
\begin{equation}\label{Property10}
\abs{\scalprod{\partial^\alpha_x \bepsS,\partial^{e_k}_v\partial^{\alpha-e_k}_x \boldf}_{\hilbertxv}} \leq  \frac{\delta_\ms^2 C_{\alpha, k}}{\eps \eta_3} +\frac{\eta_3}{\eps}\norm{\partial^{e_k}_v\partial^{\alpha-e_k}_x \boldf}_{\spacexv}^2. 
\end{equation}

\vspace{1mm}
\noindent Finally, if $|\beta|=0$, we ask for a stronger control on the projection $\pi_\boldL(\bepsS)$. There exist a positive constant $C_\alpha$ such that, for any $\eps\in(0,1]$ and any $\eta_4,\eta_5>0$,
\begin{equation}\label{Property11}
\abs{\scalprod{\partial^{\alpha}_x\bepsS,\partial^\alpha_x \boldf}_{\hilbertxv}} \leq \frac{\delta_\ms^2 C_\alpha}{\eta_4 \eta_5} +\eta_4 \norm{\pi_{\boldL}(\partial^\alpha_x\boldf)}_{\hilbertxv}^2+ \frac{\eta_5}{\eps^2} \norm{\partial^\alpha_x\boldf^\perp}_{\spacexv}^2. 
\end{equation}
The constants $C_{\alpha,\beta}$, $C_{\alpha,k}$ and $C_\alpha$ are explicit and only depend on the physical parameters of the problem, and on polynomials in the variables $\norm{\boldc}_{L^\infty_t \spacex{|\alpha|+4}}$ and $\norm{\boldu}_{L^\infty_t H^{|\alpha|+4}_x}$. In particular, they are independent of $\eps$.
\end{lemma}

\medskip
Again, the proof is left to Section \ref{sec:technical proofs}.
\bigskip


\noindent \textbf{Step 2 -- Properties of the fluid projection $\pi_\boldL$.} The derivation of the \textit{a priori} energy estimates finally requires a deeper understanding of the properties of the projectors $\pi_\bepsT$ and $\pi_\boldL$. We present in this step three lemmata which are preliminary to our following study, the main one being a result establishing a fundamental Poincar\'e-type inequality satisfied by $\pi_\boldL$.


\begin{lemma}\label{lem:T eps}
For any $\eps >0$, the operator $\bepsT=\frac{1}{\eps^2}\boldL - \frac{1}{\eps}v\cdot\nabla_x$, acting on $\sobolevTR{1}$, satisfies $\Ker\bepsT= \Ker\boldL\ \cap\ \Ker \pa{v\cdot\nabla_x}$. Moreover, the associated projection $\pi_\bepsT$ explicitly writes
\begin{equation}\label{Relation T L}
\pi_\bepsT(\boldf) = \int_{\T^3} \pi_\boldL(\boldf)\dd x,
\end{equation}
for any $\boldf\in\sobolevTR{1}$.
\end{lemma}

The proof of this lemma is very simple and is therefore omitted. It can be found for example as an initial remark in \cite[Section 4]{BriDau}.

\bigskip
Next, we present a very useful property of the projection $\pi_\boldL$, which is intensively used in what follows: the equivalence of the $\hilbertv$ and $\spacev$ norms on the space $\Ker\boldL$.
 
\begin{lemma}\label{lem:Equivalence}
There exists a positive explicit constant $C_\pi$ such that
\begin{equation}\label{Norm equivalence}
\norm{\pi_\boldL(\boldf)}_{\spacev}^2 \leq C_\pi\norm{\pi_\boldL(\boldf)}_{\hilbertv}^2,
\end{equation}
for any $\boldf\in\hilbertR$.
\end{lemma}

\begin{proof}[Proof of Lemma \ref{lem:Equivalence}]
Recalling the explicit expression of $\pi_\boldL(\boldf)$ from $\eqref{Projection}$ by means of the orthonormal basis~$\pab{\pmb{\phi}^{(k)}}_{1\leq k\leq N+4}$ in $\hilbertR$, we can successively write
\begin{equation*}
\begin{split}
\|\pi_\boldL (\boldf) & \|_{\spacev}^2
\\[2mm]  & = \sum_{i=1}^N \int_{\R^3}\abs{\sum_{k=1}^{N+4}\big\langle\boldf,\pmb{\phi}^{(k)}\big\rangle_{\hilbertv}\phi_i^{(k)}(v)}^2 \langle v\rangle^\gamma \mu_i^{-1}\dd v
\\[4mm]    & = \sum_{k,\ell=1}^{N+4} \big\langle\boldf,\pmb{\phi}^{(k)}\big\rangle_{\hilbertv} \big\langle\boldf,\pmb{\phi}^{(\ell)}\big\rangle_{\hilbertv} \sum_{i=1}^N\int_{\R^3}\phi_i^{(k)} \phi_i^{(\ell)} \langle v\rangle^\gamma \mu_i^{-1}\dd v
\\[4mm]    & = \sum_{k,\ell=1}^{N+4} \big\langle\boldf,\pmb{\phi}^{(k)}\big\rangle_{\hilbertv} \big\langle\boldf,\pmb{\phi}^{(\ell)}\big\rangle_{\hilbertv} \big\langle \pmb{\phi}^{(k)},\pmb{\phi}^{(\ell)}\big\rangle_{\spacev}
\end{split}
\end{equation*}
\begin{equation*}
\begin{split}
   & \leq \frac{1}{2}\max_{1\leq k,\ell \leq N+4} \abs{\big\langle \pmb{\phi}^{(k)},\pmb{\phi}^{(\ell)}\big\rangle_{\spacev}}
\\[2mm]       &\hspace{3.3cm}\times\sum_{k,\ell=1}^{N+4} \pa{\big\langle\boldf,\pmb{\phi}^{(k)}\big\rangle_{\hilbertv}^2 + \big\langle\boldf,\pmb{\phi}^{(\ell)}\big\rangle_{\hilbertv}^2}
\\[4mm]  & = (N+4)\max_{1\leq k,\ell \leq N+4} \abs{\big\langle \pmb{\phi}^{(k)},\pmb{\phi}^{(\ell)}\big\rangle_{\spacev}} \sum_{k=1}^{N+4} \big\langle\boldf,\pmb{\phi}^{(k)}\big\rangle_{\hilbertv}^2
\\[4mm]    & = (N+4) \max_{1\leq k,\ell \leq N+4} \abs{\big\langle \pmb{\phi}^{(k)},\pmb{\phi}^{(\ell)}\big\rangle_{\spacev}} \norm{\pi_\boldL(\boldf)}_{\hilbertv}^2,
\end{split}
\end{equation*}
thanks to Parseval's identity.

\smallskip
In particular, the specific form of the elements of the orthonormal basis allows to easily show that $\pmb{\phi}^{(k)}\in \spaceR$ for any $1\leq k\leq N+4$. Therefore, the maximum of the scalar products is a bounded quantity, and we can finally choose
$$C_\pi = (N+4)\max_{1\leq k,\ell \leq N+4} \abs{\big\langle \pmb{\phi}^{(k)},\pmb{\phi}^{(\ell)}\big\rangle_{\spacev}}$$
to infer the validity of $\eqref{Norm equivalence}$.
\end{proof}

To conclude this preliminary step, we establish a Poincar\'e-type inequality for $\pi_\boldL$ which represents a tool of crucial importance in order to apply the hypocoercive strategy of \cite{MouNeu,Bri1}. Note however that, as opposed to the works in the mono-species setting, where the solution of the Boltzmann equation is perturbed around a global equilibrium, the linearization around $\bepsM$ we perform here does not give access to a perfect Poincar\'e inequality. Indeed, the projection part $\pi_\boldL(\boldf)$ does not have zero mean on the torus because of the presence of the source term $\bepsS$ in $\eqref{perturbed BE}$. We evade this obstacle by showing that a Poincar\'e inequality can still be recovered at a lower order, allowing the appearance of a small correction of order $\mathcal{O}(\delta_\ms^2)$.

\begin{lemma}\label{lem:Poincare}
Let us denote with $C_{\T^3}$ the Poincar\'e constant on the torus, and let $\delta\ms >0$. Consider a solution $\boldf\in\sobolevTR{1}$ of the perturbed Boltzmann equation $\eqref{perturbed BE}$, satisfying initially
$$\norm{\pi_\bepsT(\boldf^\init)}_{\hilbertxv}\leq C\delta_\ms,$$
for some constant $C>0$ independent of the parameters $\eps$ and $\delta_\ms$. There exists a positive constant $C^\TT$ such that
\begin{equation}\label{Poincare inequality}
\norm{\pi_\boldL(\boldf)}_{\hilbertxv}^2 \leq 2 C_{\T^3}\norm{\nabla_x \boldf}_{\hilbertxv}^2 + \delta_\ms^2 C^\TT.
\end{equation}
In particular, $C^\TT$ is explicit and does not depend on the parameter $\eps$.
\end{lemma}

\begin{proof}[Proof of Lemma \ref{lem:Poincare}]
Pick a solution $\boldf\in\sobolevTR{1}$ of $\eqref{perturbed BE}$. Recalling the relation $\eqref{Relation T L}$ between the projectors, we observe that we can express $\pi_\boldL(\boldf)$ in terms of $\pi_\bepsT(\boldf)$. Applying Poincar\'e inequality, we first obtain
\begin{equation}\label{Poincare start}
\begin{split}
\norm{\pi_\boldL(\boldf)}_{\hilbertxv}^2  & \leq 2 \norm{\pi_\boldL(\boldf) - \frac{1}{\abs{\T^3}}\pi_\bepsT(\boldf)}_{\hilbertxv}^2 + \frac{2}{\abs{\T^3}^2}\norm{\pi_\bepsT(\boldf)}_{\hilbertxv}^2
\\[4mm]   & \leq  2 C_{\T^3} \norm{\nabla_x \pi_\boldL(\boldf)}_{\hilbertxv}^2 + \frac{2}{\abs{\T^3}^2}\norm{\pi_\bepsT(\boldf)}_{\hilbertxv}^2.
\end{split}
\end{equation}
Note that for the models in \cite{MouNeu,Bri1}, the authors are able to prove that $\pi_\bepsT(\boldf)=0$, starting from an initial datum satisfying $\pi_\bepsT(\boldf^\init)=0$. This is a peculiarity of the Boltzmann equation, which naturally preserves the action of the projection $\pi_\bepsT$ on it. Clearly, in our case we do not have the same property, and in fact our aim is to show that actually the second term accounts for an order $\mathcal{O}(\delta_\ms^2)$. To do this, we need to recover an equation for $\pi_\bepsT(\boldf)$ from $\eqref{perturbed BE}$.

\smallskip
Using Lemma \ref{lem:T eps}, thanks to the linearity of $\pi_\bepsT$ and to the orthogonality of $\bepsL$ and $\boldQ$ to $\Ker\boldL$ given by $\eqref{Property6}$, one can see that the transport terms and the Boltzmann operators cancel out, so that we are left with the identity, holding a.e. on~${\R_+\times\T^3\times\R^3}$,
\begin{equation*}
\partial_t \pi_\bepsT(\boldf) = -\frac{1}{\eps}\partial_t \pi_\bepsT(\bepsM),
\end{equation*}
where we have also used the fact that the projection operator commutes with the time derivative. In particular, integrating over $[0,t]$ the above identity, we obtain the desired equation for $\pi_\bepsT(\boldf)$, which reads
\begin{equation}\label{Initial identity}
\pi_\bepsT(\boldf) = \pi_\bepsT(\boldf^\init) - \frac{1}{\eps} \pi_\bepsT(\bepsM - \mathbf{M}^{\eps,\init}).
\end{equation}

\smallskip
Consequently, since $\bepsM$ and $\mathbf{M}^{\eps,\init}$ only depend on the macroscopic quantities $(\boldc,\boldu)$ and $(\boldc^\init,\boldu^\init)$, using the relation $\eqref{Relation T L}$ together with the formula $\eqref{Projection explicit}$ of the projection $\pi_\boldL$, it is possible to explicitly compute the value of $\pi_\bepsT(\bepsM-\mathbf{M}^{\eps,\init})$. It is first easy to check that, for any~${1\leq i\leq N}$,
\begin{equation*}
\begin{split}
\int_{\R^3}\epsM_i(t,x,v)\pa{\begin{array}{c}1\\ v\\ |v|^2 \end{array}}\dd v &= \pa{\begin{array}{c} c_i(t,x)\\ \\ \eps c_i(t,x)u_i(t,x)\\ \\\frac{3}{m_i}c_i(t,x)+\eps^2c_i(t,x)|u_i(t,x)|^2 \end{array}},
\end{split}
\end{equation*}
and the equivalent holds for $\mathbf{M}^{\eps,\init}$. It thus follows that
\begin{equation*}
\begin{split}
\pi_\bepsT(\bepsM- & \mathbf{M}^{\eps,\init}) =  \int_{\T^3} \pi_\boldL(\bepsM-\mathbf{M}^{\eps,\init})\dd x
\\[3mm]     = & \sum_{i=1}^N \cro{\frac{\eps}{c_{i,\infty}}\int_{\T^3} \pa{\tilde{c}_i(t,x)-\tilde{c}^\init(x)} \dd x}\mu_i \mathbf{e}^{\text{\tiny (i)}}
\\[3mm]  &  +\frac{\eps v}{\sum_{i=1}^N m_i c_{i,\infty}}\cdot \cro{\sum_{i=1}^N\int_{\T^3} m_i \pab{c_i u_i - c_i^\init u_i^\init}\dd x}\pab{m_i\mu_i}_{1\leq i\leq N} 
\\[3mm]  &  +\frac{\eps^2}{\sqrt{6}\sum_{i=1}^N c_{i,\infty}}\cro{\sum_{i=1}^N\int_{\T^3}m_i  \pab{c_i |u_i|^2 - c_i^\init |u_i^\init|^2}\dd x}\pa{\frac{m_i|v|^2-3}{\sqrt{6}}\mu_i}_{1\leq i\leq N}.
\end{split}
\end{equation*}
In particular, the couple $(\boldc,\boldu)$ is solution of the Maxwell-Stefan system \eqref{MS mass}--\eqref{MS momentum}--\eqref{MS incompressibility}, so that, from \eqref{equimolar vectorial}, one gets that the zeroth-order terms disappear, because both $\tilde{c}_i$ and $\tilde{c}_i^\init$ have zero mean on the torus. Therefore, replacing the above identity in $\eqref{Initial identity}$, we finally recover the expression of $\pi_\bepsT(\boldf)$, which writes explicitly
\begin{equation*}
\begin{split}
\pi_\bepsT(\boldf) = &\ \pi_\bepsT(\boldf^\init) - \frac{v}{\sum_{i=1}^N m_i c_{i,\infty}}\cdot \cro{\sum_{i=1}^N\int_{\T^3} m_i \pab{c_i u_i - c_i^\init u_i^\init} \dd x}\pab{m_i\mu_i}_{1\leq i\leq N} 
\\[3mm]      &\  - \frac{\eps}{\sqrt{6}\sum_{i=1}^N c_{i,\infty}}\cro{\sum_{i=1}^N\int_{\T^3}m_i \pab{c_i |u_i|^2 - c_i^\init |u_i^\init|^2}\dd x}\pa{\frac{m_i|v|^2-3}{\sqrt{6}}\mu_i}_{1\leq i\leq N}.
\end{split}
\end{equation*}
Recalling the uniform controls on $\boldc$ and $\boldu$ given by Theorem \ref{theo:Cauchy MS} and applying the continuous Sobolev embedding of $H^{s/2}_x$ in $L^\infty_x$ for $s>3$, one easily infers the existence of a positive constant $C^\TT$ such that
\begin{equation*}
\begin{split}
\Bigg| \frac{v}{\sum_{i=1}^N m_i c_{i,\infty}} & \cdot \sum_{i=1}^N \int_{\T^3}  m_i \pab{c_i u_i - c_i^\init u_i^\init}\dd x\Bigg|
\\[3mm]      & \leq  |v|\frac{N\abs{\T^3}}{C_0}\pa{\max_{1\leq i\leq N}m_i c_{i,\infty}} 
\\[2mm]     & \hspace{2.7cm} \times \pa{ \norm{\boldc}_{\spacex{s}} \norm{\boldu}_{H^{s-1}_x} + \norm{\boldc^\init}_{\spacex{s}} \norm{\boldu^\init}_{H^{s-1}_x} }
\\[3mm]     & \leq C^\TT |v| \pa{2\norm{\boldc_{\infty}}_{L^2_x\big(\boldc_{\infty}^{-\frac{1}{2}}\big)} + \eps\norm{\tilde{\boldc}}_{\spacex{s}} + \eps\norm{\tilde{\boldc}^\init}_{\spacex{s}}} 
\\[2mm]     & \hspace{4.7cm} \times\pa{ 2 \norm{\bar{\boldu}}_{H^{s-1}_x} + \eps\norm{\tilde{\boldu}}_{H^{s-1}_x} + \eps\norm{\tilde{\boldu}^\init}_{H^{s-1}_x}}
\\[5mm]    & \leq C^\TT \delta_\ms |v| \pab{\sqrt{N\abs{\T^3}} + \eps \delta_\ms} (1+\eps C_\ms),
\end{split}
\end{equation*}
and a similar result can be recovered for the second order terms. As $\eps\leq 1$, we can thus successively increase the value of $C^\TT>0$ to derive the final control
\begin{equation*}
\begin{split}
\norm{\pi_\bepsT(\boldf)}_{\hilbertxv}^2 & \leq 2\norm{\pi_\bepsT(\boldf^\init)}_{\hilbertxv}^2 + C^\TT \delta_\ms^2 \norm{(1+|v|+|v|^2)\boldmu}_{\hilbertv}^2
\\[3mm]     & \leq C^\TT\delta_\ms^2,
\end{split}
\end{equation*}
since the norm on the right-hand side is clearly finite, and we have also used the hypothesis on the initial datum $\boldf^\init$.

Moreover, since $\pi_\boldL$ commutes with $x$-derivatives, we can use the unique orthogonal writing~${\nabla_x \boldf=\pi_\boldL(\nabla_x \boldf) + \nabla_x \boldf^\perp}$ to get
\begin{equation*}
\norm{\nabla_x \pi_\boldL(\boldf)}_{\hilbertxv}^2 \leq \norm{\nabla_x \boldf}_{\hilbertxv}^2.
\end{equation*}
Plugging both last inequalities into $\eqref{Poincare start}$ and redefining accordingly $C^\TT$ allow to end the proof.

\end{proof}

\noindent \textbf{Step 3 -- \textit{A priori} energy estimates.} It is crucial at this point to show the strict similarities between the \textit{a priori} estimates that we obtain here, with the one derived in \cite[Section 3]{Bri1}. Indeed, it is enough to prove this analogy in order to deduce that the exact same computations of \cite{Bri1} apply in our case, leading to the expected results of existence and uniqueness. To do this, we write down all the estimates for the terms appearing in the modified Sobolev norm $\eqref{Sobolev norm}$, whose structure has been fixed at the beginning of this section. 
We list all the estimates, separating inside square brackets the extra terms that appear in our study. In particular, the extra terms involving the norms of $\boldf$ and its derivatives are estimated either by a factor of order $\eps$ or by a factor of order $\delta_\ms$, both of which can be taken as small as desired in order to close the estimates as in \cite{Bri1}. Note however that we cannot recover the exponential decay in time obtained in \cite{MouNeu,Bri1}, as the presence of the source term strongly influences the shape of the energy estimates. Despite this inconvenience, we mention that our analysis is not the sharpest possible, and we guess that an exponential decay rate could probably be recovered in the particular case when $\bar{u}=0$. We shall give a brief discussion on the topic later on. 



\smallskip
For the sake of simplicity, we shall use the notations $C^{(k)}$ and $\tilde{C}_k$ for the technical constants appearing in the following \textit{a priori} estimates. Their explicit values can be found in Appendix~\ref{App:apriori BE}, together with the detailed derivation of all the inequalities. In addition to this, in order to enlighten the computations, we suppose from now on that $\delta_\ms\leq 1$, even if this requirement may not be optimal.

\smallskip
Let $s\in\N^*$ and consider a function $\boldf$ in $\sobolevTR{s}$ which solves the perturbed Boltzmann equation $\eqref{perturbed BE}$ and satisfies initially $\norm{\pi_\bepsT(\boldf^\init)}_{\hilbertxv}=\mathcal{O}(\delta_\ms)$, as in the statement of Theorem \ref{theo:Cauchy BE}. Moreover, let us introduce the standard notation~${\boldf^\perp = \boldf - \pi_\boldL(\boldf)}$ for the part of $\boldf$ that is projected onto~$(\Ker\boldL)^\perp$. 

\medskip
\noindent The estimate for the $\hilbertxv$ norm of $\boldf$ reads
\begin{multline}\label{a priori f}
\frac{\dd}{\dd t}\norm{\boldf}_{\hilbertxv}^2 \leq -\frac{\lambda_\boldL}{\eps^2} \normb{\boldf^\perp}_{\spacexv}^2 + \frac{8(C^\LL_2+1)}{\lambda_\boldL}\mathcal{G}^0_x(\boldf,\boldf)^2
\\[3mm] \hspace{-1.5cm} + \Bigg[ \delta_\ms C^{(1)} \norm{\nabla_x \boldf}_{\hilbertxv}^2 + \tilde{C}\delta_\ms \Bigg]. 
\end{multline}

\noindent The time evolution of the $\hilbertxv$ norm of $\nabla_x \boldf$ is estimated as
\begin{multline}\label{a priori x f}
\frac{\dd}{\dd t}\norm{\nabla_x \boldf}_{\hilbertxv}^2 \leq -\frac{\lambda_\boldL}{\eps^2} \normb{\nabla_x \boldf^\perp}_{\spacexv}^2  +  \frac{8 (2 + 2C^\LL_2 + C^\LL_1 K_x)}{\lambda_\boldL}\mathcal{G}^1_x(\boldf,\boldf)^2
\\[6mm]  + \Bigg[ \delta_\ms C^{(2)} \norm{\nabla_x \boldf}_{\hilbertxv}^2 + \frac{\eps \delta_\ms C^{(3)}}{\eps^2} \normb{\boldf^\perp}_{\spacexv}^2 + \tilde{C}_x\delta_\ms \Bigg].
\end{multline}

\noindent The estimate for the $\hilbertxv$ norm of $\nabla_v\boldf$ reads
\begin{multline}\label{a priori v f}
\frac{\dd}{\dd t}\norm{\nabla_v \boldf}_{\hilbertxv}^2 \leq \frac{K_1}{\eps^2}\normb{\boldf^\perp}_{\spacexv}^2  + \frac{K_{\dd x}}{\eps^2} \norm{\nabla_x\boldf}_{\hilbertxv}^2   
\\[6mm]    -\frac{C^{\NuNu}_3}{\eps^2}\norm{\nabla_v\boldf}_{\spacexv}^2    + \frac{4(3+C^\KK_2)}{C^{\NuNu}_3}\mathcal{G}^1_{x,v}(\boldf,\boldf)^2  + \Bigg[ \frac{\tilde{C}_v \delta_\ms}{\eps^2} \Bigg].
\end{multline}

\noindent The estimate for the commutator reads, fro any $e >0$,
\begin{multline}\label{a priori xv f}
\frac{\dd}{\dd t}\scalprod{\nabla_x \boldf,\nabla_v\boldf}_{\hilbertxv} 
\\[5mm]    \qquad \leq \frac{2 C^\LL_1 e}{\eps^3}\normb{\nabla_x\boldf^\perp}_{\spacexv}^2  - \frac{1}{\eps}\norm{\nabla_x\boldf}_{\hilbertxv}^2 +\frac{C^{(4)}}{e \eps} \norm{\nabla_v\boldf}_{\spacexv}^2  
\\[6mm]   + \frac{2 e}{\eps}\mathcal{G}^1_x(\boldf,\boldf)^2 + \Bigg[ \delta_\ms e C^{(5)} \norm{\nabla_x\boldf}_{\hilbertxv}^2  \hspace{2.5cm} 
\\[4mm]   + \frac{\eps^2 e C^{(6)}}{\eps^3}\normb{\boldf^\perp}_{\spacexv}^2 + \frac{\delta_\ms^2 e \tilde{C}_{x,v}}{\eps}\Bigg].
\end{multline}

\noindent We then consider two multi-indices $\alpha,\beta\in\N^3$ such that $|\alpha|+|\beta|=s$. If $|\beta|=0$, the estimate for higher $x$-derivatives reads
\begin{multline}\label{a priori a f}
\frac{\dd}{\dd t}\norm{\partial^\alpha_x \boldf}_{\hilbertxv}^2\leq  -\frac{\lambda_\boldL}{\eps^2} \norm{\partial^\alpha_x \boldf^\perp}_{\spacexv}^2   + C^{(7)} \mathcal{G}^s_x(\boldf,\boldf)^2
\\[6mm]  \qquad\ \  + \Bigg[ \frac{\eps \delta_\ms K_\alpha}{\eps^2} \sum_{|\alpha'| \leq s-1}\norm{\partial^{\alpha'}_x \boldf^\perp}_{\spacexv}^2 
\\[3mm]   +  \delta_\ms C^{(8)}\norm{\partial^\alpha_x \boldf}_{\hilbertxv}^2  + \tilde{C}_\alpha \delta_\ms \Bigg]. 
\end{multline}

\noindent In the case when we have at least one derivative in $v$, that is when $|\beta|\geq 1$, we obtain the estimate on mixed derivatives
\begin{multline}\label{a priori ab f}
\frac{\dd}{\dd t}\norm{\multideriv \boldf}_{\hilbertxv}^2
\\[5mm]   \hspace{-4.5cm}  \leq  -\frac{C^{\NuNu}_3}{\eps^2}\norm{\multideriv \boldf}_{\spacexv}^2 +\frac{K_{s-1}}{\eps^2}\norm{\boldf}_{\sobolevxv{s-1}}^2  
\\[7mm]    \hspace{2cm} + C^{(9)} \sum_{k,\ \beta_k >0} \norm{\partial^{\beta-\ee_k}_v\partial^{\alpha+\ee_k}_x\boldf}_{\hilbertxv}^2 + \frac{4(5 + C^\KK_2 + C^\NuNu_4)}{C^\NuNu_3} \mathcal{G}^s_{x,v}(\boldf,\boldf)^2 
\\[5mm]   \hspace{-2.5cm} + \Bigg[ \frac{\eps C^\NuNu_7}{\eps^2}\norm{\boldf}_{H^{s-1}_{x,v}\big(\langle v\rangle^{\frac{\gamma}{2}}\boldmu^{-\frac{1}{2}}\big)}^2 + \tilde{C}_{\alpha,\beta} \delta_\ms \Bigg]. 
\end{multline}

\noindent  In the following, we also need to consider the particular case when we replace $\alpha$ with $\alpha-\ee_k$ and $\beta$ with $\ee_k$ in the previous estimate. Inequality $\eqref{a priori ab f}$ becomes
\begin{multline}\label{a priori ab f 2}
\frac{\dd}{\dd t}\norm{\partial^{\ee_k}_v\partial^{\alpha-\ee_k}_x \boldf}_{\hilbertxv}^2
\\[5mm]   \hspace{-2.5cm}  \leq  -\frac{C^{\NuNu}_3}{\eps^2}\norm{\partial^{\ee_k}_v\partial^{\alpha-\ee_k}_x \boldf}_{\spacexv}^2 +\frac{K_{s-1}}{\eps^2}\norm{\boldf}_{\sobolevxv{s-1}}^2  
\\[6mm]    \hspace{1cm} + C^{(9)} \norm{\partial^{\alpha}_x\boldf}_{\hilbertxv}^2 + \frac{4(5 + C^\KK_2 + C^\NuNu_4)}{C^\NuNu_3} \mathcal{G}^s_{x,v}(\boldf,\boldf)^2 
\\[6mm]    + \Bigg[ \frac{\eps C^\NuNu_7}{\eps^2}\norm{\boldf}_{H^{s-1}_{x,v}\big(\langle v\rangle^{\frac{\gamma}{2}}\boldmu^{-\frac{1}{2}}\big)}^2 + \tilde{C}_{\alpha,\beta} \delta_\ms \Bigg]. 
\end{multline}

\noindent Finally, we need to upper bound the commutator for higher derivatives
\begin{multline} \label{a priori ak f}
\frac{\dd}{\dd t}\scalprod{\partial^\alpha_x\boldf,\partial^{\ee_k}_v \partial^{\alpha-\ee_k}_x\boldf}_{\hilbertxv}
\\[4mm]  \hspace{-4.5cm} \leq  \frac{2 C^\LL_1 e}{\eps^3} \norm{\partial^\alpha_x\boldf^\perp}_{\spacexv}^2  - \frac{1}{\eps} \norm{\partial^\alpha_x \boldf}_{\hilbertxv}^2 
\\[5mm]    \hspace{-2.5cm} + \frac{C^{(10)}}{\eps e}\norm{\partial^{\ee_k}_v \partial^{\alpha-\ee_k}_x\boldf}_{\spacexv}^2 + \frac{2 e}{\eps C^\LL_1}\mathcal{G}^s_{x}(\boldf,\boldf)^2
\\[5mm]  \quad \qquad + \Bigg[  \frac{\delta_\ms e K_{\alpha,k}}{\eps}\norm{\boldf}_{H^{s-1}_{x,v}\big(\langle v\rangle^{\frac{\gamma}{2}}\boldmu^{-\frac{1}{2}}\big)}^2  + \delta_\ms e C^{(11)} \norm{\partial^\alpha_x\boldf}_{\hilbertxv}^2 
\\[4mm]   + \frac{\delta_\ms^2 e \tilde{C}_{\alpha,k}}{\eps} - \frac{1}{\eps} \scalprod{\partial^{\alpha - \ee_k}_x(v\cdot\nabla_x \boldf), v_k\partial^\alpha_x \boldf}_{\hilbertxv} \Bigg],
\end{multline}
where the estimate holds for any $e >0$, and we remember that the last term satisfies the fundamental property
\begin{equation}\label{Negative sign}
-\sum_{\substack{|\alpha| = s  \\[0.2mm]  k,\ \alpha_k >0}}\scalprod{\partial^{\alpha - \ee_k}_x(v\cdot\nabla_x \boldf), v_k\partial^\alpha_x \boldf}_{\hilbertxv} = - \sum_{|\alpha'| = s-1} \norm{v\cdot\nabla_x \partial^{\alpha'}_x\boldf}_{\hilbertxv}^2 \leq 0.
\end{equation}

\bigskip
Starting from these inequalities, we are now able to establish the link with \cite{Bri1}. For the sake of clarity, we divide the derivation of the \textit{a priori} energy estimates into two successive results. In the first one, we recover a preliminary upper bound which provides a partial negative return.

\begin{prop}\label{prop:preliminary a priori BE}
There exist $\bar{\delta}_\ms>0$ and $\eps_0\in (0,1]$ such that the following statements hold for any $s\in\N^*$.
\begin{itemize}
\item[(i)] There exist three sets of positive constants $\pab{a^{(s)}_{\alpha}}_{\alpha}^s$, $\pab{b^{(s)}_{\alpha,k}}_{\alpha,k}^s$ and $\pab{d^{(s)}_{\alpha,\beta}}_{\alpha,\beta}^s$ such that, for all $\eps\in (0,\eps_0]$,
\begin{equation}\label{A priori equivalence}
\quad \norm{\cdot}_{\mathcal{H}^s_\eps} \sim  \pa{\norm{\cdot}_{\hilbertxv}^2 + \sum_{|\alpha| \leq s}\norm{\partial^\alpha_x\cdot}_{\hilbertxv}^2 + \eps^2\sum_{\substack{|\alpha|+|\beta|\leq s   \\[0.2mm]   |\beta|\geq 1}}\norm{\multideriv\cdot}_{\hilbertxv}^2}^{1/2}.
\end{equation}
\item[(ii)] There exist four positive constants $K_0^{(s)}$, $K_1^{(s)}$, $K_2^{(s)}$ and $C^{(s)}$ such that, for any values of $\eps\in(0,\eps_0]$ and $\delta_\ms\in [0,\bar{\delta}_\ms]$, if $\boldf\in \sobolevTR{s}$ solves the perturbed Boltzmann equation $\eqref{perturbed BE}$ with initial datum $\boldf^\init$ belonging to $\sobolevTR{s}$ and satisfying ${\norm{\pi_\bepsT(\boldf^\init)}_{\hilbertxv} = \mathcal{O}(\delta_\ms)}$, then, for every time $t\geq 0$,
\begin{multline}\label{A priori estimate BE 1}
\frac{\dd}{\dd t} \norm{\boldf}_{\mathcal{H}^s_\eps}^2 \leq -K_0^{(s)}\pa{\norm{\boldf}_{H^s_{x,v}\big(\langle v\rangle^{\frac{\gamma}{2}}\boldmu^{-\frac{1}{2}}\big)}^2 + \frac{1}{\eps^2} \sum_{|\alpha|\leq s}\norm{\partial^\alpha_x\boldf^\perp}_{\spacexv}^2} 
\\[4mm]    + K_1^{(s)}\mathcal{G}^s_x(\boldf,\boldf)^2 + \eps^2 K_2^{(s)}\mathcal{G}^s_{x,v}(\boldf,\boldf)^2 +C^{(s)} \delta_\ms. 
\end{multline}
\end{itemize}
\end{prop}

\begin{proof}[Proof of Proposition \ref{prop:preliminary a priori BE}]
Let $\boldf\in\sobolevTR{s}$ be any solution of $\eqref{perturbed BE}$ with initial condition satisfying the hypothesis of our statement, and consider a solution $(\boldc,\boldu)\in H^{s+5}_x\times H^{s+4}_x$ of the Maxwell-Stefan system \eqref{MS mass}--\eqref{MS momentum}--\eqref{MS incompressibility}. Moreover, suppose that $\delta_\ms\leq 1$ for simplicity. Under these hypotheses, estimates from \eqref{a priori f} up to \eqref{a priori ak f} hold and can be applied.

Following \cite[Section 5]{Bri1}, we proceed by induction on $s\in\N^*$. For $s=1$, the modified Sobolev norm $\mathcal{H}^1_\eps$ reads, for any $\boldf\in\sobolevTR{1}$,
\begin{multline*}
\norm{\boldf}_{\mathcal{H}^1_\eps}^2 = A\norm{\boldf}_{\hilbertxv}^2 + a\norm{\nabla_x \boldf}_{\hilbertxv}^2 
\\[3mm]   + b\eps\scalprod{\nabla_x\boldf,\nabla_v\boldf}_{\hilbertxv} + d\eps^2\norm{\nabla_v\boldf}_{\hilbertxv}^2.
\end{multline*}
Therefore, we consider the linear combination $A\eqref{a priori f} + a\eqref{a priori x f} + d\eps^2\eqref{a priori v f} + b\eps\eqref{a priori xv f}$. Recalling from Lemma \ref{LemmaQ} that we have the monotone behaviour $\mathcal{G}^0_x \leq \mathcal{G}^1_x$, we can easily find some positive constants $K_1^{(1)}$, $K_2^{(1)}$ and $C^{(1)}$ such that, at first,  
\begin{equation*}
\begin{split}
\frac{\dd}{\dd t} \norm{\boldf}_{\mathcal{H}^1_\eps}^2 \leq &\ \frac{1}{\eps^2}\cro{\eps\pab{a\delta_\ms C^{(3)} + \eps K_1 d + b\eps C^{(6)}} - \frac{A\lambda_\boldL}{2}}\Vert \boldf^\perp\Vert_{\spacexv}^2
\\[3mm]     &\  + \frac{1}{\eps^2} \cro{2C^\LL_1 e b - \frac{a\lambda_\boldL}{2}}\Vert \nabla_x\boldf^\perp\Vert_{\spacexv}^2
\\[3mm]     &\  + \cro{\delta_\ms \pab{A C^{(1)} + a C^{(2)}} + K_{\dd x} d - b \pabb{1 - \delta_\ms e C^{(5)}}} \norm{\nabla_x\boldf}_{\hilbertxv}^2
\\[3mm]     &\  + \cro{\frac{C^{(4)} b}{e} - d C^{\NuNu}_3} \norm{\nabla_v\boldf}_{\spacexv}^2
\\[4mm]     &\   - \frac{A\lambda_\boldL}{2\eps^2} \Vert \boldf^\perp\Vert_{\spacexv}^2 - \frac{a\lambda_\boldL}{2\eps^2} \Vert \nabla_x\boldf^\perp\Vert_{\spacexv}^2
\\[5mm]     &\  + K_1^{(1)} \mathcal{G}_x^1(\boldf,\boldf)^2  + \eps^2 K_2^{(1)} \mathcal{G}_{x,v}^1(\boldf,\boldf)^2 + C^{(1)}\delta_\ms.  
\end{split}
\end{equation*}
If we now choose
\begin{equation*}
\bar{\delta}_\ms \leq \min\br{1,\ \frac{K_{\dd x} d}{A C^{(1)} + a C^{(2)}},\  \frac{1}{2 e C^{(5)}}},\qquad \eps_0 \leq \min\br{1,\  \frac{K_1 d}{aC^{(3)} + K_1 d + b C^{(6)}}},
\end{equation*}
we can repeat the same computations in \cite{Bri1} in order to fix the values of $A$, $a$, $b$, $d$ and~$e$. In particular, the equivalence between the modified norm $\mathcal{H}^1_\eps$ and the standard Sobolev norm~$H^1_x\hilbertv$ immediately follows, and we also recover, because $\eps\in (0,1]$, the estimate
\begin{multline}\label{Estimate for 1}
\frac{\dd}{\dd t} \norm{\boldf}_{\mathcal{H}^1_\eps}^2  \leq - \left(  \Vert \boldf^\perp\Vert_{\spacexv}^2 + \Vert \nabla_x\boldf^\perp\Vert_{\spacexv}^2\right.
\\[2mm]   \hspace{5cm}\left. + \norm{\nabla_v\boldf}_{\spacexv}^2 + \norm{\nabla_x\boldf}_{\hilbertxv}^2 \right)
\\[4mm]    - \frac{A\lambda_\boldL}{2\eps^2} \norm{\boldf^\perp}_{\spacexv}^2 - \frac{a\lambda_\boldL}{2\eps^2}\Vert \nabla_x\boldf^\perp\Vert_{\spacexv}^2 
\\[6mm]    \qquad \quad+ K_1^{(1)} \mathcal{G}_x^1(\boldf,\boldf)^2  + \eps^2 K_2^{(1)} \mathcal{G}_{x,v}^1(\boldf,\boldf)^2 + C^{(1)} \delta_\ms. 
\end{multline}
Note that we have kept the constants $A$ and $a$ in the extra negative contributions coming from $\boldf^\perp$, for the sake of clarity. In particular, we stress the fact that $A$ and $a$ have been fixed independently of $\delta_\ms$ and $\eps$.

To conclude, we observe that thanks to the Poincar\'e inequality $\eqref{Poincare inequality}$ and the equivalence between the $\spacexv$ and $\hilbertxv$ norms on $\Ker\boldL$, the following upper bounds hold
\begin{gather*}
\norm{\boldf}_{\spacexv}^2 \leq C'\pa{\norm{\boldf^\perp}_{\spacexv}^2 + \frac{1}{2}\norm{\nabla_x\boldf}_{\hilbertxv}^2 + \delta_\ms^2},
\\[4mm]   \norm{\nabla_x \boldf}_{\spacexv}^2 \leq \tilde{C}' \pa{\norm{\nabla_x \boldf^\perp}_{\spacexv}^2 + \frac{1}{2}\norm{\nabla_x\boldf}_{\hilbertxv}^2},
\end{gather*}
for some positive constants $C'$ and $\tilde{C}'$ computed from $\eqref{Norm equivalence}$ and $\eqref{Poincare inequality}$, which are independent of the parameters $\delta_\ms$ and $\eps$. By plugging the above inequalities into $\eqref{Estimate for 1}$ and by adding the small extra term $\delta_\ms^2$ to the same one multiplying $C^{(1)}$, we deduce the existence of a positive constant $K_0^{(1)}$ (independent of $\delta_\ms$ and $\eps$) such that
\begin{multline*}
\frac{\dd}{\dd t} \norm{\boldf}_{\mathcal{H}^1_\eps}^2  \leq -K_0^{(1)} \norm{\boldf}_{H^1_{x,v}\big(\langle v\rangle^{\frac{\gamma}{2}}\boldmu^{-\frac{1}{2}}\big)}^2 
\\[4mm]      \qquad\qquad\qquad  - \frac{\lambda_\boldL\min\br{A,a}}{2\eps^2} \pa{\norm{\boldf^\perp}_{\spacexv}^2 + \Vert \nabla_x\boldf^\perp\Vert_{\spacexv}^2}
\\[6mm]   + K_1^{(1)}\mathcal{G}^s_x(\boldf,\boldf)^2 + \eps^2 K_2^{(1)}\mathcal{G}^s_{x,v}(\boldf,\boldf)^2 + C^{(1)}\delta_\ms,  
\end{multline*}
Therefore, redefining $K_0^{(1)}$ small enough to ensure that also $K_0^{(1)}\leq \lambda_\boldL\min\br{A,a}/2$, we finally recover estimate $\eqref{A priori estimate BE 1}$ in the case $s=1$.

\smallskip
For the general case, we suppose that the result is true up to the integer $s-1$, and we prove that it also holds for $s$. The first statement, about the equivalence between the~$\mathcal{H}^s_\eps$ and $H^s_x\hilbertv$ norms, is straightforward. For the second property, regarding the \textit{a priori} estimate, in a similar way to what we have done for $s=1$, we define this time
\begin{eqnarray}
F_s(t) & = & \eps^2 B \sum_{\substack{|\alpha|+|\beta|=s    \\[0.2mm]    |\alpha|\geq 2}} \norm{\multideriv \boldf}_{\hilbertxv}^2 + B' \sum_{\substack{|\alpha|=s   \\[0.2mm]    k,\ \alpha_k >0}} Q_{\alpha, k}(t), \label{Fs}
\\[5mm]    Q_{\alpha,k}(t) & = & a\norm{\partial^\alpha_x\boldf}_{\hilbertxv}^2 + b\eps \scalprod{\partial^{\ee_k}_v \partial^{\alpha-\ee_k}\boldf,\partial^\alpha_x\boldf}_{\hilbertxv} 
\\[3mm]        &&\hspace{5cm} + d\eps^2 \norm{\partial^{\ee_k}_v\partial^{\alpha-\ee_k}_x\boldf}_{\hilbertxv}^2, \nonumber
\end{eqnarray}
and the positive constants have to be fixed as previously. 

\noindent By taking the linear combination $a\eqref{a priori a f} + b \eps \eqref{a priori ak f} + d \eps^2 \eqref{a priori ab f 2}$, we compute
\begin{equation*}
\begin{split}
\frac{\dd}{\dd t} Q_{\alpha, k}(t) \leq &\ \frac{1}{\eps^2}\cro{2 C^\LL_1 e b - \frac{a \lambda_\boldL}{2}} \norm{\partial^\alpha_x\boldf^\perp}_{\spacexv}^2
\\[2mm]     &\ + \cro{\frac{C^{(10)} b}{e} - C^{\NuNu}_3 d}\norm{\partial^{\ee_k}_v\partial^{\alpha-\ee_k}_x\boldf}_{\spacexv}^2
\\[2mm]     &\ + \cro{\delta_\ms C^{(8)} a + \eps^2 C^{(9)} d - b\pabb{1 - \delta_\ms e C^{(11)}}}\norm{\partial^\alpha_x\boldf}_{\hilbertxv}^2
\\[2mm]     &\  - \frac{a \lambda_\boldL}{2\eps^2} \norm{\partial^\alpha_x\boldf^\perp}_{\spacexv}^2 + \frac{a \eps\delta_\ms K_\alpha}{\eps^2} \sum_{|\alpha'| \leq s-1}\norm{\partial^{\alpha'}_x \boldf^\perp}_{\spacexv}^2
\\[2mm]     &\ + K_{s-1} d \norm{\boldf}_{\sobolevxv{s-1}}^2 + \pabb{\delta_\ms e K_{\alpha,k} b + \eps C^\NuNu_7 d }\norm{\boldf}_{H^{s-1}_{x,v}\big(\langle v\rangle^{\frac{\gamma}{2}}\boldmu^{-\frac{1}{2}}\big)}^2
\\[4mm]     &\ + \tilde{K}_1^{(s)}\mathcal{G}^s_x(\boldf,\boldf)^2 + \eps^2 \tilde{K}_2^{(2)}\mathcal{G}^s_{x,v}(\boldf,\boldf)^2 + \tilde{C}^{(s)} \delta_\ms 
\\[5mm]     &\ -b \scalprod{\partial^{\alpha - \ee_k}_x(v\cdot\nabla_x \boldf), v_k\partial^\alpha_x \boldf}_{\hilbertxv}.
\end{split}
\end{equation*}
In particular, similarly to the case $s=1$, we can redefine $\bar{\delta_\ms}$ to satisfy also 
$$\bar{\delta}_\ms \leq \min\br{\frac{C^{(9)} d}{C^{(8)} a},\ \frac{1}{2 e C^{(11)}}},$$ 
so that the first three contributions are the same as in \cite[Section 5]{Bri1}. Therefore, we can set the values of $a$, $b$, $d$ and $e$ in such a way that
\begin{equation}\label{Evolution Q}
\begin{split}
\frac{\dd}{\dd t} Q_{\alpha, k}(t) \leq &\ - K_0^{(s)} \pa{\norm{\partial^\alpha_x\boldf}_{\spacexv}^2 + \norm{\partial^{\ee_k}_v\partial^{\alpha-\ee_k}_x\boldf}_{\spacexv}^2}
\\[3mm]     &\  - \frac{a \lambda_\boldL}{2\eps^2} \norm{\partial^\alpha_x\boldf^\perp}_{\spacexv}^2 + \frac{a \eps\delta_\ms K_\alpha}{\eps^2} \sum_{|\alpha'| \leq s-1}\norm{\partial^{\alpha'}_x \boldf^\perp}_{\spacexv}^2
\\[3mm]     &\ + \tilde{K}_{s-1} \norm{\boldf}_{H^{s-1}_{x,v}\big(\langle v\rangle^{\frac{\gamma}{2}}\boldmu^{-\frac{1}{2}}\big)}^2 + \tilde{K}_1^{(s)}\mathcal{G}^s_x(\boldf,\boldf)^2 + \eps^2 \tilde{K}_2^{(2)}\mathcal{G}^s_{x,v}(\boldf,\boldf)^2 + \tilde{C}^{(s)}\delta_\ms 
\\[5mm]     &\ -b \scalprod{\partial^{\alpha - \ee_k}_x(v\cdot\nabla_x \boldf), v_k\partial^\alpha_x \boldf}_{\hilbertxv},
\end{split}
\end{equation}
where we have used the equivalence between the $\spacexv$ and the $\hilbertxv$ norms on $\Ker\boldL$, and we have also accordingly redefined all the main constants of interest to enlighten the computations. In particular, we stress again the fact that all these constants remain independent of the parameters $\delta_\ms$ and $\eps$.

\smallskip
Now, note that if we sum $\eqref{Evolution Q}$ over $|\alpha|=s$ and $k=1,2,3$, such that~${\alpha_k >0}$, the terms accounting for the scalar products disappear, thanks to the property $\eqref{Negative sign}$.
Therefore, going back to the definition $\eqref{Fs}$ of $F_s(t)$, combining $B' \eqref{Evolution Q}$ with $\eps^2 B \eqref{a priori ab f}$ we obtain
\begin{equation*}
\begin{split}
\frac{\dd}{\dd t}  F_s & (t) \leq  -C^{\NuNu}_3 B\sum_{\substack{|\alpha|+|\beta|=s   \\[0.2mm]    |\beta|\geq 2}} \norm{\multideriv \boldf}_{\spacexv}^2 
 \\[3mm]    &\  + \sum_{\substack{|\alpha|+|\beta|=s   \\[0.2mm]    |\beta|\geq 2}} C^{(9)} B\eps^2 \sum_{k,\ \beta_k >0}\norm{\partial^{\ee_k}_v\partial^{\alpha-\ee_k}_x\boldf}_{\spacexv}^2
\\[3mm]     &\ - B' \tilde{K}_0^{(s)} \sum_{\substack{|\alpha|=s   \\[0.2mm]    k,\ \alpha_k >0}} \pa{\norm{\partial^\alpha_x\boldf}_{\spacexv}^2 + \norm{\partial^{\ee_k}_v\partial^{\alpha-\ee_k}_x\boldf}_{\spacexv}^2}
\\[3mm]     &\  - B' \sum_{\substack{|\alpha|=s   \\[0.2mm]    k,\ \alpha_k >0}} \pa{\frac{a \lambda_\boldL}{2\eps^2} \norm{\partial^\alpha_x\boldf^\perp}_{\spacexv}^2 - \frac{a \eps\delta_\ms K_\alpha}{\eps^2} \sum_{|\alpha'| \leq s-1}\norm{\partial^{\alpha'}_x \boldf^\perp}_{\spacexv}^2}
\\[6mm]     &\ + \pabb{B' \tilde{K}_{s-1} + B \pab{K_{s-1} +\eps C^\NuNu_7}} \norm{\boldf}_{H^{s-1}_{x,v}\big(\langle v\rangle^{\frac{\gamma}{2}}\boldmu^{-\frac{1}{2}}\big)}^2
\\[7mm]     &\ + K_1^{(s)}\mathcal{G}^s_x(\boldf,\boldf)^2 + \eps^2 K_2^{(2)}\mathcal{G}^s_{x,v}(\boldf,\boldf)^2
+ C^{(s)}\delta_\ms^2. 
\end{split}
\end{equation*}
Redefining the constants if necessary, we can then copy the arguments used in \cite[Section~5]{Bri1} to finally obtain the existence of $\eps_0\in (0,1]$ such that, for all $\eps\in (0,\eps_0]$,
\begin{multline*}
\hspace{-0.2cm}\frac{\dd}{\dd t}  F_s (t) \leq \tilde{K}_{s-1} \norm{\boldf}_{H^{s-1}_{x,v}\big(\langle v\rangle^{\frac{\gamma}{2}}\boldmu^{-\frac{1}{2}}\big)}^2 - \pa{\sum_{|\alpha|+|\beta|=s} \norm{\multideriv \boldf}_{\spacexv}^2}
\\[5mm]     \quad - B' \frac{a \lambda_\boldL}{2\eps^2}  \sum_{|\alpha|=s} \norm{\partial^\alpha_x\boldf^\perp}_{\spacexv}^2 + \tilde{B}'\frac{a \eps\delta_\ms K_\alpha}{\eps^2} \sum_{|\alpha'| \leq s-1}\norm{\partial^{\alpha'}_x \boldf^\perp}_{\spacexv}^2
\\[6mm]    \hspace{-3cm}  + K_1^{(s)}\mathcal{G}^s_x(\boldf,\boldf)^2 + \eps^2 K_2^{(2)}\mathcal{G}^s_{x,v}(\boldf,\boldf)^2 + C^{(s)}\delta_\ms. 
\end{multline*}
To conclude, for some positive constants $(K_p)_{p\leq s}$, we only have to consider a linear combination $\sum_{p=1}^s K_p F_p(t)$ and use the induction hypothesis, together with $\eps\leq 1$ and the last requirement
\begin{equation*}
\disp  \bar{\delta}_\ms  \leq \frac{\min_p K_p}{2\tilde{B}' a K_\alpha}.
\end{equation*}
Recalling that the functionals $\mathcal{G}^p_x$ and $\mathcal{G}^p_{x,v}$ are monotonically increasing in $p$, we can finally recover the following estimate, valid for any $\eps\in (0,\eps_0]$ and any $\delta_\ms\in [0,\bar{\delta}_\ms]$,
\begin{multline*}
\frac{\dd}{\dd t}  \pa{\sum_{p=1}^s K_p F_p (t)} 
\\[2mm]   \quad \leq - K_0^{(s)}\pa{\sum_{|\alpha|+|\beta|\leq s} \norm{\multideriv \boldf}_{\spacexv}^2 + \frac{1}{\eps^2} \sum_{|\alpha|\leq s} \norm{\partial^\alpha_x\boldf^\perp}_{\spacexv}^2}
\\[7mm]     + K_1^{(s)}\mathcal{G}^s_x(\boldf,\boldf)^2 + \eps^2 K_2^{(2)}\mathcal{G}^s_{x,v}(\boldf,\boldf)^2 + C^{(s)}\delta_\ms. 
\end{multline*}
which is the expected result.

\end{proof}

\smallskip
Starting from this general preliminary estimate, we can finally prove the result that establishes the uniform \textit{a priori} control on $\boldf$, which in turn ensures the stability of the expansion around $\bepsM$.

\begin{cor}\label{cor:A priori BE}
There exist $s_0\in \N^*$, $\bar{\delta}_\ms>0$ and $\eps_0\in (0,1]$ such that, for any integer~$s \geq s_0$, there exists $\delta_\bb >0$ such that, for any $\eps\in (0,\eps_0]$ and any $\delta_\ms\in [0,\bar{\delta}_\ms]$, if $\boldf\in\sobolevTR{s}$ solves the perturbed Boltzmann equation $\eqref{perturbed BE}$, and satisfies initially
\begin{equation*}
\norm{\boldf^\init}_{\mathcal{H}^s_\eps}\leq \delta_\bb,\qquad \norm{\pi_\bepsT(\boldf^\init)}_{\hilbertxv} = \mathcal{O}(\delta_\ms),
\end{equation*}
then $\norm{\boldf}_{\mathcal{H}^s_\eps}\leq \norm{\boldf^\init}_{\mathcal{H}^s_\eps}$, for all $t\geq 0$.
\end{cor}

\begin{proof}[Proof of Corollary \ref{cor:A priori BE}]
Since we are under the same hypotheses of Proposition \ref{prop:preliminary a priori BE}, we have seen how to fix the values of $\bar{\delta}_\ms >0$ and $\eps_0\in (0,1]$, so that, for any $s\in\N^*$, the \textit{a priori} estimate $\eqref{A priori estimate BE 1}$ holds for any time $t\geq 0$. Note in particular that we can get rid of the negative terms involving the pure $x$-derivatives of $\boldf$, as their presence was needed in the previous result only to ensure the control of the small positive extra terms of~$\mathcal{O}(\eps)$, in order to close the induction procedure. Therefore, for any $s\in\N^*$ and any $t\geq 0$,  the time evolution of the $\mathcal{H}^s_\eps$ norm of $\boldf$ is uniformly controlled in $\eps\in (0,\eps_0]$ and $\delta_\ms\in [0,\bar{\delta}_\ms]$ as
\begin{multline}\label{Stability start}
\frac{\dd}{\dd t} \norm{\boldf}_{\mathcal{H}^s_\eps}^2 \leq -K_0^{(s)}\norm{\boldf}_{H^s_{x,v}\big(\langle v\rangle^{\frac{\gamma}{2}}\boldmu^{-\frac{1}{2}}\big)}^2 + K_1^{(s)}\mathcal{G}^s_x(\boldf,\boldf)^2 + \eps^2 K_2^{(s)}\mathcal{G}^s_{x,v}(\boldf,\boldf)^2  + C^{(s)}\delta_\ms.
\end{multline}
The idea is then to properly bound the functionals $\mathcal{G}^s_x$ and $\mathcal{G}^s_{x,v}$. For this, recall from Lemma \ref{LemmaQ} that we can find an integer $s_0\in\N^*$ for which estimates $\eqref{Property8}$ hold for any integer $s \geq s_0$. Moreover, thanks to the equivalence between the $\mathcal{H}^s_\eps$ norm and the standard Sobolev norm given by $\eqref{A priori equivalence}$, we can infer the existence of two positive constants $C_{\textrm{eq}}$ and~${C_{\textrm{\tiny EQ}}}$ such that
\begin{multline}\label{Sobolev equivalence}
C_{\textrm{eq}}\pa{\norm{\cdot}_{\hilbertxv}^2 + \sum_{|\alpha| \leq s}\norm{\partial^\alpha_x\cdot}_{\hilbertxv}^2 + \eps^2\sum_{\substack{|\alpha|+|\beta|\leq s   \\[0.2mm]   |\beta|\geq 1}}\norm{\multideriv\cdot}_{\hilbertxv}^2} 
\\[5mm]    \leq \norm{\cdot}_{\mathcal{H}^s_\eps}^2 \leq C_{\textrm{\tiny EQ}} \norm{\cdot}_{\sobolevxv{s}}^2.
\end{multline}
Therefore, $\mathcal{G}^s_x$ can be successively estimated as
\begin{equation*}
\begin{split}
\mathcal{G}_x^s(\boldf,\boldf)^2 \leq &\   2 (C^\QQ_s)^2 \norm{\boldf}_{H_x^s\hilbertv}^2 \norm{\boldf}_{H^s_x L^2_v\pab{\langle v\rangle^{\frac{\gamma}{2}}\boldsymbol\mu^{-\frac{1}{2}}}}^2
\\[3mm]   \leq &\  \frac{2 (C^\QQ_s)^2}{C_{\textrm{eq}}} \norm{\boldf}_{\mathcal{H}^s_\eps}^2 \norm{\boldf}_{H^s_{x,v}\pab{\langle v\rangle^{\frac{\gamma}{2}}\boldsymbol\mu^{-\frac{1}{2}}}}^2,
\end{split}
\end{equation*}
and similarly, for $\mathcal{G}^s_{x,v}$, we get
\begin{equation*}
\mathcal{G}_x^s(\boldf,\boldf)^2 \leq   \frac{2 (C^\QQ_s)^2 }{\eps^2 C_{\textrm{eq}}} \norm{\boldf}_{\mathcal{H}^s_\eps}^2 \norm{\boldf}_{H^s_{x,v}\pab{\langle v\rangle^{\frac{\gamma}{2}}\boldsymbol\mu^{-\frac{1}{2}}}}^2.
\end{equation*}
Plugging the above inequalities into $\eqref{Stability start}$, we thus obtain
\begin{equation*}
\frac{\dd}{\dd t} \norm{\boldf}_{\mathcal{H}^s_\eps}^2 \leq \pa{ \frac{2(C^\QQ_s)^2}{C_{\textrm{eq}}}  \pab{K_1^{(s)} + K_2^{(s)}} \norm{\boldf}_{\mathcal{H}^s_\eps}^2 - K_0^{(s)}}\norm{\boldf}_{H^s_{x,v}\big(\langle v\rangle^{\frac{\gamma}{2}}\boldmu^{-\frac{1}{2}}\big)}^2  + C^{(s)}\delta_\ms.
\end{equation*}
If we now choose $\delta_\bb >0$ satisfying
\begin{equation*}
\frac{2(C^\QQ_s)^2}{C_{\textrm{eq}}}\pab{K_1^{(s)} + K_2^{(s)}}\delta_\bb^2 \leq \frac{K_0^{(s)}}{2},
\end{equation*}
thanks to $\eqref{Sobolev equivalence}$, we can also infer
\begin{equation*}
\begin{split}
\frac{\dd}{\dd t} \norm{\boldf}_{\mathcal{H}^s_\eps}^2 \leq &\ - \frac{K_0^{(s)}}{2} \norm{\boldf}_{H^s_{x,v}\big(\langle v\rangle^{\frac{\gamma}{2}}\boldmu^{-\frac{1}{2}}\big)}^2 + C^{(s)}\delta_\ms 
\\[4mm]    \leq &\ - \frac{K_0^{(s)}}{2 C_{\textrm{\tiny EQ}} } \norm{\boldf}_{\mathcal{H}^s_\eps}^2 +C^{(s)}\delta_\ms, 
\end{split}
\end{equation*}
as long as $\norm{\boldf}_{\mathcal{H}^s_\eps}\leq \delta_\bb$. But now setting $\lambda_\bb = K_0^{(s)}/2C_{\textrm{\tiny EQ}}$,  from Gr\"onwall's lemma, we deduce that, for any $s\geq s_0$, for any $\eps\in (0,\eps_0]$ and for all $\delta_\ms\in [0,\bar{\delta}_\ms]$,
\begin{equation}\label{Uniform control BE}
\norm{\boldf}_{\mathcal{H}^s_\eps}^2 \leq \norm{\boldf^\init}_{\mathcal{H}^s_\eps}^2 e^{-\lambda_\bb t}  +  \frac{C^{(s)} \delta_\ms}{\lambda_\bb} \pa{1 - e^{-\lambda_\bb t}}, 
\end{equation}
or, more explicitly,
\begin{equation*}
\norm{\boldf}_{\mathcal{H}^s_\eps}\leq \disp \max\br{\delta_\bb,\ \pa{\frac{ \delta_\ms C^{(s)} }{\lambda_\bb}}^{1/2}}.
\end{equation*}
Therefore, recalling that $\delta_\bb$ has been chosen independently of the parameter $\bar{\delta}_\ms >0$, if we impose
\begin{equation*}
\bar{\delta}_\ms\leq \frac{\delta_\bb^2}{4} \frac{\lambda_\bb}{C^{(s)}},
\end{equation*}
we finally ensure the validity of the estimate $\norm{\boldf}_{\mathcal{H}^s_\eps}\leq \delta_\bb$ for any time $t\geq 0$, hence concluding the proof.

\end{proof}

\begin{remark}\label{rem:expo decay}
We emphasize that the above result only provides a stability condition for the expansion around $\bepsM$, in the sense that the $\mathcal{H}^s_\eps$ norm of $\boldf$ can be bounded uniformly in time by the constant $\delta_\bb >0$. Unfortunately, we cannot obtain a full exponential decay in time because of the presence of the extra term $C^{(s)}\delta_\ms^2$ which accounts for the macroscopic quantities that define $\bepsM$. It is particularly important to emphasize that the constant $C^{(s)}$ contains factors depending on $\norm{\boldc_{\infty}}_{L^2_x}$ and~${\norm{\bar{\boldu}}_{L^\infty_t H^{s+4}_x}}$, which are uniformly controlled by some positive constants, but cannot exhibit an exponential decay. To be more specific about this issue, we actually guess that the only problematic quantity is the incompressible velocity~$\bar{\boldu}$, because it appears as a multiplicative factor in some of the terms composing $C^{(s)}$. In particular, if we get rid of $\bar{\boldu}$ by looking at the stationary macroscopic state $(\boldc_{\infty},\mathbf{0})$, we can probably recover a global exponential decay in time for $\boldf$, by means of the exponential decays obtained for $\tilde{\boldc}$ and~$\tilde{\boldu}$. Besides, this feature can be explained with the perturbative theory of convergence towards equilibrium which, in our case, prescribes the exponential-in-time relaxation towards the sole global equilibrium $\boldmu$, having macroscopic velocity $\boldu_\infty=\mathbf{0}$. Therefore, the only physically meaningful incompressible velocity $\bar{\boldu}$ should precisely be $\mathbf{0}$. 
\end{remark}


\bigskip
\noindent \textbf{Step 3 -- Existence and uniqueness of the perturbation $\boldf$.} We are finally able to prove that the solution~$\boldf$ to the perturbed Boltzmann equation $\eqref{perturbed BE}$ exists uniquely in time, on $\R_+$. Even if the proofs are very standard, we sketch the ideas behind and the computations, for the sake of completeness.

\begin{prop}\label{prop:Existence BE}
Let the collision kernels $B_{ij}$ satisfy assumptions $\mathrm{(H1)}$--$\mathrm{(H2)}$--$\mathrm{(H3)}$--$\mathrm{(H4)}$, and consider the local Maxwellian $\bepsM$ defined by \eqref{Local Maxwellian eps}--\eqref{Meps fluid quantities}. Let $s\geq s_0$, given by Lemma \ref{LemmaQ}. There exist $\eps_0\in (0,1]$ and $\bar{\delta}_\ms$, $\delta_\bb >0$ such that, for all $\eps\in (0,\eps_0]$, all $\delta_\ms\in [0,\bar{\delta}_\ms]$, and for any initial datum $\boldf^\init$ in~$\sobolevTR{s}$ satisfying
\begin{equation*}
\norm{\boldf^\init}_{\mathcal{H}^s_\eps}\leq \delta_\bb,\qquad \norm{\pi_\bepsT(\boldf^\init)}_{\hilbertxv} = \mathcal{O}(\delta_\ms),
\end{equation*}
there exists $\boldf\in C^0\pa{\R_+;\sobolevTR{s}}$ such that $\boldF^\eps=\bepsM+\eps\boldf$ is a weak solution of the Boltzmann multi-species equation $\eqref{perturbed BE}$. Moreover, if ${\boldF^{\eps,\init}=\mathbf{M}^{\eps,\init}+\eps\boldf^\init\geq 0}$, then also $\boldF^\eps(t,x,v)\geq 0$ a.e. on $\R_+\times\T^3\times\R^3$.

\end{prop}

\begin{proof}[Proof of Proposition \ref{prop:Existence BE}]
The proof of the existence is based on a standard iterative method, where we construct a solution on a finite time interval $[0,T_0]$, and we then show that $T_0$ can be extended up to $+\infty$. Let $s \geq s_0$, given by Lemma \ref{LemmaQ} and consider a solution $(\boldc,\boldu)$ of the Maxwell-Stefan system \eqref{MS mass}--\eqref{MS momentum}--\eqref{MS incompressibility} which is at least $L^\infty\pab{\R_+;H^{s+5}(\T^3)}\times L^\infty\pab{\R_+;H^{s+4}(\T^3)}$. In particular, recall that $(\boldc,\boldu)$ also belongs to $C^0\pab{\R_+;H^{s+4}(\T^3)}\times C^0\pab{\R_+;H^{s+3}(\T^3)}$, so that the local Maxwellian $\bepsM$ is continuous with respect to $t\geq 0$. 

\smallskip
With these choices, set initially 
\begin{equation}\label{Begin iteration}
\boldf^{(0)}=\boldf^\init,\qquad T_0 = \frac{\delta_\bb \min\br{1,\frac{K_0^{(s)}}{2}}}{4 C^{(s)}\delta_\ms},
\end{equation}
and suppose that on $[0,T_0]$ a sequence of functions $\big(f^{(\bar{n})}\big)_{0<\bar{n}\leq n}$ is given up to an integer~${n \in\N^*}$, satisfying, for any $0<\bar{n}\leq n$ and for any $t\in [0,T_0]$,
\begin{equation*}
\boldf^{(\bar{n})}\in \sobolevTR{s},\qquad \norm{\pi_\bepsT\big(\boldf^{ (\bar{n}) }\big)}_{\hilbertxv} = \mathcal{O}(\delta_\ms).
\end{equation*}
By induction on $n>0$, we define the function $\boldf^{(n+1)}$ such that
\begin{equation}\label{Iterative scheme BE}
\left\{\begin{array}{l}
\partial_t \boldf^{(n+1)} + \frac{1}{\eps}v\cdot\nabla_x \boldf^{(n+1)} = \frac{1}{\eps^2} \bepsL(\boldf^{(n+1)}) + \frac{1}{\eps}\boldQ(\boldf^{(n)},\boldf^{(n+1)}) + \bepsS,\\[6mm]
\restriction{\boldf^{(n+1)}}{t=0} = \boldf^\init.
\end{array}\right.
\end{equation}
It is now a classical result the existence of a solution $\boldf^{(n+1)}\in\sobolevTR{s}$ for the above evolution equation. Indeed, recalling the definition of $\bepsT=\frac{1}{\eps^2}\boldL-\frac{1}{\eps}v\cdot\nabla_x$, we can rewrite
\begin{equation}\label{Evolution existence}
\partial_t \boldf^{(n+1)} = \bepsT(\boldf^{(n+1)}) + \frac{1}{\eps^2}\pab{\bepsL-\boldL}(\boldf^{(n+1)})+ \frac{1}{\eps}\boldQ(\boldf^{(n)},\boldf^{(n+1)}) + \bepsS.
\end{equation}
Since $\boldL$ is self-adjoint in $\hilbertTR$ and possesses a spectral gap $\lambda_\boldL$, and $v\cdot\nabla_x$ is anti-symmetric, it is easy to prove that $\bepsT$ generates a strongly continuous semigroup in~$\sobolevTR{s}$ (see \cite{Kat} for the general theory, and \cite[Section 4.2]{BriDau} for the specific case of the Boltzmann multi-species equation). In particular, as done in \cite[Section 6]{Bri1}, introducing the following functional defined on $\sobolevTR{s}$
\begin{equation*}
E_{[0,T_0]}(\boldf) = \sup_{t\in[0,T_0]} \pa{\norm{\boldf(t)}_{\mathcal{H}^s_\eps}^2 + \int_0^t \norm{\boldf(\tau)}_{H^s_{x,v}\big(\langle v\rangle^{\frac{\gamma}{2}}\boldmu^{-\frac{1}{2}}\big)}^2 \dd\tau},
\end{equation*}
from estimates \eqref{Property7}--\eqref{Property8} of Lemma \ref{LemmaQ}, it immediately follows that
$$\boldQ(\boldf^{(n)},\cdot): \pa{\sobolevTR{s},E_{[0,T_0]}(\cdot)}\to \pa{\sobolevTR{s},\norm{\cdot}_{\sobolevxv{s}}}$$ 
is a bounded linear operator, and the same holds for the penalization $\bepsL-\boldL$. To ease the computations, from now on let us drop the subscript $[0,T_0]$ in $E_{[0,T_0]}$. By means of the functional~$E$, one can then apply Duhamel's formula in combination with a suitable fixed point argument to show that, as long as $\norm{\boldf^\init}_{\mathcal{H}^s_\eps}$ is chosen small enough, there exists a solution $\boldf^{(n+1)}$ to $\eqref{Evolution existence}$, which is in $\sobolevTR{s}$ for any time $t\in [0,T_0]$. Moreover, applying $\pi_\bepsT$ to $\eqref{Evolution existence}$ and performing an identical study to the one carried out in the proof of Lemma \ref{lem:Poincare} to derive the Poincar\'e inequality $\eqref{Poincare inequality}$, we also ensure that
\begin{equation*}
\norm{\pi_\bepsT\big(\boldf^{(n+1)}\big)}_{\hilbertxv} \leq C^\TT \delta_\ms, \qquad \forall t\in[0,T_0],
\end{equation*}
so that, by induction, the sequence $\pab{\boldf^{(n)}}_{n\in\N}$ remains well-defined.

\smallskip
Starting from these considerations, one then aims at proving that the constructed sequence~${\pab{\boldf^{(n)}}_{n\in\N}}$ can be uniformly bounded in the $E$ norm. For this, since obviously~${E\big(\boldf^{(0)}\big)=\norm{\boldf^\init}_{\mathcal{H}^s_\eps}\leq \delta_\bb}$ from our initial choice $\eqref{Begin iteration}$, we proceed by induction on~${n\in\N}$, supposing that we can uniformly bound $E\big(\boldf^{(n)}\big)\leq \delta_\bb$ up to some integer $n >0$, and proving that also $E\big(\boldf^{(n+1)}\big)\leq \delta_\bb$.

\smallskip
Now, each $\boldf^{(n)}$ satisfies all the hypotheses needed in order to copy the computations used to derive the \textit{a priori} estimates in the previous step. More precisely, from Proposition \ref{prop:preliminary a priori BE}, we know that there exist $\bar{\delta}_\ms >0$ and $\eps_0\in (0,1]$ such that, for any $s \geq s_0$, for any~${\delta_\ms\in [0,\bar{\delta}_\ms]}$ and ${\eps\in (0,\eps_0]}$, one recovers the estimate
\begin{multline*}
\frac{\dd}{\dd t} \norm{\boldf^{(n+1)}}_{\mathcal{H}^s_\eps}^2 \leq -K_0^{(s)}\norm{\boldf^{(n+1)}}_{H^s_{x,v}\big(\langle v\rangle^{\frac{\gamma}{2}}\boldmu^{-\frac{1}{2}}\big)}^2 
\\[2mm]    + K_1^{(s)}\mathcal{G}^s_x\pab{\boldf^{(n)},\boldf^{(n+1)}}^2 + \eps^2 K_2^{(s)}\mathcal{G}^s_{x,v}\pab{\boldf^{(n)},\boldf^{(n+1)}}^2 +C^{(s)} \delta_\ms. 
\end{multline*}
Moreover, thanks to $\eqref{Sobolev equivalence}$ and proceeding like in the proof of Corollary \ref{cor:A priori BE}, we also establish the bounds
\begin{equation*}
\begin{split}
\mathcal{G}_x^s \big(\boldf^{(n)} &,\boldf^{(n+1)}\big)^2 
\\[2mm]      & \leq  \frac{4 (C^\QQ_s)^2}{C_{\textrm{eq}}} \pa{\norm{\boldf^{(n)}}_{\mathcal{H}^s_\eps}^2 \norm{\boldf^{(n+1)}}_{H^s_{x,v}\pab{\langle v\rangle^{\frac{\gamma}{2}}\boldsymbol\mu^{-\frac{1}{2}}}}^2 + \norm{\boldf^{(n+1)}}_{\mathcal{H}^s_\eps}^2 \norm{\boldf^{(n)}}_{H^s_{x,v}\pab{\langle v\rangle^{\frac{\gamma}{2}}\boldsymbol\mu^{-\frac{1}{2}}}}^2},
\\[6mm]   \mathcal{G}_x^s  \big(\boldf^{(n)} &,\boldf^{(n+1)}\big)^2 
\\[2mm]    & \leq \frac{4 (C^\QQ_s)^2 }{\eps^2 C_{\textrm{eq}}} \pa{\norm{\boldf^{(n)}}_{\mathcal{H}^s_\eps}^2 \norm{\boldf^{(n+1)}}_{H^s_{x,v}\pab{\langle v\rangle^{\frac{\gamma}{2}}\boldsymbol\mu^{-\frac{1}{2}}}}^2 + \norm{\boldf^{(n+1)}}_{\mathcal{H}^s_\eps}^2 \norm{\boldf^{(n)}}_{H^s_{x,v}\pab{\langle v\rangle^{\frac{\gamma}{2}}\boldsymbol\mu^{-\frac{1}{2}}}}^2}.
\end{split}
\end{equation*}
Consequently, naming $K^{(s)} = \frac{4 (C^\QQ_s)^2 }{C_{\textrm{eq}}}\pab{K_1^{(s)} + K_2^{(s)}}$, we deduce the following estimate 
\begin{equation*}
\begin{split}
\frac{\dd}{\dd t} \norm{\boldf^{(n+1)}}_{\mathcal{H}^s_\eps}^2 \leq &\ -K_0^{(s)}\norm{\boldf^{(n+1)}}_{H^s_{x,v}\big(\langle v\rangle^{\frac{\gamma}{2}}\boldmu^{-\frac{1}{2}}\big)}^2 + K^{(s)} \norm{\boldf^{(n)}}_{\mathcal{H}^s_\eps}^2 \norm{\boldf^{(n+1)}}_{H^s_{x,v}\pab{\langle v\rangle^{\frac{\gamma}{2}}\boldsymbol\mu^{-\frac{1}{2}}}}^2  
\\[5mm]  &\ \hspace{3cm} + K^{(s)} \norm{\boldf^{(n+1)}}_{\mathcal{H}^s_\eps}^2 \norm{\boldf^{(n)}}_{H^s_{x,v}\pab{\langle v\rangle^{\frac{\gamma}{2}}\boldsymbol\mu^{-\frac{1}{2}}}}^2 +C^{(s)} \delta_\ms
\\[4mm]   \leq &\ \pa{K^{(s)}E\big(\boldf^{(n)}\big) - K_0^{(s)}} \norm{\boldf^{(n+1)}}_{H^s_{x,v}\big(\langle v\rangle^{\frac{\gamma}{2}}\boldmu^{-\frac{1}{2}}\big)}^2 
\\[4mm]    &\ \hspace{3.2cm}+ K^{(s)} E\big(\boldf^{(n+1)}\big) \norm{\boldf^{(n)}}_{H^s_{x,v}\big(\langle v\rangle^{\frac{\gamma}{2}}\boldmu^{-\frac{1}{2}}\big)}^2 + C^{(s)}\delta_\ms,
\end{split}
\end{equation*}
holding for any $t\in [0,T_0]$.

Choosing $\delta_\bb>0$ such that $E\big(\boldf^{(n)}\big)\leq K_0^{(s)}/2 K^{(s)}$, we can thus integrate the previous inequality on $[0,t]$, with $t\leq T_0$, to obtain
\begin{multline*}
\norm{\boldf^{(n+1)}}_{\mathcal{H}^s_\eps}^2 + \frac{K_0^{(s)}}{2}\int_0^t \norm{\boldf^{(n+1)}(\tau)}_{H^s_{x,v}\big(\langle v\rangle^{\frac{\gamma}{2}}\boldmu^{-\frac{1}{2}}\big)}^2 \dd\tau 
\\[4mm]    \leq \norm{\boldf^\init}_{\mathcal{H}^s_\eps}^2 + K^{(s)} E\pab{\boldf^{(n+1)}} E\pab{\boldf^{(n)}} + t C^{(s)}\delta_\ms.
\end{multline*}
Therefore, if we also suppose that
$$\delta_\bb \leq \min\br{1,\frac{K_0^{(s)}}{2}} / 2 K^{(s)}\qquad \textrm{and}\qquad \delta_\bb\leq  \min\br{1,\frac{K_0^{(s)}}{2}} / 4,$$
thanks to the induction hypothesis on $E\pab{\boldf^{(n)}}$, we deduce that
\begin{equation*}
E\pab{\boldf^{(n+1)}} \leq \frac{2}{\min\br{1,\frac{K_0^{(s)}}{2}}} \norm{\boldf^\init}_{\mathcal{H}^s_\eps}^2 + \frac{2 C^{(s)}\delta_\ms}{\min\br{1,\frac{K_0^{(s)}}{2}}} T_0 \leq \delta_\bb,
\end{equation*}
thanks to our choice of $T_0$. Hence, the sequence ${\pab{\boldf^{(n)}}_{n\in\N}}$ is uniformly bounded by $\delta_\bb$ in the~$E$ norm, and thus also in~${L^\infty\pa{0,T_0;\sobolevxv{s}}\ \cap\  L^1\pa{0,T_0;H^s_{x,v}\big(\langle v\rangle^{\frac{\gamma}{2}}\boldmu^{-\frac{1}{2}}\big)}}$. Therefore, thanks to the Rellich-Kondrachov theorem on compact embeddings into less regular Sobolev spaces, we can take the limit $n\to +\infty$ in $\eqref{Iterative scheme BE}$, since $\bepsT$, $\bepsL-\boldL$ and~$\boldQ$ are continuous. In particular, we can extract a subsequence that converges towards a function~$\boldf^{(\infty)}$ which belongs to $C^0\pa{[0,T_0];\sobolevTR{s}}$ and solves the initial value problem
\begin{equation*}
\left\{\begin{array}{l}
\partial_t \boldf^{(\infty)} + \frac{1}{\eps}v\cdot\nabla_x \boldf^{(\infty)} = \frac{1}{\eps^2} \bepsL(\boldf^{(\infty)}) + \frac{1}{\eps}\boldQ(\boldf^{(\infty)},\boldf^{(\infty)}) + \bepsS,\\[6mm]
\restriction{\boldf^{(\infty)}}{t=0} = \boldf^\init.
\end{array}\right.
\end{equation*}
This proves our result on the interval $[0,T_0]$. Moreover, $\boldf^{(\infty)}$ satisfies
\begin{equation*}
\norm{\pi_\bepsT\big(\boldf^{(\infty)}\big)}_{\hilbertxv} \leq C^\TT \delta_\ms, \qquad \forall t\in[0,T_0],
\end{equation*}
and the estimate $\norm{\boldf^{(\infty)}}_{\mathcal{H}^s_\eps}\leq \delta_\bb$, thanks to Corollary \ref{cor:A priori BE}. By simply restarting this procedure on a new time interval $[T_0,2T_0]$ using $\boldf^{(\infty)}(T_0)$ as initial datum and considering the corresponding functional~${E_{[T_0,2T_0]}}$, one can repeat the previous computations, so that we recover the existence of a solution $\boldf\in C^0\pa{\R_+;\sobolevTR{s}}$, as desired. In particular, thanks to the \textit{a priori} estimates that we previously established, this solution satisfies the uniform control
\begin{equation*}
\norm{\boldf}_{\mathcal{H}^s_\eps}^2 \leq \delta_\bb,\quad\forall t\geq 0,
\end{equation*}
holding for any $\eps\in (0,\eps_0]$.

\smallskip
The positivity of the solution $\boldF^\eps=\bepsM + \eps\boldf$ is finally showed using a standard method, which can be found for example in \cite[Section 6.3]{BriDau}. This ends the proof.

\end{proof}

Finally, the constructed solution is unique, providing the full Cauchy theory of our problem.

\begin{prop}\label{prop:Uniqueness BE}
Let $s \geq s_0$ and consider a function $\boldf^\init$ satisfying the assumptions of Proposition \ref{prop:Existence BE}. There exist $\bar{\delta}_\ms >0$ and ${\eps_0\in (0,1]}$ such that, if $\boldf$ and $\boldg$ are two solutions of $\eqref{perturbed BE}$ having the same initial datum $\boldf^\init$, then $\boldf=\boldg$.
\end{prop}

\begin{proof}[Proof of Proposition \ref{prop:Uniqueness BE}]
As in the proof of the previous result, fix initially a time
\begin{equation*}
T_0 = \frac{\delta_\bb \min\br{1,\frac{K_0^{(s)}}{2}}}{4 C^{(s)}\delta_\ms}.
\end{equation*}
Next, define $\boldh = \boldf - \boldg$. Subtracting the equations satisfied by $\boldf$ and $\boldg$, we see that $\boldh$ is solution on $[0,T_0]\times\T^3\times\R^3$ of
\begin{equation*}
\left\{\begin{array}{l}
\partial_t \boldh +\frac{1}{\eps}v\cdot \nabla_x\boldh = \frac{1}{\eps^2}\bepsL(\boldh) + \frac{1}{\eps}\pabb{\boldQ(\boldh,\boldf) + \boldQ(\boldg,\boldh)},\\[4mm]
\restriction{\boldh}{t=0} = 0.
\end{array}\right.
\end{equation*}
The idea is to derive similar estimates to the ones obtained in the proof of the previous result. For this, note the linear part obeys the same upper bounds, and for the nonlinear terms we shall again use estimates $\eqref{Property8}$ of Lemma \ref{LemmaQ}. 

\smallskip
More precisely, using the \textit{a priori} estimate of Proposition \ref{prop:preliminary a priori BE}, we get
\begin{multline*}
\frac{\dd}{\dd t} \norm{\boldh}_{\mathcal{H}^s_\eps}^2 \leq -K_0^{(s)}\norm{\boldh}_{H^s_{x,v}\big(\langle v\rangle^{\frac{\gamma}{2}}\boldmu^{-\frac{1}{2}}\big)}^2 + K_1^{(s)}\pabb{\mathcal{G}^s_x\pab{\boldh,\boldf}^2 + \mathcal{G}^s_x\pab{\boldg,\boldh}^2} 
\\[2mm]       + \eps^2 K_2^{(s)}\pabb{\mathcal{G}^s_{x,v}\pab{\boldh,\boldf}^2 + \mathcal{G}^s_{x,v}\pab{\boldg,\boldh}^2}. 
\end{multline*}
The terms inside the parentheses are then bounded is the same way as before, using $\eqref{Property8}$ and the norm equivalences provided by $\eqref{Sobolev equivalence}$. Skipping the computations and denoting~${E=E_{[0,T_0]}}$, we recover the existence of a positive constant $\tilde{K}$ such that
\begin{equation*}
\begin{split}
\frac{\dd}{\dd t} \norm{\boldh}_{\mathcal{H}^s_\eps}^2 & \leq \Bigg( \tilde{K} \pabb{\norm{\boldf}_{\mathcal{H}^s_\eps}^2 + \norm{\boldg}_{\mathcal{H}^s_\eps}^2} - K_0^{(s)} \Bigg) \norm{\boldh}_{H^s_{x,v}\big(\langle v\rangle^{\frac{\gamma}{2}}\boldmu^{-\frac{1}{2}}\big)}^2
\\[3mm]      &\hspace{3cm} + \tilde{K} \pa{\norm{\boldf}_{H^s_{x,v}\big(\langle v\rangle^{\frac{\gamma}{2}}\boldmu^{-\frac{1}{2}}\big)}^2 + \norm{\boldg}_{H^s_{x,v}\big(\langle v\rangle^{\frac{\gamma}{2}}\boldmu^{-\frac{1}{2}}\big)}^2}  \norm{\boldh}_{\mathcal{H}^s_\eps}^2
\\[6mm]     & \leq \Bigg( \tilde{K} \pabb{E(\boldf) + E(\boldg)} - K_0^{(s)} \Bigg) \norm{\boldh}_{H^s_{x,v}\big(\langle v\rangle^{\frac{\gamma}{2}}\boldmu^{-\frac{1}{2}}\big)}^2
\\[3mm]      &\hspace{3cm} + \tilde{K} \pa{\norm{\boldf}_{H^s_{x,v}\big(\langle v\rangle^{\frac{\gamma}{2}}\boldmu^{-\frac{1}{2}}\big)}^2 + \norm{\boldg}_{H^s_{x,v}\big(\langle v\rangle^{\frac{\gamma}{2}}\boldmu^{-\frac{1}{2}}\big)}^2}  E(\boldh).
\end{split}
\end{equation*}
In particular, since $\boldf$ is a solution of the Boltzmann multi-species equation $\eqref{perturbed BE}$, we have seen previously that it satisfies the estimate $E(\boldf) \leq \delta_\bb$ on $[0,T_0]$, and the same holds for $\boldg$. Therefore, choosing initially $\delta_\bb\leq K_0^{(s)}/4\tilde{K}$ and integrating on $[0,t]$, with $t\leq T_0$, since $\boldh^\init=0$ we obtain this time
\begin{equation*}
\norm{\boldh}_{\mathcal{H}^s_\eps}^2 + \frac{K_0^{(s)}}{2}\int_0^t \norm{\boldh(\tau)}_{H^s_{x,v}\big(\langle v\rangle^{\frac{\gamma}{2}}\boldmu^{-\frac{1}{2}}\big)}^2 \dd\tau \leq \tilde{K} E(\boldh)\pabb{E(\boldf) + E(\boldg)},
\end{equation*}
from which immediately follows that
\begin{equation*}
E(\boldh) \leq \frac{2\tilde{K} \delta_\bb}{\min\br{1,\frac{K_0^{(s)}}{2}}} E(\boldh).
\end{equation*}
Taking $\delta_\bb >0$ sufficiently small allows to conclude that $E(\boldh)=0$, and thus $\boldf=\boldg$ for any time~$t\in [0,T_0]$. We can then repeat the computations on the time interval $[T_0,2T_0]$ wtih $E=E_{[T_0,2T_0]}$, recovering the same conclusions. An iteration of this procedure finally allows to deduce that $\boldf=\boldg$ for all $t\geq 0$. This ends the proof.

\end{proof}

\noindent \textbf{Step 4 -- Conclusion.} Theorem \ref{theo:Cauchy BE} is a direct gathering of Propositions \ref{prop:preliminary a priori BE}, \ref{prop:Existence BE}, \ref{prop:Uniqueness BE} and Corollary \ref{cor:A priori BE}.
Moreover, thanks to Theorems \ref{theo:Cauchy MS} and \ref{theo:Cauchy BE}, both the local Maxwellian $\bepsM$ and the perturbation $\boldf$ are well-defined, are unique and exist globally in time. Therefore, we can now reconstruct $\boldF^\eps = \bepsM + \eps \boldf$, which is the unique global weak solution of the Boltzmann multi-species equation $\eqref{perturbed BE}$, perturbed around the non-equilibrium state $\bepsM$. 
To conclude, since the perturbation $\boldf$ is uniformly bounded in the $\mathcal{H}^s_\eps$ norm by the constant $\delta_\bb$ (independent of $\eps$ and computed in the Corollary \ref{cor:A priori BE}), we also deduce the stability property
\begin{equation*}
\norm{\boldF^\eps-\bepsM}_{\mathcal{H}^s_\eps} = \eps\norm{\boldf}_{\mathcal{H}^s_\eps}\leq \eps \delta_\bb,\quad \forall \eps\in (0,\eps_0].
\end{equation*}

\bigskip

\section{Technical proofs of the hypocoercivity properties}\label{sec:technical proofs}

\noindent In this last section we collect all the proofs of the lemmata which lead to the hypocoercive structure of our model. 

\smallskip
Since, in what follows, we often estimate the Euclidean distance between~$\bepsM$ and~$\boldmu$, in order to enlighten our computations we here introduce the local and global Maxwellians~${\bmathepsM=\mathbf{M}_{(\mathbf{1},\eps \boldu, \mathbf{1})}}$ and $\bmathM=\mathbf{M}_{(\mathbf{1},\mathbf{0}, \mathbf{1})}$ which will allow us to write $\boldmu=\boldc_{\infty}\bmathM$ and, in this way, separate $\bepsM$ into a close-to-equilibrium part and a lower order term as~${\bepsM=\boldc_{\infty}\bmathepsM+\eps \tilde{\boldc}\bmathepsM}$.

\subsection{Proof of Lemma \ref{lemma:Properties 1-2}}\label{Lemma1}
Fix $\delta\in (0,1)$. In order to prove $\eqref{Property1}$, we initially observe that, since $\epsM_i=c_i {\mathcal{M}_i^\eps}$, we can apply to ${\mathcal{M}_i^\eps}$ similar estimates to the ones derived in \cite[Lemma 4.1]{BonBouBriGre} to obtain for any $1\leq i\leq N$ the upper bound
\begin{equation}\label{UpperM}
\epsM_i(t,x,v) \leq C_\delta^{\mathrm{up}} \pa{1+\exp\br{\frac{4 m_i}{1-\delta}\eps^2|u_i(t,x)|^2}} c_i\mathM_i^{\delta}(v),\quad t\geq 0,x\in\T^3, v\in\R^3,
\end{equation}
where
$$C_{\delta}^{\mathrm{up}}=\max_{1\leq i\leq N} \pa{\frac{m_i}{2\pi}}^{\frac{3(1-\delta)}{2}}\pa{\sup_{|v|\in\R_+}e^{-(1-\delta)\frac{m_i}{4}|v|^2}+1} >0.$$
Similarly we can recover a lower bound for $\epsM_i$ as follows. For any $1\leq i\leq N$, we can write
$${\mathcal{M}_i^\eps}(t,x,v)=\mathM_i^{1/\delta}(v){\mathcal{M}_i^\eps}(t,x,v)\mathM_i^{-1/\delta}(v),\quad t\geq 0,\ x\in\T^3,\ v\in\R^3.$$
Then, the product $\mathcal{M}^\eps \mathM_i^{-1/\delta}$ can be lower estimated as
$${\mathcal{M}_i^\eps}(v)\mathM_i^{-1/\delta}(v)\geq\pa{\frac{m_i}{2\pi}}^{\frac{3(\delta-1)}{2\delta}}\exp\br{\frac{m_i(1-\delta)}{2\delta}|v|^2-m_i\eps|v||u_i|-\eps^2\frac{m_i}{2}|u_i|^2},$$
where we have dropped the dependence on both variables $(t,x)$ for the sake of simplicity. We then distinguish two cases.

\begin{enumerate}
\item For $\disp |v|\geq\frac{4\delta\eps|u_i|}{1-\delta}$, we have
$${\mathcal{M}_i^\eps}(v)\mathM_i^{-1/\delta}(v)\geq \pa{\frac{m_i}{2\pi}}^{\frac{3(\delta-1)}{2\delta}}\exp\br{-\frac{m_i}{2}\eps^2|u_i|^2}.$$

\vspace{-1mm}
\item For $\disp |v|\leq\frac{4\delta\eps|u_i|}{1-\delta}$, we get
$${\mathcal{M}_i^\eps}(v)\mathM_i^{-1/\delta}(v)\geq \pa{\frac{m_i}{2\pi}}^{\frac{3(\delta-1)}{2\delta}}\exp\br{-4m_i\frac{\delta}{1-\delta}\eps^2|u_i|^2-\frac{m_i}{2}\eps^2|u_i|^2}.$$
\end{enumerate}
These estimates allow to deduce that
\begin{equation}\label{LowerM}
\epsM_i(v)\geq C_\delta^{\mathrm{low}}\exp\br{-m_i\frac{7\delta+1}{2(1-\delta)}\eps^2|u_i|^2}c_i\mathM_i^{1/\delta}, \quad\forall v\in\R^3,
\end{equation}
where we simply set
$$C_\delta^{\mathrm{low}}=\min_{1\leq i\leq N}\pa{\frac{m_i}{2\pi}}^{\frac{3(\delta-1)}{2\delta}}>0.$$

Gathering now $\eqref{UpperM}$ and $\eqref{LowerM}$, it is easy to deduce that, for any $1\leq i,j\leq N$, there exist two positive constants $\bar{\nu}_{ij}^{(\delta)}, \tilde{\nu}_{ij}^{(\delta)}$ such that, for all $v\in\R^3$,
$$0< \bar{\nu}_{ij}^{(\delta)}  e^{-m_j\frac{7\delta+1}{2(1-\delta)}\eps^2|u_j|^2}c_j \langle v \rangle^\gamma \leq     \epsnu_{ij}(v)\leq \tilde{\nu}_{ij}^{(\delta)} \pa{1 + e^{\frac{4 m_j}{1-\delta}\eps^2|u_j|^2}}c_j  \langle v \rangle^\gamma.$$
This in turns implies that, for any $1\leq i\leq N$, there also exist $\bar{\nu}_i^{(\delta)},\tilde{\nu}_i^{(\delta)} > 0$ such that,
for all $v\in\R^3$,
\begin{equation}\label{Upper+Lower}
\begin{split}
0<\bar{\nu}_i^{(\delta)}  \exp\Bigg\{& - \max_{1\leq i\leq N}m_i\frac{7\delta+1}{2(1-\delta)}\eps^2\norm{\boldu}_{L^\infty_tL_x^\infty}^2\Bigg\}\pa{\min_{1\leq i\leq N}c_i} \langle v  \rangle^\gamma 
\\[4mm]   & \leq   \epsnu_i(v)\leq \tilde{\nu}_i^{(\delta)} \pa{1 + \exp\Bigg\{\max_{1\leq i\leq N} m_i\frac{4\eps^2}{1-\delta}\norm{\boldu}_{L^\infty_tL_x^\infty}^2\Bigg\}}\norm{\boldc}_{L^\infty_tL_x^\infty}  \langle v \rangle^\gamma,
\end{split}
\end{equation}
recalling that $\min_i c_i >0$ a.e. on $\R_+\times\T^3$.

Since $\bepsnu$ is a multiplicative operator, from these bounds it is finally easy to compute explicitly the values of $C^{\NuNu}_1$ and $C^{\NuNu}_2$ such that
$$C^{\NuNu}_1\norm{\boldf}_{\spacev}^2\leq \scalprod{\bepsnu(\boldf),\boldf}_{\hilbertv}\leq C^{\NuNu}_2\norm{\boldf}_{\spacev}^2$$
is satisfied. Moreover, since $\langle v\rangle^\gamma \geq 1$, it is then straightforward that the $\spacev$ norm upper bounds the $\hilbertv$ norm. Thus $\eqref{Property1}$ is proved for any $\delta\in (0,1)$.

\bigskip
At last, estimate $\eqref{Property2}$ can be proved in a very similar way using the fact that $\bepsK$ can be written under a kernel form and applying $\eqref{UpperM}$ to obtain the same bounds as in \cite{BonBouBriGre}. From this and from $\eqref{Property1}$, using the Cauchy-Schwarz inequality, we infer $\eqref{Property2}$ by choosing
$$C^\LL_1 =C(\delta) \norm{\boldc}_{L^\infty_t L_x^\infty}\Bigg(1 + \exp\br{\max_{1\leq i\leq N}m_i\frac{4\eps^2}{1-\delta}\norm{\boldu}_{L^\infty_tL_x^\infty}}\Bigg)\max_{1\leq i,j\leq N}\sqrt{\frac{\bar{c}_j}{\bar{c}_i}} >0.$$
Here $C(\delta)$ is a constant only depending on the masses $(m_i)_{1\leq i\leq N}$, the number of species $N$ and an arbitrary parameter $\delta\in(\bar{\delta},1)$, where $\bar{\delta}$ is fixed independently of $\eps$. This concludes the proof.

\subsection{Proof of Lemma \ref{LemmaNu}}
We prove the result in the simple case where $|\alpha|=|\beta|=1$ (the case $|\alpha|=0$ being included). The general case can then be obtained in a similar way, by iterating our computations.

We first notice that, thanks to our assumptions on the collision kernels $B_{ij}$, it is easy to check that $\nabla_v \epsnu_i\in L_v^\infty(\R^3)$ for any $1\leq i\leq N$. In fact, choosing $\delta\in (0,1)$ independently of $\eps$ and using the upper bound $\eqref{UpperM}$ on $\epsM_j$, we obtain, for any $v\in\R^3$,
\begin{eqnarray*}
|\nabla_v\epsnu_i(v)|&=&\left|\sum_{j=1}^N\int_{\R^3\times\Sf} b_{ij}(\cos\vartheta)\gamma|v-v_*|^{\gamma -1}\frac{v-v_*}{|v-v_*|}\epsM_j(v_*)\dd v_* \dd\sigma\right|           
\\[2mm]   &\leq& C(\delta)\sum_{j=1}^N c_j\int_{\R^3}\gamma|v-v_*|^{\gamma-1}e^{-\delta m_j\frac{|v_*|^2}{2}}\dd v_* \leq C(\delta)\scalprod{\boldc,\mathbf{1}} <+\infty,              
\end{eqnarray*}
since $\gamma\in[0,1]$ and thus the above integral is clearly finite.

\smallskip
Next, considering the $(x_k,v_\ell)$ derivatives, we can write
\begin{equation*}
\begin{split}
\scalprod{\partial_{v_\ell}\partial_{x_k}\bepsnu(\boldf),\partial_{v_\ell}\partial_{x_k}\boldf}_{\hilbertxv}  =&  \sum_{i=1}^N \int_{\T^3\times\R^3}\partial_{v_\ell}(\epsnu_i\partial_{x_k}f_i)\partial_{v_\ell}\partial_{x_k}f_i\mu_i^{-1}\dd x\dd v
\\[2mm]&+ \sum_{i=1}^N\int_{\T^3\times\R^3}\partial_{v_\ell}(\partial_{x_k}\epsnu_i f_i)\partial_{v_\ell}\partial_{x_k}f_i\mu_i^{-1}\dd x\dd v.
\end{split}
\end{equation*}
Denote by $I_1$ and $I_2$, respectively, the first and second term on the right-hand side. Using Young's inequality, the first term can be estimated as
\begin{eqnarray*}
I_1 &=& \sum_{i=1}^N \pa{\int_{\T^3\times\R^3}\partial_{v_\ell}\epsnu_i\partial_{x_k}f_i\partial_{v_\ell}\partial_{x_k}f_i\mu_i^{-1}\dd x\dd v+\int_{\T^3\times\R^3} \epsnu_i |\partial_{v_\ell}\partial_{x_k}f_i|^2\mu_i^{-1}\dd x\dd v} 
\\[1mm]&\geq&    \sum_{i=1}^N \pa{-\frac{1}{2}\int_{\T^3\times\R^3}\frac{(\partial_{v_\ell}\epsnu_i)^2}{\epsnu_i}(\partial_{x_k}f_i)^2 \mu_i^{-1}\dd x\dd v+\frac{1}{2}\int_{\T^3\times\R^3} \epsnu_i |\partial_{v_\ell}\partial_{x_k}f_i|^2\mu_i^{-1}\dd x\dd v}       
\\[3mm]&\geq& C^{\NuNu}_3 \norm{\partial_{v_\ell}\partial_{x_k}\boldf}_{\spacexv}^2 -C^{\NuNu}_5\norm{\boldf}_{\sobolevxv{1}}^2,
\end{eqnarray*}
where $C^{\NuNu}_3$ is obtained from $\eqref{Upper+Lower}$ as
$$C^{\NuNu}_3=\frac{1}{2}\pa{\min_{1\leq i\leq N}\bar{\nu}_i^{(\delta)} c_i} \exp\br{-\max_{1\leq i\leq N} m_i \frac{7\delta + 1}{2(1-\delta)}\eps^2\norm{\boldu}_{L^\infty_t L^\infty_x}}>0,$$
and

$$C^{\NuNu}_5=\max_{1\leq i\leq N}\sup_{v\in\R^3}\frac{(\partial_{v_\ell}\epsnu_i)^2}{2\epsnu_i}\hspace{7cm}$$
is positive and finite thanks again to the lower bound $\eqref{Upper+Lower}$ and the fact that $\nabla_v\epsnu_i\in L_v^\infty(\R^3)$ for any $1\leq i\leq N$. 

To estimate the second term, we first notice that from $\mathbf{M^\eps} = \boldc_{\infty}\mathbf{\mathcal{M}^\eps} + \eps\mathbf{\tilde{c}}\mathbf{\mathcal{M}^\eps}$ we get, for any $(t,x,v)\in\R_+\times\T^3\times\R^3$,
\begin{equation*}
\begin{split}
\partial_{x_k}\epsnu_i  (t&,x,v) \hspace{12cm}
\\[1mm]  = &\ \sum_{j=1}^N \int_{\R^3\times\Sf}b_{ij}(\cos\vartheta)|v-v_*|^\gamma \partial_{x_k}\epsM_j(v_*)\dd v_* \dd\sigma 
\\[2mm] = &\ \eps\sum_{j=1}^N  \int_{\R^3\times\Sf}b_{ij}(\cos\vartheta)|v-v_*|^\gamma \pabb{\partial_{x_k} \tilde{c}_j+m_j c_j (v_*-\eps u_j)\cdot\partial_{x_k} u_j}{\mathcal{M}_j^\eps}(v_*)\dd v_* \dd\sigma\hspace{-3cm}
\\[4mm] = &\ \eps\tilde{\nu}_i^\eps(t,x,v),
\end{split}
\end{equation*}
hence $\tilde{\nu}_i^\eps$ and $\epsnu_i$ have the same structure. In particular, as done for $\epsnu_i$, it is easy to show that $\nabla_v \tilde{\nu}_i^\eps\in L_v^\infty(\R^3)$ and that the following upper bound, similar to $\eqref{Upper+Lower}$, is satisfied for all~$v\in\R^3$:
\begin{multline}\label{UpperTildenu}
 \tilde{\nu}_i^\eps(v)\leq C_i(\delta) \exp\br{\max_{1\leq i\leq N} m_i\frac{7+\delta}{2(1-\delta)}\eps^2\norm{\boldu}_{L^\infty_t L_x^\infty}^2}
 \\[2mm]    \times \pabb{\norm{\partial_{x_k}\tilde{\boldc}}_{L^\infty_t L_x^\infty}+\norm{\boldc}_{L^\infty_tvL_x^\infty}\pab{1 + \norm{\boldu}_{L^\infty_t L_x^\infty}}\norm{\partial_{x_k}\boldu}_{L^\infty_tL_x^\infty}}  \scalprod{v}^\gamma.
\end{multline}
Note that, in this case, we are no more able to prove a lower bound for $ \tilde{\nu}_i^\eps$, as it was done in $\eqref{Upper+Lower}$ for $\nu_i^\eps$. From this lack of positivity comes the last term in the estimate. In fact, using again Young's inequality, we easily see that $I_2$ can be estimated as
\begin{eqnarray*}
I_2&=&\eps\sum_{i=1}^N \pa{\int_{\T^3\times\R^3}\partial_{v_\ell}\tilde{\nu}_i^\eps f_i\partial_{v_\ell}\partial_{x_k}f_i\mu_i^{-1}\dd x\dd v + \int_{\T^3\times\R^3} \tilde{\nu}_i^\eps\partial_{v_\ell}f_i \partial_{v_\ell}\partial_{x_k}f_i \mu_i^{-1}\dd x\dd v}                        
\\[3mm]&\geq& \eps \sum_{i=1}^N \left(-\frac{1}{2}\int_{\T^3\times\R^3}\frac{(\partial_{v_\ell}\tilde{\nu}_i^\eps)^2}{\epsnu_i}f_i^2 \mu_i^{-1}\dd x\dd v-\frac{1}{2}\int_{\T^3\times\R^3}\tilde{\nu}_i^\eps(\partial_{v_\ell}f_i)^2 \mu_i^{-1}\dd x\dd v\right.
\\[2mm]&&\hspace{6cm}  \left. - \frac{1}{2}\int_{\T^3\times\R^3} (\epsnu_i+\tilde{\nu}_i^\eps)|\partial_{v_\ell}\partial_{x_k}f_i|^2\mu_i^{-1}\dd x\dd v\right) 
\\[5mm]&\geq&  -\eps C^{\NuNu}_4\norm{\partial_{v_\ell}\partial_{x_k}\boldf}_{\spacexv}^2 -\eps C^{\NuNu}_6 \norm{\boldf}_{\sobolevxv{1}}^2 - \eps C^{\NuNu}_7\norm{\partial_{v_\ell}\boldf}_{\spacexv},
\end{eqnarray*}
where $C^{\NuNu}_4$, $C^{\NuNu}_7$ are easily recovered from $\eqref{Upper+Lower}$ and $\eqref{UpperTildenu}$, and
$$C^{\NuNu}_6=\max_{1\leq i\leq N}\sup_{v\in\R^3}\frac{(\partial_{v_l}\tilde{\nu}_i^\eps)^2}{2\epsnu_i} >0.$$
Gathering the estimates for $I_1$ and $I_2$ concludes the proof.

\subsection{Proof of Lemma \ref{LemmaK}}

We prove the result in the simplest case where $|\alpha|=0$ and $|\beta|=1$. The extension to the general case is again straightforward. We emphasize that compared to the usual case when the Maxwellian only depends on~$v$,~$x$-derivatives of $\bepsM$ generate new terms involving products of the form $\eps \partial^{\alpha'}_x \mathbf{u} \partial^{\alpha}_x \mathbf{\tilde{c}}$, displaying a lower order in $\eps$.

\medskip
We want to estimate the term
$$\scalprod{\nabla_v \bepsK(\boldf),\nabla_v \boldf}_{\hilbertxv}=\sum_{i=1}^N \int_{\T^3\times \R^3}\nabla_v \epsK_i(\boldf) \nabla_v f_i\mu_i^{-1} \dd x \dd v,$$
where, for any $1\leq i\leq N$, $\epsK_i$ is written under its kernel form (explicitly recovered in Appendix \ref{App:Carleman}) as
$$\epsK_i(\boldf)(v)=\sum_{j=1}^N\br{\int_{\R^3}\kappa_{ij}^{(1)}(v,v_*)\fs_j\dd v_* + \int_{\R^3}\kappa_{ij}^{(2)}(v,v_*)\fs_i\dd v_* - \int_{\R^3}\kappa_{ij}^{(3)}(v,v_*)\fs_j\dd v_*}.$$
For any fixed $i,j$, we shall treat separately the three kernels.

\bigskip
\noindent \textbf{Step 1 - Estimates for $\bm{\kappa_{ij}^{(1)}}$.} Starting from $\kappa_{ij}^{(1)}$, we recall that its full explicit expression is given in Appendix \ref{App:Carleman} by $\eqref{kappa1}$. We compute its $v$-derivatives and immediately note that the singularity of the operator coming from $|v-v_*|$ imposes to introduce a proper splitting (depending on a small parameter) in order to deal with a smooth part that we can derive infinitely many times and a remainder part (taking into account the region where $v-v_*$ is close to 0) which goes to zero with the small parameter. Let us relabel the relative velocity $\eta =v-v_*$ and choose $v$ and $\eta$ as new variables of $\kappa_{ij}^{(1)}$, rewriting $v_*=v-\eta$. Thus, denoting $\bar{R}=R(v,\eta)$ and $\bar{O}=O(v,\eta)$ the quantities
\begin{equation*}
\begin{split}
\bar{R} &= \frac{m_j}{|m_i-m_j|}|\eta|,\\[3mm]
\bar{O} &= \frac{m_i}{m_i-m_j} v - \frac{m_j}{m_i-m_j}(v-\eta),
\end{split}
\end{equation*}
defined for $i\neq j$, the operator $\kappa_{ij}^{(1)}$ reads
\begin{multline*}
\kappa_{ij}^{(1)}(v,\eta)=C_{ij} c_i |\eta|^{\gamma}\int_{\Sf}\frac{b_{ij}\pa{\omega\cdot\frac{\eta}{|\eta|}}}{\abs{\frac{m_i^2+m_j^2}{(m_i-m_j)^2}+\frac{2m_i m_j}{(m_i-m_j)|m_i-m_j|}\pa{\omega\cdot\frac{\eta}{|\eta|}}}^{1-\gamma}}
\\[2mm]    \times e^{-\frac{m_i}{2}\abs{\bar{R}\omega +\bar{O}}^2+\eps m_i\pab{\bar{R}\omega +\bar{O}}\cdot u_i-\eps^2\frac{m_i}{2}|u_i|^2}\ \dd \omega.        
\end{multline*}
In particular,
\begin{equation*}
b_{ij}\pa{\omega\cdot\frac{\eta}{|\eta|}}=b_{ij}\pa{\frac{\frac{2m_im_j}{(m_i-m_j)^2}+\frac{m_i^2+m_j^2}{(m_i-m_j)|m_i-m_j|}\pa{\omega\cdot\frac{\eta}{|\eta|}}}{\frac{m_i^2+m_j^2}{(m_i-m_j)^2}+\frac{2m_im_j}{(m_i-m_j)|m_i-m_j|}\pa{\omega\cdot\frac{\eta}{|\eta|}}}}, 
\end{equation*}
and  
\begin{equation*}
\begin{split}
&\abs{\bar{R}\omega +\bar{O}}^2 = \frac{2m_j^2}{(m_i-m_j)^2}|\eta|^2 +\frac{2 m_j|\eta|}{|m_i-m_j|}(v\cdot\omega)+\frac{2 m_j^2|\eta|}{(m_i-m_j)|m_i-m_j|}(\eta\cdot\omega)
\\[2mm]   & \hspace{2.3cm} + |v|^2+\frac{2m_j}{m_i-m_j}(v\cdot\eta),
\\[4mm]    & \pa{\bar{R}\omega +\bar{O}}\cdot u_i = \frac{m_j}{|m_i-m_j|}|\eta|(\omega\cdot u_i) + v\cdot u_i + \frac{m_j}{m_i-m_j}(\eta\cdot u_i).
\end{split}
\end{equation*}
The operator $\kappa_{ij}^{(1)}(v,\eta)$ is clearly regular with respect to the variable $v$, thus we only need to ensure that we can take derivatives also with respect to $\eta$. For some small parameter $\xi >0$, let us introduce a mollified indicator function $\mathds{1}_{\br{|\cdot|\geq \xi}}$ and define the splitting
$$\kappa_{ij}^{(1)}=\kappa_{ij}^{(1),S}+\kappa_{ij}^{(1),R},$$ 
where we have called the smooth part
$$\kappa_{ij}^{(1),S}(v,\eta)= \mathds{1}_{\br{|\eta|\geq \xi}} \kappa_{ij}^{(1)}(v,\eta),$$
and $\kappa_{ij}^{(1),R}$ is the remainder. It is straightforward to check that $\kappa_{ij}^{(1),S}(v,\eta)$ is smooth in both variables $v$ and $\eta$, and that $\kappa_{ij}^{(1),R}$ is integrable with respect to $\eta$ near 0. More precisely, one can prove
 that there exists $\bar{\delta}\in (0,1)$, depending on $\norm{\boldu}$ and $\norm{\boldc}$ but not on $\eps$, such that $\kappa_{ij}^{(1),R}$ is uniformly bounded with respect to $v$, namely
$$\kappa_{ij}^{(1),R}(v,\eta) \sqrt{\frac{\mu_j(v-\eta)}{\mu_i(v)}}\leq C_{ij}(\delta) |\eta|^\gamma e^{-C(\delta)|\eta|^2},\quad \forall v\in\R^3,$$
where $C_{ij}(\delta), C(\delta) >0$ are two explicitly computable constants depending on an arbitrary parameter $\delta\in (\bar{\delta},1)$. This last fact allows to easily deduce that 
\begin{equation}\label{Estimate1}
\norm{\kappa_{ij}^{(1),R}(v,\eta)\sqrt{\frac{\mu_j(v-\eta)}{\mu_i(v)}}}_{L_{\eta}^1} \leq C^{(1)}_{ij}(\delta) \xi^{\gamma+3},  \quad \forall v\in\R^3.
\end{equation}

\noindent Now, setting
$$W_{ij}^{(1)}\pa{\omega\cdot\frac{\eta}{|\eta|}}= \abs{\frac{m_i^2+m_j^2}{(m_i-m_j)^2}+\frac{2m_i m_j}{(m_i-m_j)|m_i-m_j|}\pa{\omega\cdot\frac{\eta}{|\eta|}}}^{\gamma-1}$$
to ease the computations, we can take the $\eta$-derivative of $\kappa_{ij}^{(1),S}$ to obtain
\begin{equation*}
\begin{split}
\nabla_\eta \kappa_{ij}^{(1),S}(v,\eta) =&\  C_{ij} \nabla_\eta \mathds{1}_{\br{|\eta|\geq \xi}} |\eta|^{\gamma} \int_{\Sf} b_{ij}W_{ij}^{(1)}\epsM_i(\bar{R}\omega+\bar{O})\dd \omega
\\[2mm]&\ +C_{ij} \mathds{1}_{\br{|\eta|\geq \xi}} \nabla_\eta |\eta|^{\gamma} \int_{\Sf} b_{ij}W_{ij}^{(1)}\epsM_i(\bar{R}\omega+\bar{O})\dd \omega
\\[2mm]&\ +C_{ij}\mathds{1}_{\br{|\eta|\geq \xi}} |\eta|^{\gamma} \int_{\Sf} \nabla_\eta b_{ij}W_{ij}^{(1)}\epsM_i(\bar{R}\omega+\bar{O})\dd \omega
\\[2mm]&\ +C_{ij} \mathds{1}_{\br{|\eta|\geq \xi}} |\eta|^{\gamma} \int_{\Sf} b_{ij}\nabla_\eta W_{ij}^{(1)}\epsM_i(\bar{R}\omega+\bar{O})\dd \omega 
\\[2mm]&\ +C_{ij} \mathds{1}_{\br{|\eta|\geq \xi}}  |\eta|^{\gamma} \int_{\Sf} b_{ij}W_{ij}^{(1)}\nabla_\eta \epsM_i(\bar{R}\omega+\bar{O})\dd \omega
\\[4mm]:=&\ I^S_{1,1} + I^S_{1,2} + I^S_{1,3} + I^S_{1,4} + I^S_{1,5}.
\end{split}
\end{equation*}
Direct computations yield the four derivatives
\begin{itemize}
\item[(i)]  $\disp \nabla_\eta |\eta|^{\gamma}   = \gamma|\eta|^{\gamma-1}\frac{\eta}{|\eta|},$  \\[3mm]  
\item[(ii)] $\disp \nabla_\eta b_{ij}\pa{\omega\cdot\frac{\eta}{|\eta|}}   = b_{ij}^{\prime}\pa{\omega\cdot\frac{\eta}{|\eta|}}\frac{\frac{(m_i+m_j)^2}{(m_i-m_j)|m_i-m_j|}\pa{\frac{\omega}{|\eta|}-\frac{\eta\cdot\omega}{|\eta|^2}\frac{\eta}{|\eta|}}}{\pa{\frac{m_i^2+m_j^2}{(m_i-m_j)^2}+\frac{2m_im_j}{(m_i-m_j)|m_i-m_j|}\pa{\omega\cdot\frac{\eta}{|\eta|}}}^2},$  \\[3mm] 
\item[(iii)]  $\disp \nabla_\eta W_{ij}^{(1)}\pa{\omega\cdot\frac{\eta}{|\eta|}}   = \frac{\frac{2(\gamma-1)m_i m_j}{(m_i-m_j)|m_i-m_j|}\pa{\frac{\omega}{|\eta|}-\frac{\eta\cdot\omega}{|\eta|^2}\frac{\eta}{|\eta|}}}{\abs{\frac{m_i^2+m_j^2}{(m_i-m_j)^2}+\frac{2m_i m_j}{(m_i-m_j)|m_i-m_j|}\pa{\omega\cdot\frac{\eta}{|\eta|}}}^{2-\gamma}},$   \\[4mm]                             
\item[(iv)]    
$$\disp \nabla_\eta \epsM_i(\bar{R}\omega+\bar{O})   = -\frac{m_i}{2}\epsM_i(\bar{R}\omega+\bar{O})\pa{\nabla_\eta \abs{\bar{R}\omega+\bar{O}}^2-2\eps\nabla_\eta \pabb{(\bar{R}\omega +\bar{O})\cdot u_i}},$$
\end{itemize}
where
\begin{equation*}
\begin{split}
\hspace{-1.7cm}\nabla_\eta \abs{\bar{R}\omega+\bar{O}}^2 = &\ \frac{4m_j^2}{(m_i-m_j)^2}\eta+\frac{2m_j}{|m_i-m_j|}\frac{\eta}{|\eta|}(v\cdot\omega) +\frac{2m_j}{m_i-m_j}v 
\\[2mm]    &\  + \frac{2m_j^2}{(m_i-m_j)|m_i-m_j|}\pa{\frac{\eta}{|\eta|}(\eta\cdot\omega)+|\eta|\omega},
\end{split}
\end{equation*}
and
$$\nabla_\eta\pabb{(\bar{R}\omega+\bar{O})\cdot u_i}=\frac{m_j}{|m_i-m_j|}\frac{\eta}{|\eta|}(\omega\cdot u_i)+\frac{m_j}{m_i-m_j}u_i.\hspace{3.5cm}$$
Our aim is now to bound each term $I^S_{1,k}$ uniformly with respect to $v\in\R^3$. The estimate for~$I^S_{1,1}$ is straightforward, since $|b_{ij}|\leq C$ from Grad's cutoff assumption we have made on the collision kernels, and
$$\abs{\frac{m_i^2+m_j^2}{(m_i-m_j)^2}+\frac{2m_i m_j}{(m_i-m_j)|m_i-m_j|}\pa{\omega\cdot\frac{\eta}{|\eta|}}}^{\gamma-1}\leq \abs{\frac{m_i^2+m_j^2-2m_i m_j}{(m_i-m_j)^2}}^{\gamma-1}=1,$$
thanks to the hypothesis that $\gamma\in [0,1]$. Thus, $\big|W_{ij}^{(1)}\big|\leq 1$ and we can apply the upper bound $\eqref{UpperM}$ for the local Maxwellian $\epsM_i$  to initially infer that
$$\abs{I^S_{1,1}}\leq  C(\delta_1)c_i |\eta|^{\gamma}\int_{\Sf}\mathM_i^{\delta_1}(\bar{R}\omega+\bar{O})\dd \omega,$$
for any $\delta_1\in (0,1)$. Then, using the same estimates as in \cite[Lemma 5.1]{BriDau}, we recover the bound
$$\abs{I^S_{1,1}}\leq  C(\delta_1)c_i |\eta|^{\gamma}e^{-\delta_1 \lambda(m_i,m_j)|\eta|^2-\delta_1\lambda(m_i,m_j)\frac{\abs{|\eta|^2-2(v\cdot\eta)}^2}{|\eta|^2}}\sqrt{\frac{\mathM_i^{\delta_1}(v)}{\mathM_j^{\delta_1}(v-\eta)}}.$$
For $I^S_{1,2}$, $I^S_{1,3}$ and $I^S_{1,4}$, we obtain a similar result, since
\begin{equation*}
\begin{split}
\nabla_\eta |\eta|^\gamma\leq \gamma|\eta|^{\gamma-1},\qquad  \nabla_\eta b_{ij} \leq C_{ij}|\eta|^{-1}\quad \textrm{ and also }\quad \nabla_\eta W_{ij}^{(1)}\leq C_{ij}|\eta|^{-1},
\end{split}
\end{equation*}
thanks to the assumption that $\big|b_{ij}^{\prime}(\cos\vartheta)\big|\leq C$. Thus, for $k\in\br{2,3,4}$, we deduce that 
$$\abs{I^S_{1,k}}\leq  C(\delta_k)c_i |\eta|^{\gamma-1}e^{-\delta_k \lambda(m_i,m_j)\pa{|\eta|^2-\frac{\abs{|\eta|^2-2(v\cdot\eta)}^2}{|\eta|^2}}}\sqrt{\frac{\mathM_i^{\delta_k}(v)}{\mathM_j^{\delta_k}(v-\eta)}},$$
for any $\delta_k\in (0,1)$, $k=2,3,4$. A bit different is the last term $I^S_{1,5}$. We observe that
$$\nabla_\eta \abs{\bar{R}\omega+\bar{O}}^2\leq  C_{ij}\pab{|\eta| + |v|} \quad\mbox{ and }\quad \nabla_\eta\pabb{(\bar{R}\omega+\bar{O})\cdot u_i}\leq C_{ij}|u_i|,$$
thus

$$\nabla_\eta \epsM_i(\bar{R}\omega+\bar{O})\leq C_{ij}(\delta_5)c_i \pab{|\eta| + |v| +\eps|u_i|}\mathM_i^{\delta_5}(\bar{R}\omega+\bar{O}),$$
for any $\delta_5\in (0,1)$. Therefore, since $\eps\leq 1$, we can upper bound the last term as
$$ \abs{I^S_{1,5}}\leq  C(\delta_5)c_i |\eta|^{\gamma}\pab{|\eta|+|v|+|u_i|}e^{-\delta_5 \lambda(m_i,m_j)\pa{|\eta|^2+\frac{\abs{|\eta|^2-2(v\cdot \eta)}^2}{|\eta|^2}}}\sqrt{\frac{\mathM_i^{\delta_5}(v)}{\mathM_j^{\delta_5}(v-\eta)}}.$$
Gathering the estimates on $I^S_{1,k}$, $k\in\br{1,\ldots,5}$, and, in particular, carefully choosing the parameter~${\delta_5\in (0,1)}$ in order to control the term in $|v|$, the computations carried out in~\cite{BonBouBriGre} directly apply and provide the uniform bound     
\begin{multline}\label{Uniform eta 1 S}
\hspace{-0.3cm}\abs{\nabla_\eta \kappa_{ij}^{(1),S}(v,\eta)}\sqrt{\frac{\mu_j(v-\eta)}{\mu_i(v)}}
\\[2mm]    \leq C_{ij}(\delta)\sqrt{\frac{\bar{c}_j}{\bar{c}_i}} \norm{\boldc}_{L^\infty_t L_x^\infty}\pa{1+ \norm{\boldu}_{L^\infty_t L_x^\infty}}\pab{1 + |\eta|^{\gamma-1} + |\eta|^{\gamma+1} + |\eta|^\gamma} e^{-C(\delta)|\eta|^2},
\end{multline} 
holding for any $v\in\R^3$ and any $\delta\in (\bar{\delta},1)$, where $\bar{\delta}\in (0,1)$ is fixed independently of $\eps$.

\smallskip
In a similar way, we can estimate the $v$-derivatives of $\kappa_{ij}^{(1),S}$ and $\kappa_{ij}^{(1),R}$. It is even simpler in this case, since the only dependence on $v$ inside these operators appears in the~Maxwellian $\epsM_i(\bar{R}\omega+\bar{O})$. Looking at $\nabla_v \kappa_{ij}^{(1),S}$ (for $\nabla_v \kappa_{ij}^{(1),R}$ the same computations will apply), we easily get
\begin{equation*}
\nabla_v \kappa_{ij}^{(1),S}(v,\eta) =  C_{ij} \mathds{1}_{\br{|\eta|\geq \xi}} |\eta|^{\gamma} \int_{\Sf} b_{ij}W_{ij}^{(1)} \nabla_v \epsM_i(\bar{R}\omega+\bar{O})\dd \omega,
\end{equation*}
where this time
\begin{equation*}
\nabla_v \epsM_i(\bar{R}\omega+\bar{O}) = -\frac{m_i}{2}\epsM_i(\bar{R}\omega+\bar{O})\pa{\frac{2m_j}{|m_i-m_j|}|\eta|\omega + 2v + \frac{2m_j}{m_i-m_j}\eta - 2\eps u_i}. 
\end{equation*}
Therefore, repeating the previous considerations and thanks to the straightforward upper bound
\begin{equation*}
\nabla_v \epsM_i(\bar{R}\omega+\bar{O}) \leq C_{ij}(\delta) c_i \pab{|\eta| + |v| + |u_i|} \mathcal{M}^\delta_i(\bar{R}\omega + \bar{O}),
\end{equation*}
valid for any $\delta\in (0,1)$, we recover the uniform estimate
\begin{multline}\label{Uniform v 1 S}
\hspace{-0.3cm}\abs{\nabla_v \kappa_{ij}^{(1),S}(v,\eta)}\sqrt{\frac{\mu_j(v-\eta)}{\mu_i(v)}}
\\[2mm]     \leq C_{ij}(\delta)\sqrt{\frac{\bar{c}_j}{\bar{c}_i}} \norm{\boldc}_{L^\infty_t L_x^\infty}\pa{1+ \norm{\boldu}_{L^\infty_t L_x^\infty}}\pab{1+|\eta|^\gamma+|\eta|^{\gamma+1}} e^{-C(\delta)|\eta|^2},
\end{multline}
holding for all $v\in\R^3$ and for any $\delta\in (\bar{\delta},1)$, where $\bar{\delta}\in (0,1)$ is fixed independently of $\eps$, and can in particular be chosen in order to satisfy both $\eqref{Uniform v 1 S}$ and $\eqref{Uniform eta 1 S}$. Using the same arguments in order to derive an upper bound for $\nabla_v \kappa_{ij}^{(1),R}$, estimates \eqref{Uniform eta 1 S}--\eqref{Uniform v 1 S} finally imply that for all $v\in\R^3$ we have
\begin{equation}\label{Estimate2}
\begin{split}
\norm{\nabla_\eta \kappa_{ij}^{(1),S}(v,\eta)\sqrt{\frac{\mu_j(v-\eta)}{\mu_i(v)}}}_{L^1_\eta} & \leq C^{(1)}_{ij}(\delta,\xi),
\\[4mm]    \norm{\nabla_v \kappa_{ij}^{(1),S}(v,\eta)\sqrt{\frac{\mu_j(v-\eta)}{\mu_i(v)}}}_{L^1_\eta} & \leq C^{(1)}_{ij}(\delta,\xi),
\\[4mm]    \norm{\nabla_v \kappa_{ij}^{(1),R}(v,\eta)\sqrt{\frac{\mu_j(v-\eta)}{\mu_i(v)}}}_{L^1_\eta} & \leq C^{(1)}_{ij}(\delta,\xi).
\end{split}
\end{equation}

\medskip
\noindent \textbf{Step 2 - Estimates for $\bm{\kappa_{ij}^{(2)}}$.} We proceed in the same way as for $\kappa_{ij}^{(1)}$. We rename the relative velocity $\eta=v-v_*$, so that $v_*=v-\eta$, and the explicit form of $\kappa_{ij}^{(2)}$ given by $\eqref{kappa2}$ reads in the new configuration
\begin{multline*}
\kappa_{ij}^{(2)}(v,\eta)= \mathcal{P}_{ij}(v,\eta)  \int_{\R^2}b_{ij}\pa{\frac{\pa{\frac{m_i+m_j}{2m_j}}^2 |\eta|^2 - |X|^2}{\pa{\frac{m_i+m_j}{2m_j}}^2 |\eta|^2 + |X|^2}}\pa{\pa{\frac{m_i+m_j}{2m_j}}^2 |\eta|^2 + |X|^2}^{\frac{\gamma-1}{2}} 
\\[3mm]     \times \epsM_j\pa{R\pa{\frac{\eta}{|\eta|},\pa{0,X+\frac{1}{2}\bar{X}}}} \dd X,
\end{multline*}
with
$$\mathcal{P}_{ij}(v,\eta)=\frac{C_{ji}}{|\eta|} e^{-\frac{m_i^2}{8m_j}|\eta|^2-\frac{m_j}{8}\frac{\abs{|\eta|^2-2(v\cdot\eta)}^2}{|\eta|^2}+\eps \frac{m_i}{2}(\eta\cdot u_j)-\eps\frac{m_j}{2}\pa{\frac{|\eta|^2-2(v\cdot\eta)}{|\eta|}\frac{\eta}{|\eta|}}\cdot u_j}\sqrt{\frac{\mu_i(v)}{\mu_i(v-\eta)}}.$$
Next, for some small parameter $\xi >0$, introduce a mollified indicator function $\mathds{1}_{\br{|\cdot|\geq \xi}}$ and define the splitting
$$\kappa_{ij}^{(2)}=\kappa_{ij}^{(2),S}+\kappa_{ij}^{(2),R},$$
where again we have called the smooth part
$$\kappa_{ij}^{(2),S}(v,\eta)= \mathds{1}_{\br{|\eta|\geq \xi}} \kappa_{ij}^{(2)}(v,\eta),$$
and $\kappa_{ij}^{(2),R}$ is the remainder. 
Following again \cite[Lemma 5.1]{BriDau} and the computations in \cite{BonBouBriGre}, a similar analysis to the one carried out for $\kappa_{ij}^{(1)}$ gives in this case
$$\kappa_{ij}^{(2),R}(v,\eta) \sqrt{\frac{\mu_i(v-\eta)}{\mu_i(v)}}\leq C_{ji}(\delta)|\eta|^{\gamma-2}e^{-C(\delta)|\eta|^2},\quad \forall v\in\R^3,$$
for some positive explicit constants $C_{ji}(\delta), C(\delta)$ depending on a parameter $\delta\in (\bar{\delta},1)$, where~$\bar{\delta}\in (0,1)$. This in turns means that we recover the needed control in $L^1_\eta$, which reads
\begin{equation}\label{Estimate3}
\norm{\kappa_{ij}^{(2),R}(v,\eta)\sqrt{\frac{\mu_i(v-\eta)}{\mu_i(v)}}}_{L_{\eta}^1} \leq C^{(2)}_{ij}(\delta) \xi^{\gamma+1},  \quad \forall v\in\R^3.
\end{equation}
Next, we consider some $\eta$-derivatives of $\kappa_{ij}^{(2),S}$. Naming this time
$$W_{ij}^{(2)}(|\eta|^2,|X|^2)=\pa{\pa{\frac{m_i+m_j}{2m_j}}^2 |\eta|^2 + |X|^2}^{\frac{\gamma-1}{2}},$$
direct computations show that 
%
%
%
%
%
\begin{itemize}
\item[(i)]  $\disp \nabla_\eta\mathcal{P}_{ij}(v,\eta)\leq  C_{ji} \pa{1+ \norm{\boldu}_{L^\infty_{t,x}}}\pa{|v| + |v|^2 + |v|^4 + \sum_{k=-2}^1|\eta|^k} \mathcal{P}_{ij}(v,\eta),$  \\[2mm]
\item[(ii)]      $\disp \nabla_\eta b_{ij}(|\eta|^2,|X|^2)\leq C_{ji} \frac{|\eta|^2+|X|^4}{|\eta|^4},$   \\[2mm]
\item[(iii)]   $\disp W_{ij}^{(2)}(|\eta|^2,|X|^2)\leq C_{ji}|\eta|^{\gamma-1},$   \\[2mm]
\item[(iv)]   $\disp \nabla_\eta W_{ij}^{(2)}(|\eta|^2,|X|^2)\leq C_{ji} |\eta|^{\gamma-2},$  
\end{itemize}
and at last that
$$\nabla_{\eta}\abs{R\pa{\frac{\eta}{|\eta|},\pa{0,X+\frac{1}{2}\bar{X}}}-\eps u_j}^2 \leq C_{ji}\frac{1+\norm{\boldu}_{L^\infty_{t,x}}^2+|X|^2}{|\eta|}.$$
By gathering these estimates together, we derive the following upper bound for $\nabla_\eta \kappa_{ij}^{(2),S}$:
\begin{equation*}
\abs{\nabla_\eta \kappa_{ij}^{(2),S}(v,\eta)}\sqrt{\frac{\mu_i(v-\eta)}{\mu_i(v)}}\leq  C_{ji}(\delta) \norm{\boldc}_{L^\infty_{t,x}}\pa{1+ \norm{\boldu}_{L^\infty_{t,x}}}\pa{1+\sum_{k=0}^4 |\eta|^{\gamma-k}}e^{-C(\delta)|\eta|^2},
\end{equation*}
holding for all $v\in\R^3$ and for any $\delta\in (\bar{\delta},1)$, where $\bar{\delta}\in (0,1)$ is fixed independently of $\eps$. Then, we can obtain a similar control for $\nabla_v \kappa_{ij}^{(2),S}$, which reads 
\begin{equation*}
\abs{\nabla_v \kappa_{ij}^{(2),S}(v,\eta)}\sqrt{\frac{\mu_i(v-\eta)}{\mu_i(v)}} \leq C_{ji}(\delta) \norm{\boldc}_{L^\infty_{t,x}}\pa{1+\eps \norm{\boldu}_{L^\infty_{t,x}}}\pa{1+|\eta|^\gamma+|\eta|^{\gamma-1}}e^{-C(\delta)|\eta|^2},
\end{equation*}
and the same holds for the remainder term $\nabla_v \kappa_{ij}^{(2),R}$. The collection of these inequalities finally ensure the expected bounds in $L^1_\eta$
\begin{equation}\label{Estimate4}
\begin{split}
\norm{\nabla_\eta \kappa_{ij}^{(2),S}(v,\eta)\sqrt{\frac{\mu_i(v-\eta)}{\mu_i(v)}}}_{L^1_\eta} & \leq C^{(2)}_{ij}(\delta,\xi),
\\[4mm]    \norm{\nabla_v \kappa_{ij}^{(2),S}(v,\eta)\sqrt{\frac{\mu_i(v-\eta)}{\mu_i(v)}}}_{L^1_\eta} & \leq C^{(2)}_{ij}(\delta,\xi),
\\[4mm]    \norm{\nabla_v \kappa_{ij}^{(2),R}(v,\eta)\sqrt{\frac{\mu_i(v-\eta)}{\mu_i(v)}}}_{L^1_\eta} & \leq C^{(2)}_{ij}(\delta,\xi).
\end{split}
\end{equation}

\medskip
\noindent \textbf{Step 3 - Estimates for $\bm{\kappa_{ij}^{(3)}}$.} The case of $\kappa_{ij}^{(3)}$ is straightforward and we do not need any additional effort. Introducing as before the splitting $\kappa_{ij}^{(3)}=\kappa_{ij}^{(3),S}+\kappa_{ij}^{(3),R}$ and renaming $\eta=v-v_*$, $v_*=v-\eta$, very easy computations provide the bound for the remainder $\kappa_{ij}^{(3),R}$
\begin{equation}\label{Estimate5}
\norm{\kappa_{ij}^{(3),R}(v,\eta)\sqrt{\frac{\mu_j(v-\eta)}{\mu_i(v)}}}_{L_{\eta}^1} \leq C^{(3)}_{ij}(\delta) \xi^{\gamma+3},
\end{equation}
and the estimates for the derivatives
\begin{equation}\label{Estimate6}
\begin{split}
\norm{\nabla_\eta \kappa_{ij}^{(3),S}(v,\eta)\sqrt{\frac{\mu_j(v-\eta)}{\mu_i(v)}}}_{L^1_\eta} & \leq C^{(3)}_{ij}(\delta,\xi),
\\[4mm]   \norm{\nabla_v \kappa_{ij}^{(3),S}(v,\eta)\sqrt{\frac{\mu_j(v-\eta)}{\mu_i(v)}}}_{L^1_\eta} & \leq C^{(3)}_{ij}(\delta,\xi),  
\\[4mm]    \norm{\nabla_v \kappa_{ij}^{(3),R}(v,\eta)\sqrt{\frac{\mu_j(v-\eta)}{\mu_i(v)}}}_{L_{\eta}^1} & \leq C^{(3)}_{ij}(\delta,\xi),
\end{split}
\end{equation}
for all $v\in\R^3$. 

\medskip
\noindent \textbf{Step 4 - Conclusion.} Now, we can successively write 
\begin{equation*}
\begin{split}
&\scalprod{\nabla_v \bepsK(\boldf),\nabla_v \boldf}_{\hilbertxv} 
\\[2mm]  &\quad= \sum_{i,j=1}^N \int_{\T^3\times\R^3}  \pa{\nabla_v \int_{\R^3}\kappa_{ij}^{(1)}(v,v_*)f^*_j + \kappa_{ij}^{(2)}(v,v_*)f^*_i - \kappa_{ij}^{(3)}(v,v_*)f^*_j\dd v_*} \nabla_v f_i \mu_i^{-1}\dd x\dd v 
\\[2mm] &\quad  =    \sum_{i,j=1}^N \int_{\T^3\times\R^3}  \pa{\nabla_v \int_{\R^3}\pa{\kappa_{ij}^{(1),S}(v,\eta)+\kappa_{ij}^{(1),R}(v,\eta)}f_j(v-\eta)\dd\eta}\nabla_v f_i\mu_i^{-1}\dd x\dd v 
\\[1mm]    &\quad\quad\quad   + \sum_{i,j=1}^N \int_{\T^3\times\R^3}  \pa{\nabla_v \int_{\R^3}\pa{\kappa_{ij}^{(2),S}(v,\eta)+\kappa_{ij}^{(2),R}(v,\eta)}f_i(v-\eta)\dd\eta}\nabla_v f_i\mu_i^{-1}\dd x\dd v
\\[1mm]    &\quad\quad\quad     + \sum_{i,j=1}^N \int_{\T^3\times\R^3}  \pa{\nabla_v \int_{\R^3}\pa{\kappa_{ij}^{(3),S}(v,\eta)+\kappa_{ij}^{(3),R}(v,\eta)}f_j(v-\eta)\dd\eta}\nabla_v f_i\mu_i^{-1}\dd x\dd v 
\\[4mm]  &\ \   :=\mathcal{I}_1 + \mathcal{I}_2 +\mathcal{I}_3.
\end{split}
\end{equation*}
The behaviour of the three terms is the same. We shall thus focus on the first one and the others will be estimated in an equivalent way. Since
$$\nabla_v f_j(v-\eta)=-\nabla_\eta f_j(v-\eta),$$
we initially apply integration by parts to deduce that
\begin{equation*}
\begin{split}
\nabla_v \int_{\R^3}\Big(\kappa_{ij}^{(1),S} &(v,\eta)+\kappa_{ij}^{(1),R}(v,\eta)\Big)f_j(v-\eta)\dd\eta 
\\[2mm]   &\quad = \int_{\R^3}\pa{\nabla_v \kappa_{ij}^{(1),S}(v,\eta)-\nabla_\eta \kappa_{ij}^{(1),S}(v,\eta)+\nabla_v \kappa_{ij}^{(1),R}(v,\eta)}f_j(v-\eta)\dd\eta  
\\[2mm]   &\quad\quad\ +\int_{\R^3} \kappa_{ij}^{(1),R}(v,\eta)\nabla_v f_j(v-\eta)\dd\eta.
\end{split}
\end{equation*}
Thus, if we multiply and divide inside the integral by the factor $\sqrt{\mu_j(v-\eta)/\mu_i(v)}$, we can apply Cauchy-Schwarz inequality and an $L^1/L^2$ convolution inequality, together with estimates $\eqref{Estimate1}$ and $\eqref{Estimate2}$, to finally deduce that
$$\mathcal{I}_1 \leq C_1^{(1)}(\delta,\xi)\norm{\boldf}_{\hilbertxv}^2+\xi C_2^{(1)}(\delta)\norm{\nabla_v \boldf}_{\hilbertxv}^2,$$
for any $\xi \in (0,1)$.

\smallskip
As already mentioned, the same argument applies to $\mathcal{I}_2$ and $\mathcal{I}_3$ using estimates \eqref{Estimate3}--\eqref{Estimate4}--\eqref{Estimate5}--\eqref{Estimate6} and for some positive constants $C_k^{(2)}, C_k^{(3)}$, $k=1,2$. Gathering the bounds for the three terms concludes the proof.

\subsection{Proof of Lemma \ref{LemmaL}}

The proof follows exactly the computations presented in \cite{BonBouBriGre}, except for the final estimate where Young's inequality allows to refine the control on $\pi_{\boldL}(\boldf)$. Let us sketch the idea.

\smallskip
We initially rewrite $\bepsL$ componentwise as
\begin{equation*}
\begin{split}
\epsL_i & (\boldf) =
\\[1mm]  & \sum_{j=1}^N \int_{\R^3\times\Sf}B_{ij}(v,v_*,\vartheta)\Big(c_{i,\infty}{\mathcal{M}_i^\eps}^{\prime}\fps_j+c_{j,\infty}{\mathcal{M}_j^\eps}^{\prime *}\fp_i - c_{i,\infty}{\mathcal{M}_i^\eps}\fs_j-c_{j,\infty}{\mathcal{M}_j^\eps}^*f_i\Big)\dd v_* \dd\sigma        
\\[2mm]      & \qquad\ +\eps   \sum_{j=1}^N \int_{\R^3\times\Sf}B_{ij}(v,v_*,\vartheta)\pabb{\tilde{c}_i{\mathcal{M}_i^\eps}^{\prime}\fps_j+\tilde{c}_j{\mathcal{M}_j^\eps}^{\prime *}\fp_i-\tilde{c}_i{\mathcal{M}_i^\eps}\fs_j-\tilde{c}_j{\mathcal{M}_j^\eps}^*f_i}\dd v_* \dd\sigma
\\[6mm]    &\quad:= L_{i,\infty}^{\eps}(\boldf)+\eps\tilde{L}_i^{\eps}(\boldf). 
\end{split}
\end{equation*}
Now, we can study the Dirichlet form of $\bepsL$ by introducing the penalization with respect to $\boldL$. Recall the shorthand notation $\boldf^\perp = \boldf - \pi_\boldL(\boldf)$. Using the fact that $\bepsL$ shares the conservation properties of the Boltzmann operator $\boldQ$, we can successively write
\begin{equation*}
\begin{split}
\scalprod{\bepsL(\boldf),\boldf}_{\hilbertv} = &\ \big\langle \bepsL(\boldf),\boldf^\perp\big\rangle_{\hilbertv} 
\\[3mm]    = &\ \big\langle \boldL(\boldf),\boldf^\perp\big\rangle_{\hilbertv} + \big\langle\boldL_{\infty}^{\eps}(\boldf)-\boldL(\boldf),\boldf^\perp\big\rangle_{\hilbertv}
\\[2mm]    &\ + \eps\big\langle \tilde{\boldL}^{\eps}(\boldf),\boldf^\perp\big\rangle_{\hilbertv}.
\end{split}
\end{equation*}
We estimate each term separately. Thanks to \cite[Theorem 3.3]{BriDau}, the first term provides the spectral gap for $\boldL$ 
$$\big\langle \boldL(\boldf),\boldf^\perp\big\rangle_{\hilbertv}\leq -\lambda_{\boldL}\big\Vert\boldf^\perp\big\Vert_{\spacev}^2.$$
Recalling that $\boldL_{\infty}^{\eps}-\boldL=\mathcal{O}(\eps)$, the second term can be estimated following the strategy of \cite{BonBouBriGre}. We split $\boldf=\boldf^\perp+\pi_{\boldL}(\boldf)$ and we then apply Young's inequality with a positive constant $\eta_1/\eps$ to recover
\begin{multline*}
\big\langle\pab{\boldL_{\infty}^{\eps}-\boldL}(\boldf),\boldf^\perp\big\rangle_{\hilbertv} =  \big\langle\pab{\boldL_{\infty}^{\eps}-\boldL}(\boldf^\perp),\boldf^\perp\big\rangle_{\hilbertv} 
\\[2mm]   \hspace{5cm}+ \big\langle\pab{\boldL_{\infty}^{\eps}-\boldL}\pab{\pi_{\boldL}(\boldf)},\boldf^\perp\big\rangle_{\hilbertv} 
\\[2mm]  \leq \pab{\eps+\eta_1}C^\LL_2\big\Vert\boldf^\perp\big\Vert_{\spacev}^2+\eps^2 \norm{\boldu}_{L^\infty_{t,x}}\frac{C^\LL_2}{\eta_1}\norm{\pi_{\boldL}(\boldf)}_{\spacev}^2.\qquad
\end{multline*}
In a very similar way, using again Young's inequality with the same constant $\eta_1/\eps$ and increasing the constant $C^\LL_2$ if necessary, the third term provides 
\begin{equation*}
\begin{split}
\big\langle\tilde{\boldL}^\eps(\boldf),\boldf^\perp\big\rangle_{\hilbertv} = &\ \big\langle\tilde{\boldL}^\eps(\boldf^\perp),\boldf^\perp\big\rangle_{\hilbertv}+\big\langle\tilde{\boldL}^\eps(\pi_{\boldL}(\boldf)),\boldf^\perp\big\rangle_{\hilbertv}
\\[3mm]    \leq &\ \frac{\eta_1 C^\LL_2}{\eps}\big\Vert\boldf^\perp\big\Vert_{\spacev}^2 + \eps \norm{\tilde{\boldc}}_{L^\infty_{t,x}}\frac{C^\LL_2}{\eta_1}\norm{\pi_{\boldL}(\boldf)}_{\spacev}^2.
\end{split}
\end{equation*}
Gathering the three estimates, and using the Sobolev embedding of $H^{s/2}_x$ in $L^\infty_x$ to upper bound $\norm{\boldu}_{L^\infty_{t,x}}$ and $\norm{\tilde{\boldc}}_{L^\infty_{t,x}}$ by~$\delta_\ms$, we obtain the expected result. This ends the proof.

\subsection{Proof of Lemma \ref{LemmaQ}}
The first part comes from the conservation properties of the Boltzmann operator. In fact, we have seen that
$$\sum_{i=1}^N\int_{\R^3} Q_i(\boldg,\boldh)(v) \psi_i(v) \dd v=0,\quad \forall \boldg,\boldh\in \hilbertR,   $$
whenever $\pmb{\psi}\in\textrm{Span}\big(\mathbf{e}^{\One},\ldots,\mathbf{e}^{\Enn},v_1\mathbf{m},v_2\mathbf{m},v_3\mathbf{m},|v|^2\mathbf{m}\big)$ is a collision invariant of the mixture. In particular the above equality rewrites
$$\big\langle\boldQ(\boldg,\boldh),\pi_{\boldL}(\boldf)\big\rangle_{\hilbertv}=0,\quad \forall \boldf,\boldg,\boldh\in \hilbertR, $$
which is exactly $\eqref{Property6}$ for $\boldQ$. Then the property follows also for $\bepsL$, since by definition~$\bepsL(\boldf)=\boldQ(\bepsM,\boldf)+\boldQ(\boldf,\bepsM)$ and both terms are orthogonal to $\Ker\boldL$.

\medskip
The second part is an extension of the same property satisfied by the Boltzmann operator in the mono-species case. Let $s\in\N$ and $\alpha,\beta$ be fixed such that $|\alpha|+|\beta| = s$. We initially note that
$$\big\langle\multideriv \boldQ(\boldg,\boldh),\boldf\big\rangle_{L^2_{x,v}\big(\mu_i^{-\frac{1}{2}}\big)} = \sum_{i,j=1}^N \big\langle\multideriv Q_{ij}(g_i,h_j),f_i\big\rangle_{L^2_{x,v}\big(\mu_i^{-\frac{1}{2}}\big)}.$$
Since each $Q_{ij}$ acts like a mono-species Boltzmann operator, we can easily adapt the computations made in \cite[Appendix A]{Bri1} to prove that, for any $1\leq i,j\leq N$, there exist two nonnegative functionals $\mathcal{G}_{ij}^{x,v,s}$ and $\mathcal{G}_{ij}^{x,s}$ satisfying $\mathcal{G}_{ij}^{x,v,s+1}\leq \mathcal{G}_{ij}^{x,v,s}$, $\mathcal{G}_{ij}^{x,s+1}\leq \mathcal{G}_{ij}^{x,s}$ and such that
$$\big\langle\multideriv Q_{ij}(g_i,h_j),f_i\big\rangle_{L^2_{x,v}\big(\mu_i^{-\frac{1}{2}}\big)}\leq \left\{\begin{array}{cc}
\mathcal{G}_{ij}^{x,s}(g_i,h_j)\norm{f_i}_{L^2_{x,v}\big(\langle v \rangle^{\frac{\gamma}{2}}\mu_i^{-\frac{1}{2}}\big)} & \textrm{if }\ |\beta|=0,\\[7mm]
\mathcal{G}_{ij}^{x,v,s}(g_i,h_j)\norm{f_i}_{L^2_{x,v}\big(\langle v \rangle^{\frac{\gamma}{2}}\mu_i^{-\frac{1}{2}}\big)} & \textrm{if }\ |\beta|\geq 1.
\end{array}\right.$$
Moreover, there exists an integer $s_0 > 0$ such that, for all $s > s_0$, there exists an explicit constant $C^Q_{ij}>0$ verifying
\begin{multline*}
\mathcal{G}_{ij}^{x,s}(g_i,h_j) 
\\[1mm]     \leq C^Q_{ij}\pa{\norm{g_i}_{H_x^s L_v^2\pab{\mu_i^{-\frac{1}{2}}}}\norm{h_j}_{H_x^s L^2_v\pab{\langle v \rangle^{\frac{\gamma}{2}}\mu_j^{-\frac{1}{2}}}}+\norm{h_j}_{H_x^s L_v^2\pab{\mu_j^{-\frac{1}{2}}}}\norm{g_i}_{H^s_x L^2_v \pab{\langle v \rangle^{\frac{\gamma}{2}}\mu_i^{-\frac{1}{2}}}}},
\end{multline*}
and
\begin{multline*}
\mathcal{G}_{ij}^{x,v,s}(g_i,h_j)
\\[1mm]     \leq C^Q_{ij}\pa{\norm{g_i}_{H_{x,v}^s\pab{\mu_i^{-\frac{1}{2}}}}\norm{h_j}_{H_{x,v}^s\pab{\langle v \rangle^{\frac{\gamma}{2}}\mu_j^{-\frac{1}{2}}}}+\norm{h_j}_{H_{x,v}^s\pab{\mu_j^{-\frac{1}{2}}}}\norm{g_i}_{H^s_{x,v}\pab{\langle v \rangle^{\frac{\gamma}{2}}\mu_i^{-\frac{1}{2}}}}}.
\end{multline*}
From this, we immediately deduce that
$$\big\langle\multideriv \boldQ(\boldg,\boldh),\boldf\big\rangle_{\hilbertxv}\leq \left\{\begin{array}{cc}
\pa{\sum_{i,j=1}^N \mathcal{G}_{ij}^{x,s}(g_i,h_j)}\norm{\boldf}_{\spacexv} & \textrm{if }\ |\beta|=0,\\[9mm]
\pa{\sum_{i,j=1}^N \mathcal{G}_{ij}^{x,v,s}(g_i,h_j)}\norm{\boldf}_{\spacexv} & \textrm{if }\ |\beta|\geq 1,
\end{array}\right.$$
and defining $C^\QQ_s=N^2\disp \max_{1\leq i,j\leq N}C^Q_{ij}$ finishes the proof of \eqref{Property7}--\eqref{Property8}.

\subsection{Proof of Lemma \ref{LemmaS}}

We divide the proof into two separated parts for the sake of clarity. The first one deals with inequalities \eqref{Property9}--\eqref{Property10}, where at least one $v$-derivative is considered, since the proofs of the two estimates are almost the same. The second part treats pure spatial derivatives, which require a slightly different approach. Note that, throughout the proof, we shall carefully keep track of the parameter~$\delta_\ms^2$ which uniformly bounds the quantities $\norm{\tilde{\boldc}}_{L^\infty_t \spacex{s}}^2$ and $\norm{\boldu}_{L^\infty_t H^s_x}^2$. For simplicity, all the other factors multiplying the term of interest $\delta_\ms^2$ will always be hidden inside a suitable constant, even if a lower order (in $\delta_\ms$ or in $\eps$) appears. This choice will also help in enlightening the computations.

\bigskip 
\noindent \textbf{Step 1 - Estimates with velocity derivatives.}
The idea to recover the first estimate $\eqref{Property9}$ is to show that $\multideriv \bepsS=\mathcal{O}(\eps^{-1})$. For this, let us initially extract a power $\eps^{-1}$ from the source term and write
\begin{multline*}
\scalprod{\multideriv\bepsS,\multideriv \boldf}_{\hilbertxv}=\frac{1}{\eps}\Bigg(-\scalprod{\multideriv\pab{\partial_t\bepsM+\frac{1}{\eps}v\cdot\nabla_x \bepsM},\multideriv\boldf}_{\hilbertxv}
\\[2mm]      +\frac{1}{\eps^2}\scalprod{\multideriv\boldQ(\bepsM,\bepsM),\multideriv\boldf}_{\hilbertxv}\Bigg).
\end{multline*}
We shall treat the linear part and the nonlinear term separately.

Computing the time and space derivatives in the linear contribution, we note that, for any $1\leq i\leq N$,
\begin{multline}\label{Linear 1}
\partial_t \epsM_i+\frac{1}{\eps}v\cdot\nabla_x\epsM_i=\Bigg(\partial_t c_i +\eps m_i c_i (v-\eps u_i)\cdot \partial_t u_i
\\     +\sum_{k=1}^3 v_k\pabb{\eps\partial_{x_k}\tilde{c}_i+ m_i c_i (v-\eps u_i)\cdot \partial_{x_k}u_i}\Bigg){\mathcal{M}_i^\eps},\qquad
\end{multline}
and that the leading order seems to be $\mathcal{O}(1)$. Let us show that it is actually $\mathcal{O}(\delta_\ms)$. For this, recall that $(\boldc,\boldu)$ is the unique perturbative solution of the Maxwell-Stefan system \eqref{MS mass}--\eqref{MS momentum}--\eqref{MS incompressibility}. From the mass equation:
$$\partial_t \boldc = - \nabla_x\cdot(\boldc\boldu) = - \eps\nabla_x\tilde{\boldc} \cdot\boldu - \boldc\nabla_x\cdot\boldu$$
we easily deduce a bound for $\partial_t \boldc$. Rewriting the momentum equation $\eqref{MS momentum}$ in vectorial form:
$$\nabla_x \boldc = A(\boldc)\boldu,$$
we can derive an estimate for $\partial_t\boldu$ by taking the time derivative on both sides to obtain
$$A(\boldc)\partial_t\boldu= -\partial_t A(\boldc)\boldu +\partial_t(\nabla_x \boldc).$$
It is then straightforward to check that the right-hand side of the above equality satisfies
$$-\partial_t A(\boldc)\boldu +\partial_t(\nabla_x \boldc)\in (\Ker A)^\perp,$$
so that the relation can be inverted by applying $A(\boldc)^{-1}$, and the equation governing $\partial_t\boldu$ explicitly reads
\begin{equation*}
\partial_t\boldu = -A(\boldc)^{-1}\pabb{\partial_t A(\boldc)\boldu - \nabla_x \pab{- \eps\nabla_x\tilde{\boldc} \cdot\boldu - \boldc\nabla_x\cdot\boldu}}.
\end{equation*}
Noticing that one of the macroscopic quantities $\tilde{\boldc}$ and $\boldu$ (or their derivatives) always appears as a factor in each term of $\eqref{Linear 1}$, by means of the continuous Sobolev embedding of $H^{s/2}_x$ in $L^\infty_x$, $s>3$,  we can infer the existence of a polynomial $P$, such that, for any $1\leq i\leq N$ and for almost any~$(t,x,v)\in\R_+\times\T^3\times\R^3$,
\begin{equation*}
\abs{\partial_t \epsM_i(t,x,v)+\frac{1}{\eps}v\cdot\nabla_x\epsM_i(t,x,v)}^2  \leq \delta_\ms^2 C(\delta) P(v) \mathM_{i}^{2\delta}(v),
\end{equation*}
for a positive constant $C(\delta)$ only depending on an arbitrary parameter $\delta\in (0,1)$, which is chosen independently of $\eps$. Note in particular the factor $2$ (simply coming from the square power) in the exponent of $\mathM_i^{2\delta}$, which is crucial in order to compensate the weight $\mu_i^{-1}$ of the $\hilbertxv$ norm. Moreover, we again emphasize that the other terms containing $\norm{\boldc}_{L^\infty_t \spacex{s}}$ or $\norm{\boldu}_{L^\infty_t H^s_x}$ (for $s\leq 4$) have been simply collected inside the constant $C(\delta)$.

Now, taking $\alpha$ and $\beta$ derivatives of the linear term only increases the number of factors depending on polynomials of $v$, and (at most) $|\alpha|+4$ derivatives of $\boldc$ and $\boldu$. In particular it does not modify the leading order $\mathcal{O}(\eps^{-1})$, nor the presence of the multiplicative constant $\delta_\ms^2$. Therefore, the exact same arguments apply, and we deduce that there also exist a polynomial $P_L$ and a constant $C_L(\delta)>0$ such that, for any $1\leq i\leq N$ and for almost any~$(t,x,v)\in\R_+\times\T^3\times\R^3$,
\begin{equation}\label{Linear start}
\abs{\multideriv\pa{\partial_t \epsM_i(t,x,v)+\frac{1}{\eps}v\cdot\nabla_x\epsM_i(t,x,v)}}^2 \leq \delta_\ms^2 C_{L}(\delta) P_L (v) \mathM_{i}^{2\delta}(v),
\end{equation}
where again $\delta\in (0,1)$ is arbitrarily chosen, and the other dependencies on $\norm{\boldc}_{L^\infty_t H^{\abs{\alpha}+4}_x}$ and $\norm{\boldu}_{L^\infty_t H^{\abs{\alpha}+4}_x}$ have been hidden inside $C_L(\delta)$.

\vspace{1mm}
If we now fix $\delta\in (1/2,1)$ in $\eqref{Linear start}$, applying Young's inequality with a positive constant $\eta_2/\eps$ and recalling the that we use the notation $\mu_i=c_{i,\infty}\mathcal{M}_i$ for any $1\leq i\leq N$, we can finally estimate the linear part as
\begin{equation*}
\frac{1}{\eps}\abs{\scalprod{\multideriv\pab{\partial_t\bepsM+\frac{1}{\eps}v\cdot\nabla_x \bepsM},\multideriv\boldf}_{\hilbertxv}}\hspace{6cm}
\end{equation*}
\begin{equation} \label{LinearPart}
\begin{split}
\qquad  & \leq \frac{1}{\eta_2}\norm{\multideriv\pab{\partial_t\bepsM+\frac{1}{\eps}v\cdot\nabla_x \bepsM}}_{\hilbertxv}^2 +\frac{\eta_2}{\eps^2}\norm{\multideriv \boldf}_{\hilbertxv}^2
\\[3mm]  & \leq \frac{\delta_\ms^2 C_{L}(\delta)}{\eta_2} \sum_{i=1}^N\int_{\T^3\times\R^3} c_{i,\infty}^{-1} P_{L}(v)\mathM_i^{2\delta-1}(v) \dd x \dd v +\frac{\eta_2}{\eps^2}\norm{\multideriv \boldf}_{\hilbertxv}^2 
\\[5mm]  & \leq \frac{\delta_\ms^2 C_L(\delta)}{\eta_2\disp\min_{1\leq i\leq N} c_{i,\infty}} +\frac{\eta_2}{\eps^2}\norm{\multideriv \boldf}_{\hilbertxv}^2,  
\end{split}
\end{equation}
by simply increasing the constant $C_L(\delta)$ in order to include the value of the sum of the integrals, which are clearly finite since $2\delta-1 >0$ thanks to our choice of the parameter $\delta$.

\medskip
Let us continue with the analysis of the nonlinear term $\boldQ(\bepsM,\bepsM)$. For this, recall that the macroscopic velocity $\boldu$ is defined componentwise as $u_i=\bar{u}+\eps \tilde{u}_i$, and introduce the local Maxwellian distribution $\bmathM_{\bar{u}}=(\mathM_{\bar{u},1},\ldots,\mathM_{\bar{u},N})$, given for any $1\leq i\leq N$ by
$$\mathM_{\bar{u},i}(t,x,v)=\pa{\frac{m_i}{2\pi}}^{3/2}\exp\br{-\frac{m_i}{2}\abs{v-\eps \bar{u}(t,x)}^2},\quad t\geq 0,\ x\in\T^3,\ v\in\R^3.$$
Since the bulk velocity $\bar{u}$ is common to all the different species, $\bmathM_{\bar{u}}$ is a local equilibrium state of the mixture. This means in particular that $\boldQ(\bmathM_{\bar{u}},\bmathM_{\bar{u}})=0$. We exploit this property by following the same used to study the linearized operator $\bepsL$. Namely, we split $\bepsM=\boldc\bmathM_{\bar{u}}+(\bepsM-\boldc\bmathM_{\bar{u}})$ into a local equilibrium part which cancels the nonlinear term, plus a penalty which is close to this local equilibrium up to an order $\eps^2\norm{\tilde{\boldu}}_{L^\infty_{t,x}}$. More precisely, since also~$\boldQ(\boldc\bmathM_{\bar{u}},\boldc\bmathM_{\bar{u}})=0$, we can rewrite
\begin{equation*}
\begin{split}
\boldQ(\bepsM,\bepsM)=&\ \boldQ(\boldc\bmathM_{\bar{u}},\bepsM-\boldc\bmathM_{\bar{u}})+\boldQ(\bepsM-\boldc\bmathM_{\bar{u}},\boldc\bmathM_{\bar{u}})
\\[2mm]    & +\boldQ(\bepsM-\boldc\bmathM_{\bar{u}},\bepsM-\boldc\bmathM_{\bar{u}}).
\end{split}
\end{equation*}
Similarly to the case of the linear part, one can infer (see \cite[Lemma 4.1]{BonBouBriGre} for more details) the existence of some polynomials $P_1$ and $P_2$ in one variable, such that, for any $1\leq i\leq N$ and for almost any $(t,x,v)\in\R_+\times\T^3\times\R^3$,
%
%
%
%
\begin{equation}\label{Nonlinear start}
\begin{split}
\multideriv\pabb{c_i(t,x)\pab{\mathM^\eps_i(t,x,v) - \mathM_{\bar{u},i}(t,x,v)}} &\leq \eps^2 \delta_\ms C_1(\delta) P_1(v)\mathM_i^\delta(v),   
\\[5mm]   \multideriv \pab{c_i(t,x)\mathM_{\bar{u},i}(t,x,v)} &\leq C_2(\delta) P_2(v)\mathM_i^\delta(v),  
\end{split}
\end{equation}
for some positive constants $C_1(\delta), C_2(\delta)$ depending on a parameter $\delta\in (0,1)$, which can be chosen independently of $\eps$. It is important to note that, for $|\alpha|=0$, we can still recover the correct factor $\delta_\ms$ in the first inequality, but this is not the case for the second one. Indeed, when $|\alpha|=0$, the computations carried out in Section \ref{Lemma1}, proof of Lemma \ref{lemma:Properties 1-2}, show that
\begin{equation*}
\partial^\beta_v \pab{c_i(\mathM^\eps_i-\mathM_{\bar{u},i})} \leq \eps^2 \norm{\tilde{\boldu}}_{L^\infty_{t,x}} \tilde{C}_1(\delta) \tilde{P}_1(v)\mathM_i^\delta(v),
\end{equation*}
for some other polynomial $\tilde{P}_1$ and a constant $\tilde{C}_1(\delta)>0$. In particular, we immediately see that this does not represent an issue, since we always consider products of type~${\multideriv\pab{c_i(\mathM^\eps_i-\mathM_{\bar{u},i})} \times \multideriv \pab{c_i\mathM_{\bar{u},i}}}$, which preserve the expected $\mathcal{O}(\delta_\ms^2)$.

Indeed, using estimate $\eqref{Property8}$ from Lemma \ref{LemmaQ} applied to the decomposition of $\boldQ(\bepsM,\bepsM)$, we deduce that
\begin{equation*}
\begin{split}
\frac{1}{\eps^3}& \abs{\scalprod{\multideriv \boldQ(\bepsM,\bepsM),\multideriv \boldf}_{\hilbertxv}} 
\\[2mm]    & \quad \leq \frac{2 C^\QQ_s }{\eps}\Bigg(\norm{\boldc\bmathM_{\bar{u}}}_{\sobolevxv{s}} \norm{\frac{\bepsM-\boldc\bmathM_{\bar{u}}}{\eps^2}}_{H^s_{x,v}\big( \langle v \rangle^{\frac{\gamma}{2}} \boldsymbol\mu^{-\frac{1}{2}}\big)}
\\[1mm]    & \quad \quad + \norm{\frac{\bepsM-\boldc\bmathM_{\bar{u}}}{\eps^2}}_{\sobolevxv{s}}\norm{\boldc\bmathM_{\bar{u}}}_{H^s_{x,v}\big( \langle v \rangle^{\frac{\gamma}{2}} \boldsymbol\mu^{-\frac{1}{2}}\big)}\Bigg)\norm{\multideriv \boldf}_{\spacexv}                    
\\[2mm]     & \quad \quad + \frac{C^\QQ_s}{\eps}\Bigg(\norm{\bepsM-\boldc\bmathM_{\bar{u}}}_{\sobolevxv{s}}\norm{\frac{\bepsM-\boldc\bmathM_{\bar{u}}}{\eps^2}}_{H^s_{x,v}\big( \langle v \rangle^{\frac{\gamma}{2}} \boldsymbol\mu^{-\frac{1}{2}}\big)}
\\[1mm]     & \quad \quad + \norm{\frac{\bepsM-\boldc\bmathM_{\bar{u}}}{\eps^2}}_{\sobolevxv{s}} \norm{\bepsM-\boldc\bmathM_{\bar{u}}}_{H^s_{x,v}\big( \langle v \rangle^{\frac{\gamma}{2}} \boldsymbol\mu^{-\frac{1}{2}}\big)}\Bigg) \norm{\multideriv\boldf}_{\spacexv}.                   
\end{split}
\end{equation*}
Now, each term inside the parentheses is dealt with by simply taking the square power in the inequalities $\eqref{Nonlinear start}$. In particular, we recognize the same structure of $\eqref{Linear term}$, which allows to infer the boundedness of all the terms involving $\boldc\bmathM_{\bar{u}}$ and $\bepsM-\boldc\bmathM_{\bar{u}}$, by some constant only depending on $\norm{\boldc}_{L^\infty_t\spacex{|\alpha|+4}}$ and $\norm{\boldu}_{L^\infty_t H^{|\alpha|+4}_x}$. In this way we recover the correct leading order $\mathcal{O}(\eps^{-1})$, and we can finally apply Young's inequality with the same constant $\eta_2/\eps$ used for the linear part, to get
\begin{equation}\label{NonlinearPart}
 \frac{1}{\eps^3}\scalprod{\multideriv \boldQ(\bepsM,\bepsM),\multideriv \boldf}_{\hilbertxv}\leq      
\frac{\delta_\ms^2 C_Q(\delta)}{\eta_2} +\frac{\eta_2}{\eps^2}\norm{\multideriv \boldf}_{\spacexv}^2. 
\end{equation}
To conclude, since $\langle v\rangle^\gamma\geq 1$, it is straightforward that
\begin{equation*}
\norm{\multideriv\boldf}_{\hilbertxv}\leq \norm{\multideriv\boldf}_{\spacexv},\quad \forall\boldf\in\sobolevTR{s}.
\end{equation*}
Thanks to this inequality, we can further bound the linear part $\eqref{LinearPart}$ to finally ensure the validity of the first estimate $\eqref{Property9}$, by gathering $\eqref{LinearPart}$ and $\eqref{NonlinearPart}$ with the choices~${C_{\alpha,\beta}=\max\br{\frac{C_L(\delta)}{\min_i c_{i,\infty}},C_Q(\delta)}}$.

\bigskip
The same arguments can be used to derive estimate $\eqref{Property10}$. The only difference is that, in the $\mathcal{H}^s_\eps$ norm, we only have a factor $\eps$ multiplying the commutator. Therefore, in this case, we are constrained to preserve the order $\eps^{-1}$ in front of each term. This is easily achieved by choosing a general positive constant $\eta_3$ (in place of $\eta_2/\eps$) each time Young's inequality is applied. In this way, estimate $\eqref{Property10}$ for the commutator directly follows.

\bigskip
\noindent \textbf{Step 2 - Estimates for the spatial derivatives.}
In order to prove estimate $\eqref{Property11}$ for the $x$-derivatives, we first expand the scalar product by means of the orthogonal projection~$\pi_\boldL$. We split
\begin{equation}\label{Last estimate}
\scalprod{\partial^\alpha_x \bepsS,\partial^\alpha_x\boldf}_{\hilbertxv}=\scalprod{\pi_{\boldL}(\partial^\alpha_x\bepsS),\pi_{\boldL}(\partial^\alpha_x \boldf)}_{\hilbertxv} +\big\langle\partial^\alpha_x \bepsS^\perp,\partial^\alpha_x \boldf^\perp\big\rangle_{\hilbertxv},
\end{equation}
and we study both contributions separately. Regarding the first term, let us begin by making some considerations about $\pi_{\boldL}(\bepsS)$. Using the linearity of the orthogonal projection, together with the fact that $\boldQ(\boldf,\boldf)\in(\Ker\boldL)^\perp$, from Lemma \ref{LemmaQ}, we obtain
\begin{eqnarray*}
\pi_{\boldL}(\bepsS)&=&\pi_{\boldL}\pa{-\frac{1}{\eps}\partial_t \bepsM-\frac{1}{\eps^2}v\cdot\nabla_x \bepsM}+\frac{1}{\eps^3}\pi_{\boldL}\pab{\boldQ(\bepsM,\bepsM)}
\\[2mm]&=& -\frac{1}{\eps^2}\pi_{\boldL}\pab{\eps\partial_t\bepsM+v\cdot\nabla_x\bepsM}.            
\end{eqnarray*}
The idea is again to prove that this term is actually $\mathcal{O}(\delta_\ms)$. In fact, $\pi_{\boldL}\pab{\eps\partial_t\bepsM+v\cdot\nabla_x\bepsM}$ is nothing but the projection onto the macroscopic equations governed by the Maxwell-Stefan system \eqref{MS mass}--\eqref{MS momentum}--\eqref{MS incompressibility}. Thus, following the computations in \cite{BouGreSal2}, it is easy to see that, for any $1\leq i\leq N$,
$$\int_{\R^3}\epsM_i(t,x,v)\pa{\begin{array}{c}1\\ v\\ |v|^2 \end{array}}\dd v=\pa{\begin{array}{c} c_i(t,x)\\ \\ \eps c_i(t,x)u_i(t,x)\\ \\\frac{3}{m_i}c_i(t,x)+\eps^2c_i(t,x)|u_i(t,x)|^2 \end{array}}.$$
We can then deduce, for any $1\leq i\leq N$, that
\begin{equation*}
\int_{\R^3}\pab{\eps\partial_t\epsM_i+v\cdot\nabla_x \epsM_i}\dd v=\eps\pabb{\partial_t c_i+\nabla_x\cdot (c_i u_i)},\hspace{5.4cm}
\end{equation*}

\begin{equation*}
\int_{\R^3}m_i v\pab{\eps\partial_t\epsM_i+v\cdot\nabla_x \epsM_i}\dd v=\eps^2m_i\pabb{\partial_t(c_i u_i)+\nabla_x\cdot(c_i u_i\otimes u_i)}+\nabla_x c_i,\hspace{1.5cm}
\end{equation*}

\begin{multline*}
\int_{\R^3}\frac{m_i|v|^2-3}{\sqrt{6}}\pab{\eps\partial_t\epsM_i+v\cdot\nabla_x \epsM_i}\dd v=
\\[1mm]   \frac{2\eps}{\sqrt{6}}\nabla_x\cdot(c_i u_i) + \frac{\eps^3}{\sqrt{6}}\pabb{m_i \partial_t (c_i |u_i|^2)+3\nabla_x\cdot(c_i|u_i|^2u_i)},
\end{multline*}
since

$$\int_{\R^3}|v|^2v\cdot\nabla_x\epsM_i\dd v=\eps\frac{5}{m_i}\nabla_x\cdot (c_i u_i)+\eps^3\frac{3}{m_i}\nabla_x\cdot(c_i|u_i|^2u_i).$$
The fact that $(\boldc,\boldu)$ is a solution of the system \eqref{MS mass}--\eqref{MS momentum}--\eqref{MS incompressibility} ensures that, for any~${1\leq i\leq N}$,
$$\partial_t c_i+\nabla_x\cdot (c_i u_i)=0,$$
with
$$\sum_{i=1}^N\nabla_x c_i=0,\qquad \partial_t\pa{\sum_{i=1}^N c_i}=0,\qquad \nabla_x\cdot \pa{\sum_{i=1}^N c_i u_i}=0.  $$
This allows to conclude from $\eqref{Projection explicit}$ that the orthogonal projection of $\bepsS$ onto $\Ker\boldL$ explicitly writes
\begin{equation*}
\begin{split}
\pi_{\boldL} & (\bepsS)=
\\[2mm]   & -\frac{v}{\sum_{i=1}^N m_i c_{i,\infty}}\cdot \cro{\sum_{i=1}^N m_i\pabb{\partial_t(c_i u_i)+\nabla_x\cdot(c_i u_i\otimes u_i)}}\pab{m_i\mu_i}_{1\leq i\leq N}
\\[2mm]    & -\frac{\eps}{\sqrt{6}\sum_{i=1}^N c_{i,\infty}}\cro{\sum_{i=1}^N\pabb{m_i\partial_t (c_i |u_i|^2)+3\nabla_x\cdot(c_i|u_i|^2u_i)}}\pa{\frac{m_i|v|^2-3}{\sqrt{6}}\mu_i}_{1\leq i\leq N}.
\end{split}
\end{equation*}
Supposing that we have all the required regularity on the solution $(\boldc,\boldu)$, as done in the previous step, we now note that either $\tilde{c}_i$ or $u_i$ always appear in each term of the sums in the left-hand side. Therefore, we can infer the existence of two constants $C_1,\:C_2 >0$ such that
\begin{gather*}
\norm{\sum_{i=1}^N\pabb{\partial_t(c_i u_i)+\nabla_x\cdot(c_i u_i\otimes u_i)}}_{L_x^2}^2 \leq C_1 \delta_\ms^2,
\\[3mm]   \norm{\sum_{i=1}^N\pabb{m_i \partial_t (c_i |u_i|^2)+3\nabla_x\cdot(c_i|u_i|^2u_i)}}_{L_x^2}^2  \leq C_2 \delta_\ms^2,
\end{gather*}
which in turn ensures that it is possible to control $\pi_{\boldL}(\bepsS)$ in the $\hilbertxv$ norm as
\begin{equation*}
\begin{split}
\norm{\pi_{\boldL}(\bepsS)}_{\hilbertxv}^2  & \leq C_{\pi_\boldL}\delta_\ms^2 \norm{(1+|v|+|v|^2)\boldmu}_{\hilbertxv}^2 
\\[3mm]   & \leq C_{\pi_\boldL} \delta_\ms^2, 
\end{split}
\end{equation*}
since the integral is clearly finite (we have also used that $\eps\in (0,1]$).

Moreover, because the orthogonal projection only acts on the velocity variable, it commutes with the $x$-derivatives. Note that this property was not at hand in the previous step, when dealing with $v$-derivatives. Therefore, in the case $|\beta|=0$ we are looking at here, the arguments used for $\pi_{\boldL}(\bepsS)$ apply in the exact same way to $\pi_{\boldL}(\partial^\alpha_x \bepsS)$, allowing to deduce that
\begin{equation*}
\norm{\pi_{\boldL}(\partial^\alpha_x\bepsS)}_{\hilbertxv}^2 = \norm{\partial^\alpha_x \pi_{\boldL}(\bepsS)}_{\hilbertxv}^2 \leq C_\alpha^{(1)} \delta_\ms^2. 
\end{equation*}
We conclude that the first contribution can easily be estimated applying Young's inequality with a constant $\eta_4 >0$, to get
\begin{equation}\label{ProjectionPart}
\begin{split}
\scalprod{\pi_{\boldL}(\partial^\alpha_x\bepsS),\pi_{\boldL}(\partial^\alpha_x\boldf)}_{\hilbertxv} \leq &\ \frac{1}{\eta_4}\norm{\pi_{\boldL}(\partial^\alpha_x\bepsS)}_{\hilbertxv}^2 + \eta_4 \norm{\pi_{\boldL}(\partial^\alpha_x\boldf)}_{\hilbertxv}^2
\\[5mm]   \leq &\  \frac{\delta_\ms^2 C_\alpha^{(1)}}{\eta_4} + \eta_4 \norm{\pi_{\boldL}(\partial^\alpha_x\boldf)}_{\hilbertxv}^2. 
\end{split}
\end{equation}

The second contribution in $\eqref{Last estimate}$ is handled in a very similar way, by explicitly writing the definition of the orthogonal part. It gives
\begin{eqnarray*}
\big\langle\partial^\alpha_x\bepsS^\perp,\partial^\alpha_x\boldf^\perp\big\rangle_{\hilbertxv}&=&\big\langle\partial^\alpha_x\bepsS-\pi_{\boldL}(\partial^\alpha_x\bepsS),\partial^\alpha_x\boldf^\perp\big\rangle_{\hilbertxv}
\\[3mm]    &=&\big\langle\partial^\alpha_x\bepsS,\partial^\alpha_x\boldf^\perp\big\rangle_{\hilbertxv}-\big\langle\pi_{\boldL}(\partial^\alpha_x\bepsS),\partial^\alpha_x\boldf^\perp\big\rangle_{\hilbertxv}.
\end{eqnarray*}
Looking at these two terms, we regognize the same structure of the scalar products that led to estimates $\eqref{LinearPart}$, $\eqref{NonlinearPart}$ and $\eqref{ProjectionPart}$. In fact, all the computations carried out until now only depend on the particular form of the source term, and not on the scalar product in itself (the final upper bounds have been always obtained thanks to an application of Young's inequality). Using Young's inequality with two positive constants $\eta_5/\eps$ and $\eta_5$, we can thus repeat the previous considerations to recover (increasing the values of the constants if necessary)
\begin{equation*}
\begin{split}
\big\langle\partial^\alpha_x\bepsS,\partial^\alpha_x\boldf^\perp\big\rangle_{\hilbertxv}\leq \frac{\delta_\ms^2 \pab{C_L(\delta) + C_Q(\delta)}}{\eta_5} + \frac{\eta_5}{\eps^2}\norm{\partial^\alpha_x\boldf^\perp}_{\spacexv}^2,
\end{split}
\end{equation*}
from \eqref{LinearPart}--\eqref{NonlinearPart}, and
\begin{equation*}
\begin{split}
\big\langle\pi_{\boldL}(\partial^\alpha_x\bepsS),\partial^\alpha_x\boldf^\perp\big\rangle_{\hilbertxv} \leq \frac{\delta_\ms^2 C_\alpha^{(1)}}{\eta_5} + \eta_5 \norm{\partial^\alpha_x\boldf^\perp}_{\hilbertxv}^2,
\end{split}
\end{equation*}
from $\eqref{ProjectionPart}$. Therefore, since the $\spacexv$ norm controls the $\hilbertxv$ norm, we obtain
\begin{equation}\label{OrthogonalPart}
\begin{split}
\big\langle\partial^\alpha_x\bepsS^\perp,\partial^\alpha_x\boldf^\perp\big\rangle_{\hilbertxv} \leq &\  \frac{\delta_\ms^2 C_\alpha^{(2)}}{\eta_5} + \frac{\eta_5}{\eps^2} \norm{\partial^\alpha_x\boldf^\perp}_{\spacexv}^2. 
\end{split}
\end{equation}

\vspace{1mm}
We finally infer the validity of the last estimate $\eqref{Property11}$, by gathering \eqref{ProjectionPart}--\eqref{OrthogonalPart} with the choice $C_{\alpha} = \max\br{\eta_5 C_\alpha^{(1)},\eta_4 C_\alpha^{(2)}}$. This concludes the proof of Lemma \ref{LemmaS}. 
\bigskip


\appendix


\section{Explicit Carleman representation of the operator $\bepsK$}\label{App:Carleman}

\noindent We here provide the basic tools that are used in Lemma \ref{LemmaK} to prove the regularizing effect of $\multideriv \bepsK$. Looking at the work of Mouhot and Neumann \cite{MouNeu} in the mono-species context, the authors recover this property by transferring to the kernel of the compact operator $K$ all the derivatives, which are then computed and estimated explicitly. This analysis is possible mainly because the kernel of $K$ has itself an explicit expression. Ideally, one may want to apply a similar strategy in our multi-species framework, but this would require knowing the structure of the kernel of $\bepsK$.

In this first appendix we derive an explicit expression of the kernel of $\bepsK$, following the methods of \cite{BriDau} where a Carleman representation of the Boltzmann multi-species operator was obtained. In particular, we shall rework the pointwise estimates established by the authors, replacing them by a series of pointwise equalities where all the technical computations are made fully explicit.

\vspace{1mm}
Let us begin by recalling that $\bepsK=(\epsK_1,\ldots,\epsK_N)$  can be written componentwise, for any $\boldf\in\hilbertR$, under the kernel form \cite[Lemma 5.1]{BriDau}
\begin{equation*}
\begin{split}
 \epsK_i(\boldf)&(v) = 
 \\[2mm]   & \sum_{j=1}^N C_{ij} \int_{\R^3}\left(\frac{1}{|v-v_*|}\int_{\tilde{E}^{ij}_{v v_*}}\frac{B_{ij}\left(v-V(w,v_*),\frac{v_*-w}{|w-v_*|}\right)}{|w-v_*|}\epsM_i(w)\ \dd \tilde{E}(w)\right)\fs_j\dd v_*
 \\[2mm]    + &  \sum_{j=1}^N C_{ji} \int_{\R^3}\left(\frac{1}{|v-v_*|}\int_{E^{ij}_{v v_*}}\frac{B_{ij}\left(v-V(v_*,w),\frac{w-v_*}{|w-v_*|}\right)}{|w-v_*|}\epsM_j(w)\ \dd E(w)\right)f_i^*\dd v_* 
 \\[2mm]      -  &  \sum_{j=1}^N\int_{\R^3}\pa{\int_{\Sf}B_{ij}(\abs{v-v_*},\cos\theta)\epsM_i\dd \sigma}\fs_j\dd v_*, 
\end{split}
\end{equation*}
where we have defined $V(w,v_*)=v_*+m_i m_j^{-1}w-m_i m_j^{-1}v$ and called $C_{ij},C_{ji}>0$ some explicit constants which only depend on the masses $m_i, m_j$. Moreover, we have denoted by $\dd E$ the Lebesgue measure on the hyperplane $E_{v v_*}^{ij}$, orthogonal to $v-v_*$ and passing through 
\[       V_E(v,v_*)=\frac{m_i+m_j}{2 m_j}v-\frac{m_i-m_j}{2m_j}v_*,      \]
and by $\dd \tilde{E}$ the Lebesgue measure on the space $\tilde{E}_{v v_*}^{ij}$ which corresponds to the hyperplane $E_{v v_*}^{ij}$ whenever $m_i=m_j$, and to the sphere of radius $R=R(v,v_*)$
\[        R=\frac{m_j}{|m_i-m_j|}|v-v_*|     \]
and centred at $O=O(v,v_*)$
\[       O=\frac{m_i}{m_i-m_j}v-\frac{m_j}{m_i-m_j}v_*,     \]
whenever $m_i\neq m_j$.

We can thus define, for any $1\leq i,j\leq N$, the kernels
\begin{equation*}
\begin{split}
\kappa_{ij}^{(1)}(v,v_*) & = \frac{C_{ij}}{|v-v_*|}\int_{\tilde{E}^{ij}_{v v_*}}\frac{B_{ij}\left(v-V(w,v_*),\frac{v_*-w}{|w-v_*|}\right)}{|w-v_*|}\epsM_i(w)\ \dd \tilde{E}(w),
\\[2mm]     \kappa_{ij}^{(2)}(v,v_*) & = \frac{C_{ji}}{|v-v_*|}\int_{E^{ij}_{v v_*}}\frac{B_{ij}\left(v-V(v_*,w),\frac{w-v_*}{|w-v_*|}\right)}{|w-v_*|}\epsM_j(w)\ \dd E(w),
\\[5mm]    \kappa_{ij}^{(3)}(v,v_*) & = \int_{\Sf}B_{ij}(\abs{v-v_*},\cos\theta)\epsM_i\dd \sigma,
\end{split}
\end{equation*}
where we have dropped the parameter $\eps$ in order to enlighten our notations. In this way, each $\epsK_i$ can be rewritten as
\[         \epsK_i(\boldf)(v)=\sum_{j=1}^N\br{\int_{\R^3}\kappa_{ij}^{(1)}(v,v_*)\fs_j\dd v_* + \int_{\R^3}\kappa_{ij}^{(2)}(v,v_*)\fs_i\dd v_* - \int_{\R^3}\kappa_{ij}^{(3)}(v,v_*)\fs_j\dd v_*}.            \]
Let us now fix two indices $i,j\in\br{1,\ldots,N}$ and study each of the three kernels separately.

\subsection{Explicit form of $\kappa_{ij}^{(1)}$}  

The first kernel is easy to make explicit, since the domain of integration $\tilde{E}_{v v_*}^{ij}$ is a sphere. We thus perform an initial change of variables consisting on a translation of its centre and a dilation of its radius, in order to end up on $\Sf$. In this new coordinate system, $\kappa_{ij}^{(1)}$ reads
\[              \kappa_{ij}^{(1)}(v,v_*)=\frac{C_{ij}m_j^2}{(m_i -m_j)^2}\abs{v-v_*}\int_{\Sf}\frac{b_{ij}(v,v_*,\omega)}{|R\omega+O-v_*|^{1-\gamma}}\epsM_i(R\omega+O)\ \dd \omega              \]
where the angular part $b_{ij}$ explicitly writes
\[               b_{ij}(v,v_*,\omega) = b_{ij}\pa{\frac{\pab{v-V(R\omega+O,v_*)}\cdot\pab{v_*-\pa{R\omega+O}}}{\abs{v-V(R\omega+O,v_*)}\abs{R\omega+O-v_*}}},                 \]
recalling that we have defined $V(w,v_*)=v_*+m_i m_j^{-1}w-m_i m_j^{-1}v$. Simple algebraic manipulations show that
\begin{equation*}
\begin{split}
v-V(R\omega+O,v_*) = &\  v-v_*-\frac{m_i}{m_j}(R\omega+O)+\frac{m_i}{m_j}v
\\[2mm]   &   -\frac{m_i}{\abs{m_i-m_j}}\abs{v-v_*}\omega-\frac{m_j}{m_i-m_j}(v-v_*), 
\end{split}
\end{equation*}
and similarly
\[    v_*-\pa{R\omega+O}=-\frac{m_j}{\abs{m_i-m_j}}\abs{v-v_*}\omega-\frac{m_i}{m_i-m_j}(v-v_*).    \]
It is then easy to check that
\begin{equation*}
\begin{split}
\pab{v-V(R\omega+O,v_*)}& \cdot \pab{v_*-\pa{R\omega+O}} 
\\[2mm]      = &\ \frac{2m_im_j}{(m_i-m_j)^2}|v-v_*|^2+\frac{m_i^2+m_j^2}{(m_i-m_j)|m_i-m_j|}|v-v_*|(v-v_*)\cdot \omega
\\[2mm]      = &\  |v-v_*|^2\pa{\frac{2m_im_j}{(m_i-m_j)^2}+\frac{m_i^2+m_j^2}{(m_i-m_j)|m_i-m_j|}\frac{(v-v_*)\cdot\omega}{|v-v_*|}}
\end{split}
\end{equation*}
and
\begin{equation*}
\begin{split}
\abs{v-V(R\omega+O,v_*)}^2 & = \abs{R\omega+O-v_*}^2
\\[2mm]    & = \frac{m_i^2+m_j^2}{(m_i-m_j)^2}|v-v_*|^2+ \frac{2m_im_j}{(m_i-m_j)|m_i-m_j|}|v-v_*|(v-v_*)\cdot\omega
\\[2mm]     & =  |v-v_*|^2\pa{\frac{m_i^2+m_j^2}{(m_i-m_j)^2}+\frac{2m_im_j}{(m_i-m_j)|m_i-m_j|}\frac{(v-v_*)\cdot\omega}{|v-v_*|}}.   
\end{split}
\end{equation*}
Thus, the angular part depends on $v$ only through the cosine of the deviation angle between $\frac{v-v_*}{|v-v_*|}$ and $\omega\in\Sf$, and finally reads
\[     b_{ij}\pa{\omega\cdot\frac{v-v_*}{|v-v_*|}}=b_{ij}\pa{\frac{\frac{2m_im_j}{(m_i-m_j)^2}+\frac{m_i^2+m_j^2}{(m_i-m_j)|m_i-m_j|}\frac{(v-v_*)\cdot\omega}{|v-v_*|}}{\frac{m_i^2+m_j^2}{(m_i-m_j)^2}+\frac{2m_im_j}{(m_i-m_j)|m_i-m_j|}\frac{(v-v_*)\cdot\omega}{|v-v_*|}}}.                 \] 
We can then rewrite our integral term $\kappa_{ij}^{(1)}$ as
\begin{multline}\label{kappa1}   
\kappa_{ij}^{(1)}(v,v_*)=C_{ij}c_i |v-v_*|^{\gamma}\int_{\Sf}\frac{b_{ij}\pa{\omega\cdot\frac{v-v_*}{|v-v_*|}}c_i }{\abs{\frac{m_i^2+m_j^2}{(m_i-m_j)^2}+\frac{2m_i m_j}{(m_i-m_j)|m_i-m_j|}\frac{(v-v_*)\cdot\omega}{|v-v_*|}}^{1-\gamma}} 
\\[3mm]      \times e^{-\frac{m_i}{2}\abs{R\omega +O}^2+\eps m_i (R\omega +O)\cdot u_i-\eps^2\frac{m_i}{2}|u_i |^2}\ \dd \omega,
\end{multline}
where we have renamed for simplicity
\[         C_{ij}=\frac{C_{ij} m_j^2}{(m_i -m_j)^2}\pa{\frac{m_i}{2\pi}}^{3/2}        \]
and the exponent explicitly reads
\begin{equation*}
\begin{split}
\abs{R\omega +O}^2  = &   \frac{m_j^2}{(m_i-m_j)^2}|v-v_*|^2 +\frac{2 m_j|v-v_*|}{(m_i-m_j)|m_i-m_j|}(m_i v-m_j v_*)\cdot\omega
\\[2mm]     &   + \frac{1}{(m_i-m_j)^2}|m_i v-m_j v_*|^2,
\end{split}
\end{equation*}

\vspace{3mm}
\[       (R\omega +O)\cdot u_i = \frac{m_j}{|m_i-m_j|}|v-v_*| \omega\cdot u_i + \frac{(m_i v-m_j v_*)\cdot u_i}{m_i-m_j}.          \]
This concludes the study for $\kappa_{ij}^{(1)}$.

\subsection{Explicit form of $\kappa_{ij}^{(2)}$}  

Recovering the explicit expression of $\kappa_{ij}^{(2)}$ is more subtle. We recall that the domain of integration $E_{v v_*}^{ij}$ is the hyperplane orthogonal to $v-v_*$ and passing through 
\[       V_E(v,v_*)=\frac{m_i+m_j}{2 m_j}v-\frac{m_i-m_j}{2m_j}v_*.        \]
Let us consider  $\omega\in\pab{\Span(v-v_*)}^\perp$ and let us make the initial change of variables $w=V_E(v,v_*)+\omega$ which translates $E_{v v_*}^{ij}$ to the parallel hyperplane passing through the origin of $\R^3$. Thus $\kappa_{ij}^{(2)}$ transforms into
\begin{equation*}
\begin{split}
\kappa_{ij}^{(2)}(v,v_*)   &  = \frac{C_{ji}}{|v-v_*|}\int_{E^{ij}_{v v_*}}\frac{B_{ij}\left(v-V(v_*,w),\frac{w-v_*}{|w-v_*|}\right)}{|w-v_*|}\epsM_j(w)\ \dd E(w) 
\\[2mm]     &  = \frac{C_{ji}}{|v-v_*|}\int_{(v-v_*)^\perp}\frac{b_{ij}(v,v_*,\omega)}{|v-V(v_*,V_E(v,v_*)+\omega)|^{1-\gamma}}\epsM_j(V_E(v,v_*)+\omega)\ \dd \omega,
\end{split}
\end{equation*}
where the angulat part writes
\[      b_{ij}(v,v_*,\omega)= b_{ij}\pa{\frac{\pab{ v-V(v_*,V_E(v,v_*)+\omega)\cdot \pab{V_E(v,v_*)+\omega-v_*} }}{\abs{v-V(v_*,V_E(v,v_*)+\omega)}|V_E(v,v_*)+\omega-v_*|}}.     \]
Easy calculations show that
\begin{gather*}
V_E(v,v_*)+\omega-v_*=\frac{m_i+m_j}{2m_j}(v-v_*)+\omega, 
\\[2mm]   v-V(v_*,V_E(v,v_*)+\omega)=\frac{m_i+m_j}{2m_j}(v-v_*)-\omega,  
\end{gather*}
and
\begin{gather*}
\pab{v-V(v_*,V_E(v,v_*)+\omega)}\cdot \pab{V_E(v,v_*)+\omega-v_*}=\pa{\frac{m_i+m_j}{2m_j}}^2 |v-v_*|^2 - |\omega|^2,
\\[2mm]     \abs{v-V(v_*,V_E(v,v_*)+\omega)}^2 = \abs{V_E(v,v_*)+\omega-v_*}^2= \pa{\frac{m_i+m_j}{2m_j}}^2 |v-v_*|^2 + |\omega|^2.    
\end{gather*}
Moreover, the exponent of the Maxwellian can be computed as follows. We initially develop the square to get
\[        \abs{V_E(v,v_*)+\omega-\eps u_j}^2=\abs{V_E(v,v_*)+\omega}^2 -2\eps\pab{V_E(v,v_*)+\omega}\cdot u_j +\eps^2| u_j |^2.      \]
The first term can be rewritten as
\begin{equation*}
\begin{split}
\abs{V_E(v,v_*)+\omega}^2  &  = \abs{\omega+\frac{1}{2}(v+v_*)+\frac{m_i}{2m_j}(v-v_*)}^2 
\\[2mm]      &  =  \abs{\omega+\frac{1}{2}(v+v_*)}^2 +\frac{m_i^2}{4m_j^2}|v-v_*|^2+\frac{m_i}{2m_j}\pab{|v|^2-|v_*|^2} 
\\[2mm]      &  =   \abs{\omega+\frac{1}{2}V^\perp}^2 +\frac{1}{4}\abs{V^\parallel}^2+\frac{m_i^2}{4m_j^2}|v-v_*|^2+\frac{m_i}{2m_j}\pab{|v|^2-|v_*|^2}, 
\end{split}
\end{equation*}
where we have decomposed $v+v_*=V^\parallel + V^\perp$, with $V^\parallel$ being the projection onto $\mathrm{Span}(v-v_*)$ and $V^\perp$ being the orthogonal part. In the same way, the second term reads
\[          \pab{V_E(v,v_*)+\omega}\cdot u_j=\pa{\frac{1}{2}V^\parallel+\frac{m_i}{2 m_j}(v-v_*)}\cdot u_j + \pa{\frac{1}{2}V^\perp+\omega}\cdot u_j.         \]
Since by the definition of $V^\parallel$
\[      \abs{V^\parallel}^2 =\frac{\pab{(v+v_*)\cdot(v-v_*)}^2}{|v-v_*|^2}=\frac{\abs{|v|^2-|v_*|^2}^2}{|v-v_*|^2},      \]  
the kernel $\kappa_{ij}^{(2)}$ becomes
\begin{gather*}
\kappa_{ij}^{(2)}(v,v_*)=\mathcal{P}_{ij}(v,v_*)\int_{(v-v_*)^\perp}b_{ij}(v,v_*,\omega)W_{ij}(v,v_*,\omega)\epsM_j\pa{\omega+\frac{1}{2}V^\perp}\dd E(\omega),
\end{gather*}
where
\begin{gather*}
\mathcal{P}_{ij}(v,v_*)=\frac{C_{ji}}{|v-v_*|} e^{-\frac{m_i^2}{8m_j}|v-v_*|^2-\frac{m_j}{8}\frac{\abs{|v|^2-|v_*|^2}^2}{|v-v_*|^2}+\eps \frac{m_i}{2}(v-v_*)\cdot u_j + \eps\frac{m_j}{2}\frac{|v|^2-|v_*|^2}{|v-v_*|}\frac{(v-v_*)\cdot u_j}{|v-v_*|}}\sqrt{\frac{\mu_i}{\mu_i^*}},
 \\[3mm]    b_{ij}(v,v_*,\omega)=b_{ij}\pa{\frac{\pa{\frac{m_i+m_j}{2m_j}}^2 |v-v_*|^2 - |\omega|^2}{\pa{\frac{m_i+m_j}{2m_j}}^2 |v-v_*|^2 + |\omega|^2}}, \qquad
 \\[3mm]     W_{ij}(v,v_*,\omega)=\pa{\pa{\frac{m_i+m_j}{2m_j}}^2 |v-v_*|^2 + |\omega|^2}^{\frac{\gamma-1}{2}}.
\end{gather*}

Finally, it remains to take care of the domain of integration which still depends on $(v,v_*)$. The idea is to transform the hyperplane defined by $(v-v_*)^\perp$ to end up on $\R^2$. Proceeding as in \cite[Proposition 2.4]{Mou1}, we first observe that the integral is even with respect to $v-v_*$, since it only depends on its modulus. Thus, we focus on the set of relative velocities $v-v_*$ such that the first coordinate is nonnegative. Call $e_1$ the first unit vector of the corresponding orthonormal basis and, for any fixed $\frac{v-v_*}{|v-v_*|}$, introduce the linear transformation
\[          \mathcal{L}\pa{\frac{v-v_*}{|v-v_*|},\mathcal{X}}=2\frac{\pa{e_1+\frac{v-v_*}{|v-v_*|}}\cdot\mathcal{X}}{\abs{e_1+\frac{v-v_*}{|v-v_*|}}^2}\pa{e_1+\frac{v-v_*}{|v-v_*|}}-\mathcal{X},\quad \forall\mathcal{X}\in\R^3,           \] 
which corresponds to the specular reflection through the axis defined by $e_1+\frac{v-v_*}{|v-v_*|}$. Now, $\mathcal{L}$ is a diffeomorphism from $\br{\mathcal{X}=\pa{0,X}, X\in\R^2}$ onto $(v-v_*)^\perp$, with unitary Jacobian matrix. Thus, we can use this linear transformation to pass from $(v-v_*)^\perp$ to $\R^2$ into the integral of $\kappa_{ij}^{(2)}$, which can be finally explicitly written as
\begin{multline}\label{kappa2}    
\hspace{-2mm}\kappa_{ij}^{(2)}(v,v_*)  = 
\\[2mm]   \mathcal{P}_{ij}(v,v_*)  \int_{\R^2}b_{ij}\pa{\frac{\pa{\frac{m_i+m_j}{2m_j}}^2 |v-v_*|^2 - |X|^2}{\pa{\frac{m_i+m_j}{2m_j}}^2 |v-v_*|^2 + |X|^2}} \pa{\pa{\frac{m_i+m_j}{2m_j}}^2 |v-v_*|^2 + |X|^2}^{\frac{\gamma-1}{2}}
\\[3mm]      \times \epsM_j\pa{\mathcal{L}\pa{\frac{v-v_*}{|v-v_*|},\pa{0,X+\frac{1}{2}\bar{X}}}} \dd X, 
\end{multline}
where we have called $\bar{X}\in\R^2$ the preimage of $V^\perp\in (v-v_*)^\perp$ through the transformation $\mathcal{L}$, and we have used the straightforward identity
\[          \abs{\mathcal{L}\pa{\frac{v-v_*}{|v-v_*|},(0,X)}}=|X|, \qquad \forall X\in\R^2,\ \forall v,v_*\in\R^3.        \]
This concludes the study for $\kappa_{ij}^{(2)}$.

\subsection{Explicit form of $\kappa_{ij}^{(3)}$}  

The analysis of $\kappa_{ij}^{(3)}$ is the easiest one, since it is already fully explicit. It simply reads
\begin{equation*}
\begin{split}
\kappa_{ij}^{(3)}(v,v_*)  & =  \int_{\Sf}B_{ij}(\abs{v-v_*},\cos\theta)\epsM_i\dd \sigma 
\\[3mm]     & = \int_{\Sf} b_{ij}\pa{\sigma\cdot \frac{v-v_*}{|v-v_*|}}|v-v_*|^\gamma \epsM_i(v)\dd\sigma. 
\end{split}
\end{equation*}
This concludes its study.


\section{Proofs of the \textit{a priori} energy estimates for the Boltzmann equation}\label{App:apriori BE}

\noindent We shall follow the computations in \cite[Appendix B]{Bri1}, in order to show that we recover very similar \textit{a priori} estimates for the quantities in $\sobolevxv{s}$.

Let us thus consider an integer $s\in\N^*$ and a function $\boldf\in\sobolevTR{s}$ which solves the perturbed Boltzmann equation
\begin{equation}\label{Initial Equation}
\partial_t\boldf+\frac{1}{\eps}v\cdot\nabla_x \boldf=\frac{1}{\eps^2}\bepsL(\boldf)+\frac{1}{\eps}\boldQ(\boldf,\boldf)+\bepsS,
\end{equation}
and satisfies initially $\norm{\pi_{\bepsT}(\boldf^\init)}_{\hilbertxv}=\mathcal{O}(\delta_\ms)$. In particular, we suppose from now on that $\delta_\ms\leq 1$. Indeed, even if not optimal, this choice helps in enlightening the presentation.

To simplify the computations, we recall that we denote $\boldf^\perp=\boldf-\pi_\boldL(\boldf)$ the projection onto $(\Ker\boldL)^\perp$.

\subsection{Time evolution of $\norm{\boldf}_{\hilbertxv}^2$}  

We initially take the scalar product of $\eqref{Initial Equation}$ against $\boldf \boldmu^{-1}$ and we integrate in $x$ and $v$ to get
\begin{equation*}
\begin{split}
\frac{\dd}{\dd t}\norm{\boldf}_{\hilbertxv}^2   =  &\ \frac{2}{\eps^2}\scalprod{\bepsL(\boldf),\boldf}_{\hilbertxv} - \frac{2}{\eps}\scalprod{v\cdot\nabla_x\boldf, \boldf}_{\hilbertxv}
\\[3mm]    &  + \frac{2}{\eps}\scalprod{\boldQ(\boldf,\boldf),\boldf}_{\hilbertxv}+2\scalprod{\bepsS,\boldf}_{\hilbertxv}.
\end{split}
\end{equation*}
Thanks to the anti-symmetry of $v\cdot\nabla_x$, the term containing the transport operator vanishes. In order to bound the linear term, we exploit the spectral gap estimate $\eqref{Property5}$ satisfied by $\bepsL$. We successively obtain
\begin{multline}\label{Linear term}
\frac{2}{\eps^2}\langle  \bepsL(\boldf), \boldf\rangle_{\hilbertxv} \leq  -\frac{2}{\eps^2}\pabb{\lambda_{\boldL}-(\eps +\eta_1) C^\LL_2}\normb{\boldf^\perp}_{\spacexv}^2  
\\[3mm]    \hspace{7cm} + \frac{2 \delta_\ms  C^\LL_2}{\eta_1}\norm{\pi_{\boldL}(\boldf)}_{\spacexv}^2
\\[4mm]  \leq  -\frac{2}{\eps^2}\pabb{\lambda_{\boldL}-(\eps +\eta_1) C^\LL_2}\normb{\boldf^\perp}_{\spacexv}^2 + \delta_\ms\frac{4 C^\LL_2 C_\pi C_{\T^3}}{\eta_1}\norm{\nabla_x\boldf}_{\hilbertxv}^2
\\[4mm]    + \delta_\ms^3 \frac{2 C^\LL_2 C_\pi C^\TT}{\eta_1},
\end{multline}
where we have also used the equivalence between the $\spacexv$ and $\hilbertxv$ norms on $\Ker \boldL$ from Lemma \ref{lem:Equivalence}, and the Poincar\'e inequality $\eqref{Poincare inequality}$. 

The bilinear term is handled thanks to properties \eqref{Property6}--\eqref{Property7}. Applying Young's inequality with a positive constant $\eta/\eps$, we recover
\begin{equation}\label{Non-linear term}
\frac{2}{\eps}\scalprod{\boldQ(\boldf,\boldf),\boldf}_{\hilbertxv}\leq \frac{2}{\eta}\mathcal{G}^0_x(\boldf,\boldf)^2+\frac{2\eta}{\eps^2}\normb{\boldf^\perp}_{\spacexv}^2.
\end{equation}
Finally, the source term is dealt with using estimate $\eqref{Property11}$ for the $x$-derivatives, which also holds when $|\alpha|=0$, with some positive constant $C$. Applying again the Poincar\'e inequality $\eqref{Poincare inequality}$, we compute
\begin{equation}\label{Source term}
\begin{split}
2\langle \bepsS,\boldf\rangle_{\hilbertxv} \leq & \ \frac{2 \delta_\ms^2 C}{\eta_4 \eta_5} +2 \eta_4 \norm{\pi_{\boldL}(\boldf)}_{\hilbertxv}^2+ \frac{2 \eta_5}{\eps^2} \normb{\boldf^\perp}_{\spacexv}^2
\\[6mm]   \leq &\   \frac{2 \delta_\ms^2 C}{\eta_4 \eta_5} + \frac{2 \eta_5}{\eps^2} \normb{\boldf^\perp}_{\spacexv}^2 + 4 \eta_4 C_{\T^3} \norm{\nabla_x\boldf}_{\hilbertxv}^2
\\[4mm]   & \hspace{7.5cm} + 2\eta_4 \delta_\ms^2 C^\TT.
\end{split}
\end{equation}

Gathering inequalities \eqref{Linear term}--\eqref{Source term}, with the choices $\eta = \eta_1 = \eta_5 = \frac{\lambda_\boldL}{4(C^\LL_2 + 1)}$ and also $\eta_4=\delta_\ms$ and $\eps\leq \eta$, we obtain 
\begin{equation}
\begin{split}
\frac{\dd}{\dd t}\norm{\boldf}_{\hilbertxv}^2 \leq &\  -\frac{\lambda_\boldL}{\eps^2} \normb{\boldf^\perp}_{\spacexv}^2 + \frac{8(C^\LL_2+1)}{\lambda_\boldL}\mathcal{G}^0_x(\boldf,\boldf)^2
\\[3mm]  &\  + \delta_\ms \pa{8 C_{\T^3} + \frac{32 C^\LL_2 C_\pi C_{\T^3} (C^\LL_2 + 1)}{\lambda_\boldL}}\norm{\nabla_x \boldf}_{\hilbertxv}^2 
\\[3mm]  &\   + \delta_\ms \pa{2 C^\TT + \frac{8 C^\LL_2 C_\pi C^\TT(C^\LL_2+1)}{\lambda_\boldL} + \frac{8 C (C^\LL_2+1)}{\lambda_\boldL}},
\end{split}
\end{equation}
holding for any $\delta_\ms \in [0,1]$. By choosing
\begin{gather*}
C^{(1)} = 8 C_{\T^3} + \frac{32 C^\LL_2 C_\pi C_{\T^3} (C^\LL_2 + 1)}{\lambda_\boldL}, \\[2mm]
\tilde{C} = 2 C^\TT + \frac{8 C^\LL_2 C_\pi C^\TT(C^\LL_2+1)}{\lambda_\boldL} + \frac{8 C (C^\LL_2+1)}{\lambda_\boldL},
\end{gather*}
we thus recover the first estimate $\eqref{a priori f}$.

\subsection{Time evolution of $\norm{\nabla_x\boldf}_{\hilbertxv}^2$}  

The time evolution of the $\hilbertxv$ norm of $\nabla_x\boldf$ is given by
\begin{equation*}
\begin{split}
\frac{\dd}{\dd t}\norm{\nabla_x \boldf}_{\hilbertxv}^2   =  &\ \frac{2}{\eps^2}\scalprod{\nabla_x\bepsL(\boldf),\nabla_x\boldf}_{\hilbertxv} - \frac{2}{\eps}\scalprod{\nabla_x \pa{v\cdot\nabla_x\boldf},\nabla_x \boldf}_{\hilbertxv}
\\[3mm]    &  + \frac{2}{\eps}\scalprod{\nabla_x\boldQ(\boldf,\boldf),\nabla_x\boldf}_{\hilbertxv}+2\scalprod{\nabla_x\bepsS,\nabla_x \boldf}_{\hilbertxv}.
\end{split}
\end{equation*}
First of all, since the transport operator commutes with $x$-derivatives, the anti-symmetry property allows again to get rid of it. We then study the linearized operator $\bepsL$. Applying the Leibniz derivation rule and using the orthogonality of $\boldQ$ to $\Ker \boldL$ given by $\eqref{Property6}$, we initially observe that
\begin{equation}\label{Linear initial}
\begin{split}
\scalprod{\nabla_x\bepsL(\boldf),\nabla_x\boldf}_{\hilbertxv} = &\  \scalprodb{\boldQ(\nabla_x\bepsM,\boldf)+\boldQ(\boldf,\nabla_x\bepsM),\nabla_x\boldf^\perp}_{\hilbertxv}
\\[3mm]     & \ + \scalprod{\bepsL(\nabla_x\boldf),\nabla_x\boldf}_{\hilbertxv}.
\end{split}
\end{equation}
The addends involving the derivative of $\bepsM$ are controlled by a factor of order $\mathcal{O}(\eps\delta_\ms )$. Using estimate $\eqref{Property2}$ we thus get
\begin{multline*}
\scalprodb{\boldQ(\nabla_x\bepsM,\boldf)+\boldQ(\boldf,\nabla_x\bepsM),\nabla_x\boldf^\perp}_{\hilbertxv}   
\\[2mm]    \leq \eps\delta_\ms C^\LL_1 K_x \norm{\boldf}_{\spacexv} \normb{\nabla_x\boldf^\perp}_{\spacexv}.
\end{multline*}
Using the fact that $\boldf=\pi_\boldL(\boldf)+\boldf^\perp$ and the equivalence of the $\hilbertxv$ and $\spacexv$ norms for $\pi_\boldL(\boldf)$, we can then apply Young's inequality with a positive constant $\eta/\eps$ (and increment $K_x$ if necessary) to recover the upper bound
\begin{multline*}
\scalprodb{\boldQ(\nabla_x\bepsM,\boldf)+\boldQ(\boldf,\nabla_x\bepsM),\nabla_x\boldf^\perp}_{\hilbertxv}
\\[2mm]   \leq  \eta \delta_\ms C^\LL_1 K_x \normb{\nabla_x\boldf^\perp}_{\spacexv}^2 + \frac{2\eps\delta_\ms C^\LL_1 K_x}{\eta} \normb{\boldf^\perp}_{\spacexv}^2 
\\[3mm]    +  \frac{2\eps^2\delta_\ms C_\pi  C^\LL_1 K_x}{\eta} \norm{\pi_\boldL(\boldf)}_{\hilbertxv}^2,
\end{multline*}
where we have also used that $\eps\leq 1$. Moreover, the second term of $\eqref{Linear initial}$ is handled thanks to estimate $\eqref{Property5}$ on the spectral gap of $\bepsL$, and can be bounded as
\begin{multline*}
\scalprod{\bepsL(\nabla_x\boldf), \nabla_x\boldf}_{\hilbertxv}\leq      -\pabb{\lambda_{\boldL}-(\eps +\eta_1) C^\LL_2}\norm{\nabla_x \boldf^\perp}_{\spacexv}^2 
\\[3mm]                +\eps^2 \delta_\ms  \frac{C^\LL_2}{\eta_1}\norm{\pi_{\boldL}(\boldf)}_{\spacexv}^2.
\end{multline*}
Choosing $\eta = \eta_1$, we can finally apply the Poincar\'e inequality $\eqref{Poincare inequality}$ to obtain
\begin{multline}\label{Linear term x}
\frac{2}{\eps^2}\scalprod{\nabla_x\bepsL(\boldf),\nabla_x\boldf}_{\hilbertxv} 
\\[3mm]\leq  -\frac{2}{\eps^2}\pabb{\lambda_\boldL - \eps C^\LL_2 - \eta_1 (C^\LL_2 + \delta_\ms C^\LL_1 K_x)} \normb{\nabla_x \boldf^\perp}_{\spacexv}^2\hspace{2.5cm}
\\[5mm]   \qquad\qquad + \frac{4\delta_\ms C^\LL_1 K_x}{\eps \eta_1} \normb{\boldf^\perp}_{\spacexv}^2  + \frac{8 \delta_\ms C_\pi C_{\T^3} (C^\LL_2+2 C^\LL_1 K_x)}{\eta_1} \norm{\nabla_x \boldf}_{\hilbertxv}^2
\\[4mm]    + \frac{4 \delta_\ms^3 C_\pi C^\TT (C^\LL_2+2 C^\LL_1 K_x)}{\eta_1}.
\end{multline}
The non-linear term is easily handled thanks to properties \eqref{Property6}--\eqref{Property7}. Applying Young's inequality with a positive constant $\eta_1/\eps$ we successively get
\begin{equation}\label{Non-linear term x}
\begin{split}
\frac{2}{\eps}\scalprod{\nabla_x\boldQ(\boldf,\boldf),\nabla_x\boldf}_{\hilbertxv} = &\ \frac{2}{\eps}\scalprodb{\nabla_x\boldQ(\boldf,\boldf),\nabla_x\boldf^\perp}_{\hilbertxv}
\\[3mm]   \leq &\ \frac{2}{\eta_1} \mathcal{G}^1_x(\boldf,\boldf)^2 +\frac{2\eta_1}{\eps^2}\normb{\nabla_x \boldf^\perp}_{\spacexv}^2.
\end{split}
\end{equation}
Finally, the source term is dealt with using estimate $\eqref{Property11}$ on $x$-derivatives as before, to get
\begin{equation}\label{Source term x}
2\scalprod{\nabla_x\bepsS,\nabla_x \boldf}_{\hilbertxv} \leq \frac{2\delta_\ms^2 C_x}{\eta_4 \eta_5} + 2\eta_4 \norm{\nabla_x\boldf}_{\hilbertxv}^2+ \frac{2\eta_5}{\eps^2} \normb{\nabla_x\boldf^\perp}_{\spacexv}^2,
\end{equation}
where we have also used that $\norm{\pi_{\boldL}(\nabla_x\boldf)}_{\hilbertxv}^2\leq \norm{\nabla_x\boldf}_{\hilbertxv}^2$.


To conclude, summing equations \eqref{Linear term x}--\eqref{Source term x} and recalling that $\delta_\ms\leq 1$, we recover the estimate
\begin{multline*}
\frac{\dd}{\dd t}\norm{\nabla_x \boldf}_{\hilbertxv}^2 \leq -\frac{2}{\eps^2}\pabb{\lambda_\boldL - \eps C^\LL_2 - \eta_1 (1 + C^\LL_2 + C^\LL_1 K_x) - \eta_5} \normb{\nabla_x \boldf^\perp}_{\spacexv}^2
\\[3mm]    \hspace{1.5cm} + \frac{4\delta_\ms C^\LL_1 K_x}{\eps \eta_1} \normb{\boldf^\perp}_{\spacexv}^2 +  \frac{2}{\eta_1}\mathcal{G}^1_x(\boldf,\boldf)^2
\\[4mm]    \hspace{5cm}+ \pa{2\eta_4 + \frac{8\delta_\ms C_\pi C_{\T^3} (C^\LL_2 + 2 C^\LL_1 K_x)}{\eta_1}}\norm{\nabla_x \boldf}_{\hilbertxv}^2 
\\[4mm]    + \frac{2\delta_\ms^2 C_x}{\eta_4\eta_5} + \frac{4\delta_\ms C_\pi C^\TT (C^\LL_2 + 2 C^\LL_1 K_x)}{\eta_1}.
\end{multline*}
By choosing 
$$\eta_1 = \eta_5 = \frac{\lambda_\boldL}{2(2 + 2C^\LL_2 + C^\LL_1 K_x)},\qquad \eta_4=\delta_\ms,\qquad \eps\leq \eta_1,$$
we finally obtain
\begin{multline*}
\frac{\dd}{\dd t}\norm{\nabla_x \boldf}_{\hilbertxv}^2 \leq -\frac{\lambda_\boldL}{\eps^2} \normb{\nabla_x \boldf^\perp}_{\spacexv}^2
\\[4mm]     + \frac{\delta_\ms}{\eps} \frac{8 K_x (2 + 2C^\LL_2 + C^\LL_1 K_x)}{\lambda_\boldL} \normb{\boldf^\perp}_{\spacexv}^2 +  \frac{8 (2 + 2C^\LL_2 + C^\LL_1 K_x)}{\lambda_\boldL}\mathcal{G}^1_x(\boldf,\boldf)^2
\\[5mm]   \qquad\qquad\qquad + \delta_\ms \pa{2 + \frac{16 C_\pi C_{\T^3} (C^\LL_2 + 2K_x)(2 + 2C^\LL_2 + C^\LL_1 K_x)}{\lambda_\boldL}}\norm{\nabla_x \boldf}_{\hilbertxv}^2 
\\[4mm]    + \delta_\ms \pab{4 C_x + 8 C_\pi C^\TT (C^\LL_2 + 2 C^\LL_1 K_x)}\frac{2 + 2C^\LL_2 + C^\LL_1 K_x}{\lambda_\boldL},
\end{multline*}
which is exactly $\eqref{a priori x f}$, with
\begin{gather*}
C^{(2)} = 2 + \frac{16 C_\pi C_{\T^3} (C^\LL_2 + 2 C^\LL_1 K_x)(2 + 2C^\LL_2 + C^\LL_1 K_x)}{\lambda_\boldL},\\[3mm]
C^{(3)} = \frac{8 K_x (2 + 2C^\LL_2 + K_x)}{\lambda_\boldL},\\[3mm]
\tilde{C}_x = \pab{4 C_x + 8 C_\pi C^\TT (C^\LL_2 + 2 C^\LL_1 K_x)}\frac{2 + 2C^\LL_2 + C^\LL_1 K_x}{\lambda_\boldL}.
\end{gather*}

\subsection{Time evolution of $\norm{\nabla_v \boldf}_{\hilbertxv}^2$}  

The evolution equation for the $\hilbertxv$ norm of $\nabla_v \boldf$ writes
\begin{equation*}
\begin{split}
\frac{\dd}{\dd t}\norm{\nabla_v \boldf}_{\hilbertxv}^2   =  &\ \frac{2}{\eps^2}\scalprod{\nabla_v\bepsL(\boldf),\nabla_v\boldf}_{\hilbertxv} - \frac{2}{\eps}\scalprod{\nabla_v \pa{v\cdot\nabla_x\boldf},\nabla_v \boldf}_{\hilbertxv}
\\[3mm]    &  + \frac{2}{\eps}\scalprod{\nabla_v\boldQ(\boldf,\boldf),\nabla_v\boldf}_{\hilbertxv}+2\scalprod{\nabla_v\bepsS,\nabla_v \boldf}_{\hilbertxv}.
\end{split}
\end{equation*}
The first term can be rewritten using the operators $\bepsK$ and $\bepsnu$ and is then dealt with thanks to estimates \eqref{Property3}--\eqref{Property4}. We have, for any $\xi >0$,
\begin{equation*}
\begin{split}
\frac{2}{\eps^2}\scalprod{\nabla_v\bepsL(\boldf),\nabla_v\boldf}_{\hilbertxv} = &\  \frac{2}{\eps^2}\scalprod{\nabla_v\bepsK(\boldf),\nabla_v\boldf}_{\hilbertxv} - \frac{2}{\eps^2}\scalprod{\nabla_v\bepsnu(\boldf),\nabla_v\boldf}_{\hilbertxv}
\\[4mm]  \leq &\ \frac{2(C^\KK_1(\xi)+C^{\NuNu}_5)}{\eps^2}\norm{\boldf}_{\hilbertxv}^2 - \frac{2 C^{\NuNu}_3}{\eps^2}\norm{\nabla_v\boldf}_{\spacexv}^2 
\\[3mm]  &   \hspace{5.5cm} +\frac{2\xi C^\KK_2}{\eps^2}\norm{\nabla_v\boldf}_{\hilbertxv}^2.
\end{split}
\end{equation*}
Now, we use the identity 
$$\norm{\boldf}_{\hilbertxv}^2=\normb{\boldf^\perp}_{\hilbertxv}^2+\norm{\pi_\boldL(\boldf)}_{\hilbertxv}^2,$$
and the fact that the $\spacexv$ norm controls the $\hilbertxv$ norm to deduce the first upper bound
\begin{multline}\label{Linear term v}
\frac{2}{\eps^2}\scalprod{\nabla_v\bepsL(\boldf),\nabla_v\boldf}_{\hilbertxv}
\\[3mm]    \leq  \frac{2\pab{C^\KK_1(\xi)+C^{\NuNu}_5}}{\eps^2 }\normb{\boldf^\perp}_{\spacexv}^2 -\frac{2}{\eps^2}\pa{C^{\nu}_3 - \xi C^\KK_2}\norm{\nabla_v\boldf}_{\spacexv}^2
\\[4mm]     + \frac{4 C_{\T^3}\pab{C^\KK_1(\xi)+C^{\NuNu}_5}}{\eps^2}\norm{\nabla_x\boldf}_{\hilbertxv}^2   + \frac{2 \delta_\ms^2 C^\TT \pab{C^\KK_1(\xi)+C^{\NuNu}_5}}{\eps^2},
\end{multline}
where we also applied the Poincar\'e inequality $\eqref{Poincare inequality}$. Next, the transport term is easily estimated thanks to Young's inequality as
\begin{equation}\label{Transport term v}
\begin{split}
-\frac{2}{\eps}\scalprod{\nabla_v \pa{v\cdot\nabla_x\boldf},\nabla_v \boldf}_{\hilbertxv} = &  -\frac{2}{\eps}\scalprod{\nabla_x\boldf,\nabla_v \boldf}_{\hilbertxv}
\\[3mm] \leq &\  \frac{2}{\eta}\norm{\nabla_x \boldf}_{\hilbertxv}^2 + \frac{2\eta}{\eps^2}\norm{\nabla_v\boldf}_{\spacexv}^2,
\end{split}
\end{equation}
holding for any $\eta >0$.
The non-linear term is handled in a similar way, using Young's inequality with the same constant $\eta/\eps$ to recover
\begin{equation}\label{Non-linear term v}
\frac{2}{\eps}\scalprod{\nabla_v\boldQ(\boldf,\boldf),\nabla_v\boldf}_{\hilbertxv}\leq  \frac{2}{\eta}\mathcal{G}^1_{x,v}(\boldf,\boldf)^2 + \frac{2\eta}{\eps^2}\norm{\nabla_v\boldf}_{\spacexv}^2.
\end{equation}
Finally, the source term is dealt with using estimate $\eqref{Property9}$, which gives
\begin{equation}\label{Source term v}
2\scalprod{\nabla_v\bepsS,\nabla_v \boldf}_{\hilbertxv} \leq  \frac{2 \delta_\ms^2 C_v}{\eta_2} +\frac{2 \eta_2}{\eps^2}\norm{\nabla_v\boldf}_{\spacexv}^2.
\end{equation}

Therefore, summing \eqref{Linear term v}--\eqref{Source term v} with the choices $\disp\xi = \eta = \eta_2 = \frac{C^{\NuNu}_3}{6+2C^\KK_2}$, we can finally recover
\begin{multline}
\frac{\dd}{\dd t}\norm{\nabla_v \boldf}_{\hilbertxv}^2  \leq   \frac{2 \pab{C^\KK_1(\xi)+C^{\NuNu}_5}}{\eps^2}\normb{\boldf^\perp}_{\spacexv}^2   -  \frac{C^{\NuNu}_3}{\eps^2}\norm{\nabla_v\boldf}_{\spacexv}^2    
\\[4mm]   \qquad + \frac{1}{\eps^2}\pa{ 4 C_{\T^3}(C^\KK_1(\xi)+C^\NuNu_5) + \frac{\eps^2(12 + 4C^\KK_2)}{C^{\NuNu}_3}}\norm{\nabla_x\boldf}_{\hilbertxv}^2 + \frac{12 + 4C^\KK_2}{C^{\NuNu}_3}\mathcal{G}^1_{x,v}(\boldf,\boldf)^2
\\[4mm]    + \frac{\delta_\ms}{\eps^2} \pa{ 2\delta_\ms C^\TT(C^\KK_1(\xi) + C^\NuNu_5) + \frac{\delta_\ms\eps^2(12+4C^\KK_2)}{C^\NuNu_3}}
\end{multline}
which is estimate $\eqref{a priori v f}$ with the choices
\begin{gather*}
K_1 = 2(C^\KK_1(\xi) + C^\NuNu_5),  \\[3mm]
K_{\dd x} = 4 C_{\T^3} (C^\KK_1(\xi) + C^\NuNu_5) + \frac{4(3 + C^\KK_2)}{C^{\NuNu}_3},   \\[3mm]
\tilde{C}_v = 2 C^\TT (C^\KK_1(\xi) + C^\NuNu_5) + \frac{4(3+C^\KK_2)}{C^\NuNu_3},
\end{gather*}
where we have also used that both $\eps\leq 1$ and $\delta_\ms \leq 1$.

\subsection{Time evolution of $\scalprod{\nabla_x\boldf,\nabla_v\boldf}_{\hilbertxv}$} 

The equation describing the time evolution of the $\hilbertxv$ norm of the commutator is given by
\begin{equation*}
\begin{split}
\frac{\dd}{\dd t}\scalprod{\nabla_x \boldf,\nabla_v\boldf}_{\hilbertxv} =  &\  \frac{2}{\eps^2}\scalprod{\nabla_x \bepsL(\boldf),\nabla_v\boldf}_{\hilbertxv} -\frac{2}{\eps}\scalprod{\nabla_x (v\cdot\nabla_x\boldf),\nabla_v\boldf}_{\hilbertxv} 
\\[3mm]           &      + \frac{2}{\eps} \scalprod{\nabla_x \boldQ(\boldf,\boldf),\nabla_v\boldf}_{\hilbertxv} + \scalprod{\nabla_x \bepsS,\nabla_v\boldf}_{\hilbertxv}.
\end{split}
\end{equation*}
For the linear term we shall successively apply the Leibniz derivation rule and the decomposition $\nabla_x \boldf=\nabla_x\boldf^\perp+\pi_\boldL(\nabla_x\boldf)$ to obtain
\begin{equation*}
\begin{split}
\frac{2}{\eps^2} \scalprod{\nabla_x \bepsL(\boldf),\nabla_v\boldf}_{\hilbertxv} = & \  \frac{2}{\eps^2} \scalprodb{\bepsL(\nabla_x \boldf^\perp),\nabla_v\boldf}_{\hilbertxv} 
\\[3mm]    & + \frac{2}{\eps^2} \scalprod{(\bepsL-\boldL)(\pi_\boldL(\nabla_x \boldf)),\nabla_v\boldf}_{\hilbertxv}
\\[3mm]   & + \frac{2}{\eps^2} \scalprod{\boldQ(\nabla_x\bepsM,\boldf)+\boldQ(\boldf,\nabla_x\bepsM),\nabla_v\boldf}_{\hilbertxv},
\end{split}
\end{equation*}
where we have used that $\boldL(\pi_\boldL(\nabla_x\boldf))=0$. Now, the first term is handled thanks to estimate $\eqref{Property2}$ as
\begin{equation*}
\frac{2}{\eps^2} \scalprodb{\bepsL(\nabla_x \boldf^\perp),\nabla_v\boldf}_{\hilbertxv} \leq \frac{2 C^\LL_1}{\eps^2}\normb{\nabla_x\boldf^\perp}_{\spacexv} \norm{\nabla_v\boldf}_{\spacexv},
\end{equation*}
and can be bounded using Young's inequality with a positive constant $e/\eps$, which gives
\begin{equation*}
\begin{split}
\frac{2}{\eps^2} \scalprodb{\bepsL(\nabla_x \boldf^\perp),\nabla_v\boldf}_{\hilbertxv}\leq &\ \frac{2 C^\LL_1 e}{\eps^3}\normb{\nabla_x\boldf^\perp}_{\spacexv}^2 + \frac{2 C^\LL_1}{\eps e} \norm{\nabla_v\boldf}_{\spacexv}^2.
\end{split}
\end{equation*}
The second term is of order $\mathcal{O}(\eps\delta_\ms)$ and can be handled more easily. We again use estimate $\eqref{Property2}$ and Young's inequality with a positive constant $\eta$, together with the equivalence of the $\hilbertxv$ and $\spacexv$ norms on $\Ker\boldL$, to get
\begin{equation*}
\begin{split}
\frac{2}{\eps^2} \langle(\bepsL-\boldL)(\pi_\boldL  & (\nabla_x \boldf)),\nabla_v\boldf\rangle_{\hilbertxv}
\\[3mm]   & \leq \frac{2 \eta\delta_\ms C^\LL_1 C_\pi}{\eps}\norm{\pi_\boldL(\nabla_x\boldf)}_{\hilbertxv}^2 + \frac{2 C^\LL_1}{\eta\eps}\norm{\nabla_v\boldf}_{\spacexv}^2
\\[3mm]   & \leq  \frac{2 \eta\delta_\ms C^\LL_1 C_\pi}{\eps}\norm{\nabla_x\boldf}_{\hilbertxv}^2 + \frac{2 C^\LL_1}{\eta\eps}\norm{\nabla_v\boldf}_{\spacexv}^2.
\end{split}
\end{equation*}
At last, we have already seen how to treat the third term. Skipping the details, we obtain
\begin{multline*}
\frac{2}{\eps^2} \scalprod{\boldQ(\nabla_x\bepsM,\boldf)+\boldQ(\boldf,\nabla_x\bepsM),\nabla_v\boldf}_{\hilbertxv}
\\[3mm]    \leq \frac{2 \delta_\ms C^\LL_1 K_x}{\eta\eps}\norm{\nabla_v\boldf}_{\spacexv}^2  +\frac{4\eta \delta_\ms C^\LL_1 K_x}{\eps}\normb{\boldf^\perp}_{\spacexv}^2 
\\[4mm]   + \frac{8\eta\delta_\ms C^\LL_1 K_x C_\pi C_{\T^3}}{\eps} \norm{\nabla_x\boldf}_{\hilbertxv}^2 + \frac{4\eta\delta_\ms^3 C^\LL_1 K_x C_\pi C^\TT}{\eps}.
\end{multline*}
using again Young's inequality with $\eta >0$, and the Poincar\'e inequality $\eqref{Poincare inequality}$. Collecting these upper bounds, we finally derive the estimate for the linear term
\begin{multline}\label{Linear term xv}
\frac{2}{\eps^2}\scalprod{\nabla_x \bepsL(\boldf),\nabla_v\boldf}_{\hilbertxv}
\\[3mm]    \leq  \frac{2 C^\LL_1 e}{\eps^3}\normb{\nabla_x\boldf^\perp}_{\spacexv}^2  + \pa{\frac{2+2 \delta_\ms C^\LL_1 K_x}{\eta\eps} +\frac{2C^\LL_1}{\eps e}} \norm{\nabla_v\boldf}_{\spacexv}^2
\\[4mm]    \hspace{1.5cm} + \frac{4\eta\delta_\ms C^\LL_1 K_x}{\eps}\normb{\boldf^\perp}_{\spacexv}^2  + \frac{2 \eta\delta_\ms (1 + 4 C^\LL_1 K_x C_\pi C_{\T^3})}{\eps}\norm{\nabla_x\boldf}_{\hilbertxv}^2
\\[3mm]    + \frac{4\eta\delta_\ms^3 C^\LL_1 K_x C_\pi C^\TT}{\eps}.
\end{multline}
For the transport term we successively integrate by parts, first in $x$ and then in~$v$. Direct calculations allows to recover
\begin{equation}\label{Transport term xv}
\begin{split}
-\frac{2}{\eps}\scalprod{\nabla_x (v\cdot\nabla_x\boldf),\nabla_v\boldf}_{\hilbertxv} = &  -\frac{1}{\eps}\norm{\nabla_x\boldf}_{\hilbertxv}^2 - \frac{1}{\eps}\norm{v\cdot\nabla_x\boldf}_{\hilbertxv}^2
\\[3mm]  \leq & -\frac{1}{\eps}\norm{\nabla_x\boldf}_{\hilbertxv}^2.
\end{split}
\end{equation}
The non-linear term is treated thanks to Young's inequality with the usual positive constant~$\eta $, and we get
\begin{equation}\label{Non-linear term xv}
\begin{split}
\frac{2}{\eps} \scalprod{\nabla_x \boldQ(\boldf,\boldf),\nabla_v\boldf}_{\hilbertxv}\leq &\  \frac{2\eta}{\eps}\mathcal{G}^1_{x,v}(\boldf,\boldf)^2 + \frac{2}{\eta\eps}\norm{\nabla_v\boldf}_{\spacexv}^2.
\end{split}
\end{equation}
The source term is then treated using estimate $\eqref{Property10}$, which gives
\begin{equation}\label{Source term xv}
2\scalprod{\nabla_x \bepsS,\nabla_v \boldf}_{\hilbertxv} \leq  \frac{2 \delta_\ms^2 C_{x,v}}{\eps \eta_3} +\frac{2\eta_3}{\eps}\norm{\nabla_v \boldf}_{\spacexv}^2.  
\end{equation}

Gathering inequalities \eqref{Linear term xv}--\eqref{Source term xv} with the choices $\eta = e$ and $\eta_3 = 1/e$, we then recover the following estimate
\begin{multline}
\frac{\dd}{\dd t}\scalprod{\nabla_x \boldf,\nabla_v\boldf}_{\hilbertxv} 
\\[3mm]    \leq  -\frac{1}{\eps}\pabb{1-2 e \delta_\ms (1 + 4 C^\LL_1 K_x C_\pi C_{\T^3})}\norm{\nabla_x\boldf}_{\hilbertxv}^2 +  \frac{2 e}{\eps}\mathcal{G}^1_{x,v}(\boldf,\boldf)^2\hspace{2cm}
\\[4mm]   \qquad \qquad + \frac{4 e C^\LL_1 K_x}{\eps}\normb{\boldf^\perp}_{\spacexv}^2 + \frac{5+2 C^\LL_1(1 + K_x)}{\eps e}\norm{\nabla_v\boldf}_{\spacexv}^2
\\[5mm]    +\frac{2 C^\LL_1 e}{\eps^3}\normb{\nabla_x\boldf^\perp}_{\spacexv}^2 + \frac{2 e \delta_\ms^2 (2 C^\LL_1 K_x C_\pi C^\TT + C_{x,v})}{\eps},
\end{multline}
where we have also used that $\delta_\ms\leq 1$. The choices
\begin{gather*}
C^{(4)} =  5+2 C^\LL_1(1 + K_x), \qquad C^{(5)} = 2 (1 + 4 C^\LL_1 K_x C_\pi C_{\T^3}), \\[3mm]
C^{(6)} = 4 C^\LL_1 K_x, \qquad \tilde{C}_{x,v} = 2 (2 C^\LL_1 K_x C_\pi C^\TT + C_{x,v})
\end{gather*}
finally lead to estimate $\eqref{a priori xv f}$.

\subsection{Time evolution of $\norm{\partial^\alpha_x \boldf}_{\hilbertxv}^2$} 

Consider now $\alpha\in\N^3$ such that $|\alpha|\leq s$. Using the Leibniz derivation rule for multi-indices, a direct iteration of the computations that we have made for $\norm{\nabla_x\boldf}_{\hilbertxv}^2$ immediately gives estimate $\eqref{a priori a f}$, for some constants $C^{(7)}$, $K_\alpha$, $C^{(8)}$ and $\tilde{C}_\alpha$ which only depend on $\lambda_\boldL$, $C^\LL_2$, $C_\pi$, $C_{\T^3}$, $C^\TT$ and $C_\alpha$.

\subsection{Time evolution of $\multideriv\boldf$} 

Let $\alpha$, $\beta\in\N^3$ be two multi-indices such that $|\alpha|+|\beta|\leq s$. The estimate for the $\hilbertxv$ norm of $\multideriv \boldf$ is obtained in a very similar way to the the one derived for the $\hilbertxv$ norm of~$\nabla_v\boldf$, therefore we shall skip some passages. Initially, the evolution equation reads
\begin{equation*}
\begin{split}
\frac{\dd}{\dd t}\Big\| \multideriv \boldf & \Big\|_{\hilbertxv}^2   
\\[2mm]   =  &\ \frac{2}{\eps^2}\scalprod{\multideriv\bepsL(\boldf),\multideriv\boldf}_{\hilbertxv} - \frac{2}{\eps}\scalprod{\multideriv\pa{v\cdot\nabla_x\boldf},\multideriv \boldf}_{\hilbertxv}
\\[3mm]    &  + \frac{2}{\eps}\scalprod{\multideriv\boldQ(\boldf,\boldf),\multideriv\boldf}_{\hilbertxv}+2\scalprod{\multideriv\bepsS,\multideriv \boldf}_{\hilbertxv}.
\end{split}
\end{equation*}
Thanks to estimates \eqref{Property3}-\eqref{Property4} on $\bepsK$ and $\bepsnu$, the linear term can be bounded as already seen in the case of $\nabla_v\boldf$. We get
\begin{multline}\label{Linear term ab}
\frac{2}{\eps^2}\scalprod{\multideriv\bepsL(\boldf),\multideriv\boldf}_{\hilbertxv}
\\[2mm]    \leq -\frac{2}{\eps^2}\pabb{C^{\NuNu}_3-\eps\mathds{1}_{\br{|\alpha|\geq1}}C^{\NuNu}_4 - \xi C^\KK_2}\norm{\multideriv \boldf}_{\spacexv}^2  \qquad\qquad
\\[3mm]      +\frac{2}{\eps^2}\pabb{C^{\NuNu}_5+\eps\mathds{1}_{\br{|\alpha|\geq1}}C^\NuNu_6 + C^\KK_1(\xi)}\norm{\boldf}_{\sobolevxv{s-1}}^2
\\[3mm]      + \mathds{1}_{\br{|\alpha|\geq1}}\frac{C^{\NuNu}_7}{\eps}\sum_{0<|\alpha^\prime|+|\beta^\prime|\leq s-1}\norm{\partial^{\beta^\prime}_v\partial^{\alpha^\prime}_x \boldf}_{\spacexv}^2.
\end{multline}
Next, direct calculations and the anti-symmetry of the transport term show that
\begin{equation*}
\scalprod{\multideriv\pa{v\cdot\nabla_x\boldf},\multideriv \boldf}_{\hilbertxv} = \sum_{k,\ \beta_k >0} \beta_k \scalprod{\partial^{\beta-\ee_k}_v\partial^{\alpha+\ee_k}_x\boldf,\multideriv \boldf}_{\hilbertxv}.
\end{equation*}
Since $k\leq 3$ and $\beta_k\leq s$, using Young's inequality with a positive constant $\eta/\eps >0$ we can recover the estimate
\begin{multline}\label{Transport term ab}
- \frac{2}{\eps}\scalprod{\multideriv\pa{v\cdot\nabla_x\boldf},\multideriv \boldf}_{\hilbertxv}
\\[2mm]    \leq  \frac{6\eta}{\eps^2} \norm{\multideriv\boldf}_{\spacexv}^2  + \frac{2s}{\eta} \sum_{k, \beta_k >0} \norm{\partial^{\beta-\ee_k}_v\partial^{\alpha+\ee_k}_x\boldf}_{\hilbertxv}^2.
\end{multline}
We again apply Young's inequality with the same constant $\eta/\eps$ to control the non-linear term
\begin{equation}\label{Non-linear term ab}
\frac{2}{\eps}\scalprod{\multideriv\boldQ(\boldf,\boldf),\multideriv\boldf}_{\hilbertxv}\leq \frac{2\eta}{\eps^2}\norm{\multideriv\boldf}_{\spacexv}^2 + \frac{2}{\eta} \mathcal{G}^s_{x,v}(\boldf,\boldf)^2,
\end{equation}
while the source term is dealt with thanks to estimate $\eqref{Property9}$
\begin{equation}\label{Source term ab}
2\scalprod{\multideriv\bepsS,\multideriv \boldf}_{\hilbertxv} \leq  \frac{2\delta_\ms^2 C_{\alpha,\beta}}{\eta_2} +\frac{2\eta_2}{\eps^2}\norm{\multideriv \boldf}_{\spacexv}^2.
\end{equation}
We now collect estimates \eqref{Linear term ab}--\eqref{Source term ab} with the choice $\eta_2 = \eta$, to recover
\begin{multline*}
\frac{\dd}{\dd t}\norm{\multideriv \boldf}_{\hilbertxv}^2
\\[2mm]  \leq   -\frac{2}{\eps^2}\pabb{C^{\NuNu}_3-\eps\mathds{1}_{\br{|\alpha|\geq1}}C^{\NuNu}_4 - \xi C^\KK_2  - 5 \eta}\norm{\multideriv \boldf}_{\spacexv}^2   \hspace{2cm}
\\[3mm]     +\frac{2}{\eps^2}\pabb{C^{\NuNu}_5+\eps\mathds{1}_{\br{|\alpha|\geq1}}C^\NuNu_6 + C^\KK_1(\xi)}\norm{\boldf}_{\sobolevxv{s-1}}^2 \hspace{2.5cm}
\\[3mm]  + \mathds{1}_{\br{|\alpha|\geq1}}\frac{C^{\NuNu}_7}{\eps}\sum_{0<|\alpha^\prime|+|\beta^\prime|\leq s-1}\norm{\partial^{\beta^\prime}_v\partial^{\alpha^\prime}_x \boldf}_{\spacexv}^2 \hspace{0.5cm}
\\[3mm]    + \frac{2s}{\eta} \sum_{k,\ \beta_k >0} \norm{\partial^{\beta-\ee_k}_v\partial^{\alpha+\ee_k}_x\boldf}_{\hilbertxv}^2  + \frac{2}{\eta} \mathcal{G}^s_{x,v}(\boldf,\boldf)^2 + \frac{2 \delta_\ms^2 C_{\alpha,\beta}}{\eta}.
\end{multline*}
Recalling that both $\eps\leq 1$ and $\delta_\ms\leq 1$, and using the upper bound
\begin{equation*}
\mathds{1}_{\br{|\alpha|\geq1}}\sum_{0<|\alpha^\prime|+|\beta^\prime|\leq s-1}\norm{\partial^{\beta^\prime}_v\partial^{\alpha^\prime}_x \boldf}_{\spacexv}^2  \leq \norm{\boldf}_{H^{s-1}_{x,v}\big(\langle v\rangle^{\frac{\gamma}{2}}\boldmu^{-\frac{1}{2}}\big)}^2,
\end{equation*}
together with $\mathds{1}_{\br{|\alpha|\geq1}}\leq 1$, so that we finally obtain estimate $\eqref{a priori ab f}$ by choosing
\begin{gather*}
\xi = \eta = \frac{C^{\NuNu}_3}{2(5 + C^\KK_2 + C^\NuNu_4)},\qquad \eps\leq \eta, \\[3mm]
K_{s-1} = 2\pab{ C^\NuNu_5 + C^\NuNu_6 + C^\KK_1(\xi) }, \\[3mm]
C^{(9)} = \frac{4 s (5 + C^\KK_2 + C^\NuNu_4)}{C^\NuNu_3},  \\[3mm]
\tilde{C}_{\alpha,\beta} = \frac{4 C_{\alpha,\beta} (5 + C^\KK_2 + C^\NuNu_4)}{C^\NuNu_3}.
\end{gather*}

\subsection{Time evolution of $\scalprod{\partial^\alpha_x \boldf, \partial^{\ee_k}_v \partial^{\alpha-\ee_k}_x\boldf}_{\hilbertxv}$} 

At last, consider a multi-index $\alpha\in\N^3$, with $|\alpha|\leq s$ and $\alpha_k >0$. The equation satisfied by the $\hilbertxv$ norm of the commutator for higher derivatives is
\begin{equation*}
\begin{split}
\frac{\dd}{\dd t}\langle\partial^\alpha_x\boldf,\partial^{\ee_k}_v  & \partial^{\alpha-\ee_k}_x\boldf\rangle_{\hilbertxv} 
\\[2mm]     = &\  \frac{2}{\eps^2}\scalprod{\partial^\alpha_x\bepsL(\boldf),\partial^{\ee_k}_v \partial^{\alpha-\ee_k}_x \boldf}_{\hilbertxv}  - \frac{2}{\eps} \scalprod{\partial^\alpha_x(v\cdot\nabla_x\boldf),\partial^{\ee_k}_v \partial^{\alpha-\ee_k}_x\boldf}_{\hilbertxv}
\\[3mm]  &\  + \frac{2}{\eps}\scalprod{\partial^\alpha_x\boldQ(\boldf,\boldf),\partial^{\ee_k}_v \partial^{\alpha-\ee_k}_x\boldf}_{\hilbertxv}  + 2\scalprod{\partial^\alpha_x\bepsS,\partial^{\ee_k}_v \partial^{\alpha-\ee_k}_x\boldf}_{\hilbertxv}.
\end{split}
\end{equation*}
Thanks to Leibniz derivation rule, the linear operator can be split as
\begin{equation*}
\begin{split}
\partial^\alpha_x \bepsL(\boldf) =  & \  \partial^\alpha_x\boldQ(\bepsM,\boldf) + \partial^\alpha_x\boldQ(\boldf,\bepsM)
\\[4mm]    =  &\  \bepsL(\partial^\alpha_x\boldf) + \sum_{\substack{\gamma \leq \alpha  \\   |\gamma| < |\alpha|}}\binom{\alpha}{\gamma} \pabb{\boldQ(\partial^\gamma_x \bepsM,\partial^{\alpha-\gamma}_x\boldf)+\boldQ(\partial^{\gamma}_x\boldf,\partial^{\gamma-\alpha}_x \bepsM)}
\\[3mm]    = &\  \bepsL(\partial^\alpha_x\boldf^\perp) + (\bepsL-\boldL)(\pi_\boldL(\partial^\alpha_x\boldf)) 
\\[2mm]     &  \hspace{3cm} + \sum_{\substack{\gamma \leq \alpha  \\   |\gamma| < |\alpha|}}\binom{\alpha}{\gamma} \pabb{\boldQ(\partial^\gamma_x \bepsM,\partial^{\alpha-\gamma}_x\boldf)+\boldQ(\partial^{\gamma}_x\boldf,\partial^{\gamma-\alpha}_x \bepsM)}
\end{split}
\end{equation*}
where the second and third terms are of order $\mathcal{O}(\eps\delta_\ms )$. Using estimate $\eqref{Property2}$ we can thus initially bound the linear term as
\begin{multline*}
\frac{2}{\eps^2}\scalprod{\partial^\alpha_x\bepsL(\boldf),\partial^{\ee_k}_v \partial^{\alpha-\ee_k}_x \boldf}_{\hilbertxv} 
\\[3mm]    \leq  \frac{2 C^\LL_1}{\eps^2} \norm{\partial^\alpha_x\boldf^\perp}_{\spacexv}\norm{\partial^{\ee_k}_v \partial^{\alpha-\ee_k}_x\boldf}_{\spacexv} \hspace{3.5cm}
\\[3mm]     + \frac{2 \delta_\ms C^\LL_1 K_{\dd x}' }{\eps} \norm{\pi_\boldL(\partial^\alpha_x\boldf)}_{\spacexv}\norm{\partial^{\ee_k}_v \partial^{\alpha-\ee_k}_x\boldf}_{\spacexv}
\\[3mm]     + \frac{2 \delta_\ms C^\LL_1 \tilde{K}_{\dd x} }{\eps} \sum_{|\alpha'| < |\alpha|} \norm{\partial^{\alpha'}_x\boldf}_{\spacexv}\norm{\partial^{\ee_k}_v \partial^{\alpha-\ee_k}_x\boldf}_{\spacexv},
\end{multline*}
for some positive (explicitly computable) constants $K_{\dd x}'$ and $\tilde{K}_{\dd x}$. We then apply Young's inequality to the first term with a positive constant $e/\eps$ and to the second and third ones with the constant $e >0$, to recover 
\begin{multline}\label{Linear term a}
\frac{2}{\eps^2}\scalprod{\partial^\alpha_x\bepsL(\boldf),\partial^{\ee_k}_v \partial^{\alpha-\ee_k}_x \boldf}_{\hilbertxv} 
\\[2mm]    \leq   \frac{2 C^\LL_1 e}{\eps^3} \norm{\partial^\alpha_x\boldf^\perp}_{\spacexv}^2   + \frac{2 \delta_\ms e C^\LL_1 K_{\dd x}' C_\pi}{\eps} \norm{\partial^\alpha_x\boldf}_{\hilbertxv}^2 \hspace{2cm}
\\[4mm]   \hspace{2cm} + \frac{2 C^\LL_1}{\eps e}\pabb{1+\delta_\ms \pab{K_{\dd x}' C_\pi + 2\tilde{K}_{\dd x}(1 + C_\pi }}\norm{\partial^{\ee_k}_v \partial^{\alpha-\ee_k}_x\boldf}_{\spacexv}^2    
\\[3mm]     \hspace{2cm} + \frac{2\delta_\ms e  C^\LL_1 \tilde{K}_{\dd x}}{\eps} \sum_{|\alpha'| < |\alpha|} \norm{\partial^{\alpha'}_x\boldf}_{\spacexv}^2.    
\end{multline}
The transport term is more tricky in this case. Indeed, integrating by parts in $x_k$ and then in $v_k$, we get the identity
\begin{equation}\label{Transport term a}
\begin{split}
- \frac{2}{\eps} \langle\partial^\alpha_x(v\cdot\nabla_x\boldf), & \partial^{\ee_k}_v \partial^{\alpha-\ee_k}_x\boldf\rangle_{\hilbertxv} \\[2mm]     = & - \frac{1}{\eps} \norm{\partial^\alpha_x \boldf}_{\hilbertxv}^2 - \frac{1}{\eps} \scalprod{\partial^{\alpha - \ee_k}_x(v\cdot\nabla_x \boldf), v_k\partial^\alpha_x \boldf}_{\hilbertxv},
\end{split}
\end{equation}
where the second term comes from the integration by parts when one derives the maxwellian weight $\boldmu^{-1}$ with respect to the variable $v_k$. Note that this term does not have an explicit sign and could therefore create a problem when trying to close the estimates. However, it is important to recall that we are not interested in the estimate of each single $|\alpha| = s$ with $\alpha_k >0$, but we only care about controlling the sum of all these terms. In particular, one can prove that when summing over $|\alpha|\leq s$ and $\alpha_k >0$ with $k=1,2,3,$ we get
\begin{equation*}
\begin{split}
-\sum_{\substack{|\alpha| = s  \\  k,\ \alpha_k >0}}\scalprod{\partial^{\alpha - \ee_k}_x(v\cdot\nabla_x \boldf), v_k\partial^\alpha_x \boldf}_{\hilbertxv} = & - \sum_{|\alpha'| = s-1} \norm{v\cdot\nabla_x (\partial^{\alpha'}_x\boldf)}_{\hilbertxv}^2 \leq 0,
\end{split}
\end{equation*}
so that the problematic term actually exhibits an explicit sign and we can get rid of it.

\smallskip
Finally, thanks to Young's inequality applied with the positive constant $e/C^\LL_1$, we can control the non-linear term as
\begin{equation}\label{Non-linear term a}
\frac{2}{\eps}\scalprod{\partial^\alpha_x\boldQ(\boldf,\boldf),\partial^{\ee_k}_v \partial^{\alpha-\ee_k}_x\boldf}_{\hilbertxv} \leq  \frac{2 e}{\eps C^\LL_1}\mathcal{G}^s_{x}(\boldf,\boldf)^2  + \frac{2 C^\LL_1}{\eps e}\norm{\partial^{\ee_k}_v \partial^{\alpha-\ee_k}_x\boldf}_{\spacexv},
\end{equation}
while the source term is handled thanks to estimate $\eqref{Property10}$, and writes
\begin{equation}\label{Source term a}
2\scalprod{\partial^\alpha_x \bepsS,\partial^{e_k}_v\partial^{\alpha-e_k}_x \boldf}_{\hilbertxv} \leq  \frac{2 \delta_\ms^2 C_{\alpha, k}}{\eps \eta_3} +\frac{2 \eta_3}{\eps}\norm{\partial^{e_k}_v\partial^{\alpha-e_k}_x \boldf}_{\spacexv}^2.
\end{equation}
Gathering estimates \eqref{Linear term a}--\eqref{Source term a}, together with the use of the upper bound
\begin{equation*}
\sum_{|\alpha'| < |\alpha|} \norm{\partial^{\alpha'}_x\boldf}_{\spacexv}^2 \leq \norm{\boldf}_{H^{s-1}_{x,v}\big(\langle v\rangle^{\frac{\gamma}{2}}\boldmu^{-\frac{1}{2}}\big)}^2,
\end{equation*}
choosing $\eta_3 = C^\LL_1/ e$ and recalling that $\delta_\ms\leq 1$, we finally obtain
\begin{multline}
\frac{\dd}{\dd t}\scalprod{\partial^\alpha_x\boldf,\partial^{\ee_k}_v \partial^{\alpha-\ee_k}_x\boldf}_{\hilbertxv}
\\[2mm]    \leq - \frac{1}{\eps}\pabb{1 - 2 \delta_\ms e C^\LL_1 K_{\dd x}' C_\pi} \norm{\partial^\alpha_x \boldf}_{\hilbertxv}^2   + \frac{2 C^\LL_1 e}{\eps^3} \norm{\partial^\alpha_x\boldf^\perp}_{\spacexv}^2 \hspace{1cm}
\\[4mm]      + \frac{2 C^\LL_1}{\eps e}\pabb{3+K_{\dd x}'  C_\pi + 2 \tilde{K}_{\dd x}(1+C_\pi) }\norm{\partial^{\ee_k}_v \partial^{\alpha-\ee_k}_x\boldf}_{\spacexv}^2
\\[5mm]    \hspace{3.5cm} + \frac{2 \delta_\ms  e C^\LL_1 \tilde{K}_{\dd x} }{\eps} \norm{\boldf}_{H^{s-1}_{x,v}\big(\langle v\rangle^{\frac{\gamma}{2}}\boldmu^{-\frac{1}{2}}\big)}^2 + \frac{2 e}{\eps C^\LL_1}\mathcal{G}^s_{x}(\boldf,\boldf)^2 + \frac{2 \delta_\ms^2 e C_{\alpha,k}}{\eps C^\LL_1} 
\\[5mm]     - \frac{1}{\eps} \scalprod{\partial^{\alpha - \ee_k}_x(v\cdot\nabla_x \boldf), v_k\partial^\alpha_x \boldf}_{\hilbertxv},
\end{multline}
which is estimate $\eqref{a priori ak f}$ with the choices
\begin{gather*}
C^{(10)} = 3+K_{\dd x}'  C_\pi + 2 \tilde{K}_{\dd x}(1+C_\pi), \qquad C^{(11)} = 2  C^\LL_1 K_{\dd x}' C_\pi, \\[3mm]
K_{\alpha,k} = 2 C^\LL_1 \tilde{K}_{\dd x}, \qquad \tilde{C}_{\alpha, k} = \frac{2 C_{\alpha,k}}{C^\LL_1}.
\end{gather*}



\providecommand{\href}[2]{#2}
\providecommand{\arxiv}[1]{\href{http://arxiv.org/abs/#1}{arXiv:#1}}
\providecommand{\url}[1]{\texttt{#1}}
\providecommand{\urlprefix}{URL }

%
%
\bigskip
\signandrea

\bigskip
\signmarc

\end{document}